\documentclass[francais]{smfart}

\usepackage[frenchb]{babel}
\usepackage[T1]{fontenc}
\usepackage[utf8]{inputenc}
\usepackage{amssymb,graphicx,lmodern,xypic}
\usepackage{smfenum}
\theoremstyle{plain}
\newtheorem{theo}[subsubsection]{Th\'eor\`eme}

\newtheorem{prop}[subsubsection]{Proposition}
\newtheorem{lemm}[subsubsection]{Lemme}

\theoremstyle{definition}
\newtheorem{defi}[subsubsection]{D\'efinition}

\newtheorem{eg}[subsubsection]{Exemple}
\newtheorem{rem}[subsubsection]{Remarque}
\newtheorem{ex}[subsubsection]{Exercice}
\newtheorem{conj}[subsubsection]{Conjecture}

\def\resp{\emph{resp.}\xspace}
\def\ie{\emph{i.e.}\xspace}

\def\P{{\mathbf{P}}}
\def\Ber{\text{\upshape{Ber}}}
\def\BP{{\hat{{\mathbf P}}}^1}
\def\C{{\mathbf C}}
\def\R{{\mathbf R}}
\def\Q{{\mathbf Q}}
\def\Z{{\mathbf Z}}
\def\N{{\mathbf{N}}}

\def\AB{{\mathcal{M}} }

\def\Jul{\mathop{\mathscr{J}}}
\def\JulK{\mathop{\mathscr{K}}}
\def\Fat{\mathop{\mathscr{F}}}

\def\F{{\mathbf F}}
\def\Hp{\mathbf H_p}
\def\Hp{\mathbf H_p}
\def\HpR{\mathbf H_p^{\R}}
\def\HpQ{\mathbf H_p^{\Q}}
\def\Hpsing{\mathbf H_p^{\text{\upshape sing}}}

\def\Res{\operatorname{Res}}
\def\Mob{\operatorname{\text{\upshape M\"ob}}}
\def\PSL{\operatorname{PSL}}
\def\Log{\operatorname{log}}
\def\Mandel{\operatorname{M}}

\let\bar\overline
\let \mathscr\mathcal

\def\can{{\text{\upshape can}}}
\def\sing{{\text{\upshape sing}}}
\def\abs#1{\left|{#1}\right|}
\def\norm#1{\left\|{#1}\right\|}

\let\ra\rightarrow
\let\phi\varphi

\def\whE{\mathop{\mathscr E}}

\def\PGL{\operatorname{PGL}}
\def\SL{\operatorname{SL}}

\def\Crit{\operatorname{Crit}}

\def\dist{\operatorname{dist}}

\def\deg{\operatorname{deg}}
\def\diam{\operatorname{diam}}

\def\Bo{B^\circ}
\def\Bf{B^\bullet}
\addtocounter{section}{0}             % Start with section 1
\numberwithin{equation}{subsection}       % Number formulas within sections

\title[Dynamique $p$-adique]{Dynamique~$p$-adique\\
(d'apr\`es les expos\'es de Jean-Christophe Yoccoz)}

\author{Serge Cantat, Antoine Chambert-Loir}
\address{Universit\'e de Rennes~1, Campus de Beaulieu, F-35042 Rennes Cedex}
\email{(serge.cantat,antoine.chambert-loir)@univ-rennes1.fr}

\begin{abstract}
Ce texte est une introduction \`a la dynamique
des fractions rationnelles sur un corps~$p$-adique.
Il reprend le cours que Jean-Christophe Yoccoz avait fait
lors des \emph{\'Etats de la recherche} de mai 2006 et
en d\'eveloppe certains points. La plupart des r\'esultats pr\'esent\'es
sont d\^us \`a Benedetto, B\'ezivin, Hsia, Lubin, Morton, Rivera-Letelier
et Silverman. 
\end{abstract}

\begin{altabstract}
This paper is an introduction to the dynamics of rational fractions over a $p$-adic field. It follows the lectures that Jean-Christophe Yoccoz
gave during the \emph{\'Etats de la recherche} in May 2006,
and develops certain points. Most of the results to be explained
here are due to Benedetto, B\'ezivin, Hsia, Lubin, Morton, Rivera-Letelier
and Silverman. 
\end{altabstract}

\begin{document}

\maketitle

%
%%%%%%%%%%%%%%%%%%%%%%%%%%%%%%%%%%%%%%%%%%%%%%%%%%%%%%%%%%%%%%%%%%
%

 \setcounter{tocdepth}{1}
 \tableofcontents

%%%%%%%%%%%%%%%%%%%%%%%%%%%%%%%%%%%%%%%
%%%%%%%%%%%%%%%%%%%%%%%%%%%%%%%%%%%%%%%

\section{Introduction et motivations}

%%%%%%%%%%%%%%%%%%%%%%%%%%%%%%%%%%%%%%%
%%%%%%%%%%%%%%%%%%%%%%%%%%%%%%%%%%%%%%%

%\medskip

Afin de motiver les sp\'ecialistes de dynamique
holomorphe, formulons quelques probl\`emes ouverts de dynamique
complexe qui peuvent \^etre reli\'es \`a des questions de dynamique
$p$-adique (voir~\cite{McMullen:survey} pour plus de d\'etails).

\subsection{Hyperbolicit\'e}

Soit $f\in \C(z)$ une fraction rationnelle d'une variable complexe. Nous
identifierons~$f$ \`a l'endomorphisme holomorphe de la droite projective
$\P^1(\C)$ qu'elle induit, et  nous noterons~$d$ son degr\'e. Nous supposerons 
toujours que~$d$ est sup\'erieur ou \'egal \`a~$2$. 

Les points critiques de~$f$ sont les points~$z$ de la droite projective 
en lesquels la diff\'erentielle de $f:\P^1(\C)\to \P^1(\C)$ s'annule;
notons $\Crit(f)$ l'ensemble de ces points.
L'ensemble postcritique de~$f$ est la r\'eunion des orbites positives
des points critiques de~$f$ :
\[
P(f)= \bigcup_{\substack{n>0 \\ y\in\Crit(f)}} f^n(y).
\]

Les points p\'eriodiques
de~$f$ peuvent \^etre organis\'es en trois  cat\'egories de la
mani\`ere suivante. Soit~$z$  
un point p\'eriodique de p\'eriode~$p$.
La diff\'erentielle 
de $f^p$ au point~$z$ est une application lin\'eaire de $T_z(\P^1(\C))$
dans lui m\^eme ; c'est donc une homoth\'etie dont le rapport est not\'e
$(f^p)'(z)$. Lorsque~$z$ est diff\'erent de $\infty$, $(f^p)'(z)$ est
la d\'eriv\'ee usuelle de la fraction rationnelle $f^p$, \'evalu\'ee en~$z$.\label{page.derivee}
On dit que
\begin{itemize}
\item  \emph{$z$ est r\'epulsif} si $\abs{(f^p)'(z) } > 1$,
\item  \emph{$z$ est indiff\'erent} si $\abs{(f^p)'(z) } = 1$,
\item  \emph{$z$ est attractif} si $\abs{(f^p)'(z) } < 1$,
et~\emph{super-attractif} si $(f^p)'(z)=0$.
\end{itemize}
Lorsque~$z$ est un point p\'eriodique attractif de p\'eriode~$p$, l'orbite
de~$z$ est un cycle attractif : 
il existe un ouvert de $\P^1(\C)$ contenant l'orbite de
$z$ tel que l'orbite de tout point dans cet ouvert soit attir\'ee par l'orbite de
$z$ (voir~\cite{Milnor:book} \S 8 et~\S 9). Par ailleurs, comme tout cycle attractif 
attire au moins un point critique de~$f$, l'union des cycles
attractifs est contenue dans l'adh\'erence de l'ensemble postcritique
(\cite{Milnor:book}, th\'eor\`eme~8.6).

L'ensemble de Fatou de~$f$ est le plus grand ouvert de $\P^1(\C)$ sur lequel la
famille $(f^n)$ des it\'er\'es de~$f$ (consid\'er\'es comme applications holomorphes
de~$\P^1(\C)$ dans lui-m\^eme) forme, localement, une famille normale. 
Il contient les cycles attractifs. 
Son compl\'ementaire est l'ensemble de Julia~$\Jul(f)$ ; il co\"{\i}ncide avec 
l'adh\'erence des points p\'eriodiques r\'epulsifs de~$f$ (\cite{Milnor:book},
\S 14).

On dit que~$f$ est \emph{hyperbolique} lorsque les trois conditions
\'equivalentes suivantes sont satisfaites (voir~\cite{McMullen:survey},
th\'eor\`eme~2.2; voir aussi~\cite{Flexor}) :
\begin{enumerate}\def\theenumi{\roman{enumi}}\def\labelenumi{(\theenumi)}
\item l'adh\'erence de l'ensemble postcritique de~$f$ est disjointe
de l'ensemble de Julia~$\Jul(f)$ ;
\item tous les points critiques de~$f$ sont attir\'es par des cycles attractifs ;
\item l'application~$f$ est dilatante au voisinage de~$\Jul(f)$, ce qui signifie
qu'il existe une m\'etrique conforme~$\rho$ d\'efinie sur un voisinage de~$\Jul(f)$  telle
que $\abs{f'(z)}_\rho >1$ pour tout point~$z$ de~$\Jul(f)$.
\end{enumerate}
L'application~$f$ est donc hyperbolique lorsque l'orbite des points critiques 
(au voisinage desquels $f$ cesse \'evidemment d'\^etre dilatante)
ne se m\^ele pas \`a l'ensemble de Julia 
(l\`a o\`u sont localis\'es les points r\'epulsifs).  
La condition~(ii) montre que l'hyperbolicit\'e est une condition 
ouverte ; ainsi, l'ensemble
des fractions rationnelles de degr\'e~$d$ qui sont hyperboliques forme
un ouvert dans l'espace des fractions rationnelles de degr\'e $d$.
La {\emph{conjecture
d'hyperbolicit\'e}} stipule que cet ouvert est dense : 

\begin{conj}[voir~\cite{McMullen:survey}]\label{conj:hyperbolicite}
Soit~$d$ un entier sup\'erieur ou \'egal \`a~$2$. L'ensemble des applications
rationnelles qui sont hyperboliques forme un ouvert \emph{dense} dans
l'ensemble des applications rationnelles de degr\'e~$d$.
\end{conj}

Dans le cas des {\emph{polyn\^omes quadratiques}}
\begin{equation}\label{eq:quadratique}
f_c(z) = z^2+c
\end{equation}
une variante de la conjecture affirme que l'ensemble des nombres complexes~$c$ tels que
le point critique~$0$ est attir\'e par un cycle attractif de~$f_c$
forme un ouvert dense
de la droite complexe~$\C$. 
Puisque tout cycle attractif attire l'unique point critique~$0$,
il revient au m\^eme de dire que, pour un ouvert
dense de param\`etres~$c$, ou bien $f_c$ poss\`ede un cycle attractif, ou bien 
$f_c^n(0)$ tend vers l'infini. M\^eme dans ce cas, la conjecture est 
encore ouverte \`a l'heure actuelle.

\subsection{Renormalisation}

Soit $f_c$ un polyn\^ome quadratique pour lequel
l'ensemble postcritique est born\'e ; 
autrement dit, $c$ appartient par d\'efinition 
\`a l'ensemble de Mandelbrot  
\[
\Mandel =\{c \in \C \, ; \, \text{la suite $(f_c^n(0))$ est born\'ee}\} 
\]
(voir~\cite{Milnor:book}, appendix G). 
On dit que l'it\'er\'e $f_c^n$ de $f_c$ est \emph{renormalisable}
(ou que $f_c$ est renormalisable au temps~$n$) 
s'il existe deux ouverts~$U$ et~$V$ simplement connexes
contenant l'origine~$0$,  
tels que 
\begin{enumerate}\def\theenumi{\alph{enumi}}\def\labelenumi{(\theenumi)}
\item  $f_c^n(U)=V$ et $\overline{U} \subset V$ ;
\item l'application $f_c^n:U\to V$ est propre et de degr\'e~$2$ ;
\item pour tout entier $k\geq 0$, $f_c^{nk}(0)$ appartient \`a~$U$.
\end{enumerate}
Autrement dit,   l'it\'er\'e $f_c^n$ de $f_c$ (qui est un polyn\^ome de degr\'e~$2^n$)
se comporte au voisinage du point critique~$0$ comme un polyn\^ome
de degr\'e~$2$ dont l'orbite postcritique est born\'ee. On d\'emontre alors
que $f_c^n:U \to V$ est topologiquement conjugu\'e, sur~$U$, \`a un polyn\^ome
quadratique $f_{c'}$ (voir~\cite{Douady-Hubbard:1985}). La condition~(c) assure
que l'orbite du point critique de $f_{c'}$ est \`a nouveau born\'ee,
donc que $c'\in\Mandel$.

On dit que le polyn\^ome quadratique $f_c$ est \emph{infiniment renormalisable} s'il l'est au temps~$n$ pour un ensemble infini d'entiers~$n$. Pour montrer la conjecture
d'hyperbolicit\'e, il suffit de montrer que l'ensemble des param\`etres $c$ pour lesquels
$f_c$ est infiniment renormalisable est d'int\'erieur vide. Ceci est expliqu\'e, par exemple, dans l'article de synth\`ese de McMullen 
\cite{McMullen:survey} (voir aussi \cite{Flexor}).
Les param\`etres pour lesquels
$f_c$ est renormalisable forment des petites copies de l'ensemble de Mandelbrot
\`a l'int\'erieur de lui-m\^eme.  
La conjecture d'hyperbolicit\'e pour la famille
$(f_c)$ s'av\`ere donc \'equivalente
au fait que toute intersection d\'ecroissante 
de telles copies de $\Mandel$ soit d'int\'erieur
vide. Le th\'eor\`eme suivant concerne
l'intersection de cet ensemble avec l'axe r\'eel.

\begin{theo}[Lyubich]
L'ensemble des valeurs \emph{r\'eelles} du param\`etre~$c$
tel que l'application $z\mapsto z^2+c$ soit \emph{infiniment
renormalisable} est de mesure de Lebesgue nulle.
\end{theo}

Si la conjecture d'hyperbolicit\'e n'est pas satisfaite, 
on peut trouver un nombre alg\'ebrique~$c$ 
pour lequel $f_c$ est infiniment renormalisable. 
La conjecture suivante est donc plus forte que la conjecture d'hyperbolicit\'e
pour la famille $f_c(z)=z^2+ c$, $c\in \C$. 

\begin{conj}\label{conj:renormalgebrique}
Si~$c$ est un nombre alg\'ebrique, l'application \mbox{$z\mapsto z^2+c$}
n'est pas infiniment renormalisable.
\end{conj}

\subsection{Dynamique~$p$-adique}

En vue de comprendre la dynamique
de l'application $z\mapsto z^2+c$ lorsque
$c$ est un nombre alg\'ebrique,
il est naturel de chercher \`a d\'ecrire la restriction
de cette dynamique \`a l'ensemble des nombres complexes~$z$ qui 
sont des nombres alg\'ebriques. Nous avons alors
\`a notre disposition diverses topologies 
sur l'ensemble~$\bar\Q$ des nombres alg\'ebriques.
Le premier niveau de cette \'etude, celui qui est abord\'e ici,
en concerne les aspects $p$-adiques, ou $p$ est un nombre premier,
et plus pr\'ecis\'ement
ceux li\'es \`a la dynamique de cette application sur l'ensemble
des {\og nombres complexes $p$-adiques\fg}. 
Le deuxi\`eme niveau prendrait en compte l'action du  groupe de Galois de~$\bar\Q_p$ sur~$\Q_p$\footnote{Pour tout nombre premier $p$, le groupe
des automorphismes de~$\C_p$ qui sont continus pour la topologie $p$-adique
co\"{\i}ncide avec le groupe de Galois de ~$\bar\Q_p$ sur~$\Q_p$.}
et le troisi\`eme niveau
se devrait de mettre ensemble les diff\'erentes topologies de~$\Q$,
c'est-\`a-dire les informations $p$-adiques, lorsque $p$ parcourt
l'ensemble des nombres premiers, 
et les informations complexes.

Cet article pr\'esente ainsi
quelques r\'esultats fondamentaux de dynamique~$p$-adique;
il ne contient quasiment aucune d\'emonstration compl\`ete.
Le paragraphe~\ref{par:berko} introduit la droite projective au sens de
Berkovich : il s'agit d'un espace compact connexe m\'etrisable
contenant la droite projective {\og na\"{\i}ve\fg} $\P^1(\C_p)$ comme partie dense
et totalement discontinue.
Le suivant d\'ecrit l'action des fractions rationnelles 
sur cet espace. Le paragraphe~\ref{par:fatoujulia} amorce l'\'etude
de la dynamique des transformations rationnelles. Celle-ci est
pr\'ecis\'ee au paragraphe~\ref{par:poly} dans le cas des polyn\^omes.

Signalons aussi que ces notes ne pr\'esentent qu'un \'etat partiel
de la question au moment de la conf\'erence. Nous esp\'erons qu'elles
fourniront aide et motivation pour se plonger dans les articles originaux.

Quelques compl\'ements
concernant les ph\'enom\`enes d'\'equidistribution sont \'evoqu\'es dans l'article
de Chambert-Loir de ce volume. Le lecteur int\'eress\'e par les aspects ergodiques
pourra \'egalement consulter l'article~\cite{Favre-Rivera-Letelier:2008}
de Favre et Rivera-Letelier ; pour des r\'esultats effectifs en degr\'e $2$ 
sur une extension finie de $\Q_p$, on consultera
\cite{Benedetto-Briend-Perdrix:2007}.
Signalons aussi l'ouvrage~\cite{baker-rumely2010} consacr\'e
\`a la th\'eorie du potentiel sur la droite projective au sens de Berkovich;
il contient notamment un long chapitre sur la dynamique des applications
rationnelles.

\medskip

\noindent{\bf{Remerciements. ---}} 
Nous remercions Jean-Yves Briend, Antoine Ducros et Charles Favre 
pour leurs commentaires et pour avoir mis \`a notre disposition 
leurs notes des expos\'es de Yoccoz.
Nous remercions aussi le rapporteur dont les
nombreuses remarques ont permis d'am\'eliorer ce texte.

Dans la derni\`ere phase de r\'evision de cet article,
le second auteur \'etait invit\'e par l'Institute for Advanced Study
de Princeton et b\'en\'eficiait du soutien de la NSF,
via l'accord n\textsuperscript o~DMS-06350607.

%Serge l\'eg\`ere variation sur cette derni\`ere phrase : suppression d'un "et"

%%%%%%%%%%%%%%%%%%%%%%%%%%%%%%%%%%%%%%%
%%%%%%%%%%%%%%%%%%%%%%%%%%%%%%%%%%%%%%%

\section{Espace de Berkovich}\label{par:berko}

%%%%%%%%%%%%%%%%%%%%%%%%%%%%%%%%%%%%%%%
%%%%%%%%%%%%%%%%%%%%%%%%%%%%%%%%%%%%%%%

Dans cette partie nous pr\'esentons une construction {\og \`a la Dedekind \fg}
de l'espace de Berkovich, aussi appel\'e espace hyperbolique~$p$-adique. 
Comme nous le verrons, l'in\'egalit\'e 
triangulaire ultram\'etrique conf\`ere \`a l'espace des boules de 
$\P^1(\C_p)$ une structure d'arbre r\'eel ; c'est ce qui permet 
de construire l'espace hyperbolique~$p$-adique.
L'approche de Berkovich sera ensuite rapidement r\'esum\'ee. 

\subsection{Notations}\label{par.notations}

Soit $p$ un nombre premier, fix\'e dans tout ce texte.
Nous noterons $\F_p$ le corps \`a~$p$ \'el\'ements.

La valeur absolue $p$-adique~$\abs\cdot_p$ sur
le corps~$\Q$  des nombres rationnels est d\'efinie par  
\[
\abs{p^r\frac ab}_p=p^{-r}
\]
si $a$, $b$ et~$r$ sont des entiers, $a$ et~$b$ n'\'etant pas multiples de~$p$,
et $\abs 0_p=0$. 
Cette valeur absolue est \emph{ultram\'etrique}
au sens o\`u $\abs{x+y}_p\leq\max(\abs x_p,\abs y_p)$ pour 
tous nombres rationnels $x$ et~$y$.
Le corps~$\Q$ n'est pas complet pour cette valeur absolue\footnote{Cela contredirait le th\'eor\`eme de Baire. L'efficacit\'e $p$-adique de la m\'ethode de Newton permet aussi d'expliciter des suites de Cauchy qui n'ont pas de limite ; par exemple la suite $(u_n)$ d\'efinie par $u_0=1$
et $u_{n+1}=(u_n-7/u_n)/2$ n'a pas de limite dans~$\Q$
(muni de la topologie induite par une valeur absolue arbitraire,
par exemple $2$-adique)
mais converge dans~$\Q_2$ vers une racine carr\'ee de~$-7$.}; 
on note $\Q_p$ son compl\'et\'e;
la valeur absolue $p$-adique s'\'etend de mani\`ere unique 
en une valeur absolue ultram\'etrique sur~$\Q_p$.
Pour all\'eger la notation,
on notera $\abs\cdot$ la valeur
absolue $p$-adique 
dans la suite de ce texte.

Par l'in\'egalit\'e ultram\'etrique,
l'ensemble des nombres $p$-adiques de valeur absolue inf\'erieure
ou \'egale \`a~$1$ est un sous-anneau de~$\Q_p$ 
qu'on note~$\Z_p$ et qu'on appelle l'anneau des entiers $p$-adiques.

Soit $\overline{\Q_p}$ une cl\^oture alg\'ebrique du corps~$\Q_p$.
La valeur absolue $p$-adique s'\'etend de mani\`ere unique \`a~$\overline{\Q_p}$
mais ce corps n'est pas complet\footnote{Il est de dimension infinie d\'enombrable sur $\Q_p$.}.
Son compl\'et\'e, le corps des {\og nombres complexes $p$-adiques\fg}
est alors not\'e~$\C_p$; c'est un corps complet et alg\'ebriquement
clos.

L'ensemble $\abs{\Q_p^*}$
des valeurs prises par la valeur absolue sur~$\Q_p^*$
co\"{\i}ncide avec l'ensemble~$p^\Z$ des puissances enti\`eres de~$p$ ; 
celui des valeurs prises par la valeur
absolue sur~$\overline{\Q_p}^*$ est \'egal \`a~$p^\Q$,
et ceci reste vrai pour $\C_p^*$.

La valeur absolue~$p$-adique \'etant ultram\'etrique, 
l'ensemble~$\mathscr O$ des \'el\'ements~$z\in\C_p$
tels que $\abs z\leq 1$ est un anneau, l'\emph{anneau des entiers} de~$\C_p$.
On a $\mathscr O\cap\Q_p=\Z_p$.
Le sous-ensemble~$\mathfrak m$ de~$\mathscr O$ form\'e des~$z$ tels
que $\abs z<1$ en est l'unique id\'eal maximal. Son corps r\'esiduel
$\mathscr O/\mathfrak m$ est une cl\^oture alg\'ebrique de~${\F_p}$ ; on la
note $\bar{\F_p}$ et l'on note $z\mapsto \bar z$ l'application 
de r\'eduction de~$\mathscr O$ dans~$\bar{\F_p}$.

%Serge Changement de \mathfrak en \mathscr pour O
%Serge Changement pour la structure de la derni\`ere phrase.

Nous notons $\P^1(\C_p)$ la droite projective sur~$\C_p$,
r\'eunion de~$\C_p$ et d'un point~$\infty$. Les coordonn\'ees
homog\`enes sur $\P^1(\C_p)$ seront not\'ees $[x:y]$ ; la
coordonn\'ee affine sur $\C_p$ sera not\'ee $z$, de sorte que
$z=x/y$. Enfin, $\Mob$ d\'esigne le groupe $\PSL_2(\C_p)$ des transformations
de M\"obius \`a coefficients dans~$\C_p$; il op\`ere naturellement
par homographies sur~$\P^1(\C_p)$.  
De m\^eme, $\P^1(\overline{\F_p})=\bar{\F_p}\cup\{\infty\}$
est la droite projective sur~$\overline{\F_p}$; son groupe
d'homographies est le groupe~$\PSL_2(\overline{\F_p})$.

L'application de r\'eduction modulo $\mathfrak m$ d\'efinie sur~$\mathscr O$
s'\'etend en une application de~$\P^1(\C_p)$ \`a valeurs
dans $\P^1(\bar{\F_p})$, 
\[
\pi: \P^1(\C_p)\to \P^1(\bar{\F_p}),
\] 
pour laquelle
$\pi([x:y])=\infty$ quel que soit le point~$z$ satisfaisant $\abs{z} > 1$,
y compris $z=\infty$.
On peut d\'ecrire~$\pi$ de la fa\c{c}on suivante.
Soit $z$ un point de~$\P^1(\C_p)$ de coordonn\'ees homog\`enes~$[x:y]$;
quitte \`a diviser $x$ et $y$ par celui dont la valeur absolue est \'egale \`a $\max(\abs x,\abs y)$, nous pouvons supposer
que $\max(\abs x,\abs y)=1$; ceci  signifie que $x$ et $y$ appartiennent \`a~$\mathscr O$ mais qu'ils n'appartiennent pas tous deux \`a~$\mathfrak m$.
Alors, $\pi(z)$ est le point de~$\P^1(\overline{\F_p})$
de coordonn\'ees homog\`enes~$[\bar x:\bar y]$.
Ce point ne d\'epend que de $z$ et pas du choix de~$x$ et~$y$; si $z\in\mathscr O$,
c'est-\`a-dire si l'on peut prendre $y=1$, alors $[\bar x:\bar y]=[\bar z:\bar 1]$;
sinon, on peut prendre $x=1$ et $y=1/z$, d'o\`u $[\bar x:\bar y]=[\bar 1:\bar 0]=\infty$.

\def\bFp{\bar{\F_p}}
Pour tout \'el\'ement $\alpha$ de $\P^1(\bar{\F_p})$, nous noterons
$B(\alpha)$ la partie de~$\P^1(\C_p)$ d\'efinie comme 
la fibre $\pi^{-1}(\alpha)$ de~$\pi$. 
Ainsi, lorsque $\alpha\neq \infty$,
 $B(\alpha)$ est l'ensemble des $z\in \mathscr O$ 
tel que $z\equiv \alpha\pmod{\mathfrak m}$. 
% soit encore
% $\abz{z-a}<1$ si $a$ est un \'el\'ement quelconque de~$\mathfrak O$
% tel que $a\equiv\alpha \pmod{\mathfrak m}$.
De m\^eme,  $B(\infty)$ est l'ensemble $\{\abs z>1\}\cup\{\infty\}$.
Ces parties sont appel\'ees \emph{boules} (voir
le paragraphe~\ref{par.boules}).

\subsection{R\'eduction}\label{par:reduction}

Soit $R$ une fraction rationnelle de degr\'e~$d$ \`a coefficients dans~$\C_p$
qu'on \'ecrit sous la forme $R=P/Q$ pour deux
polyn\^omes~$P$ et~$Q$ \`a coefficients dans~$\C_p$ premiers entre eux avec
$d=\max(\deg P,\deg Q)$. 
L'endomorphisme de $\P^1(\C_p)$ d\'etermin\'e par~$R$ est encore not\'e~$R$. 
En coordonn\'ees 
homog\`enes $[x:y]$, avec $z=x/y$ et $\infty= [1:0]$, nous avons donc 
\[
R([x:y])=[A(x,y):B(x,y)]
%% acl - ajout de parentheses
\]
o\`u $A$ et $B$ sont les polyn\^omes homog\`enes de degr\'e~$d$ d\'efinis
par $A(X,Y)=Y^d P(X/Y)$ et $B(X,Y)=Y^d Q(X/Y)$. 
Quitte \`a diviser~$A$ et~$B$ par le coefficient de plus grande valeur absolue, 
nous pouvons supposer que tous les
coefficients $a_i$ et $b_j$ de~$A$ et de~$B$ appartiennent \`a~$\mathscr O$ 
et qu'au moins  un des coefficients n'est pas dans $\mathfrak m$ ; 
autrement dit, 
\[
A(X,Y) =\sum a_i X^i Y^{d-i}, \quad B(X,Y)=\sum b_j X^j Y^{d-j}
\]
avec 
\begin{equation}\label{eq:normalisation}
\max_{i,j}(\abs{a_i}, \abs{b_j})=1.
\end{equation}
Notons $\bar A$ et $\bar B$ les polyn\^omes \`a coefficients dans~$\bFp$
obtenus de~$A$ et~$B$ par r\'eduction modulo~$\mathfrak m$ 
de leurs coefficients.
Si $\bar B\neq 0$, on d\'efinit la r\'eduction de~$R$ 
modulo $\mathfrak m$ comme la fraction
rationnelle $\bar R=\bar A/\bar B$; elle ne d\'epend pas
des choix faits pour~$A$ et~$B$.
Lorsque $\bar B=0$, on pose $\bar R=\infty$.

\begin{eg}
Supposons que $R(z)=z^2 + c$.
Si $\abs c\leq 1$,  on peut prendre $A(x,y)=x^2+cy^2$ et $B(x,y)=y^2$,
de sorte que $\bar{R}(z)= z^2+\bar c$ ; en particulier,
$\bar R(z)=z^2$ si $\abs c<1$.
En revanche, si $\abs c>1$, on prend $A(x,y)=c^{-1}x^2+y^2$ et $B(x,y)=c^{-1}y^2$, et l'on a $\bar{R}(z)=\infty$.

Lorsque $R$ est la transformation de M\"obius d\'efinie 
 par $R(z)=pz +1/p$,  alors $\bar{R}(z)=\infty$ pour tout $z\in \P^1(\bar{\F}_p)$. 
 En particulier, $\bar R$ n'est pas inversible. 
Cet exemple montre en particulier que
$\bar{R}\circ \pi$ ne co\"{\i}ncide pas toujours avec $\pi \circ R$.
\end{eg}

On dit que~$R$ a une r\'eduction {\emph{non triviale}} si $\bar R$ n'est pas 
constante, ce qui revient \`a dire que $\bar A$ et $\bar B$ 
ne sont pas multiples l'un de l'autre. 

Le degr\'e de $\bar R$ est inf\'erieur ou \'egal \`a celui de~$R$; en
fait, la diff\'erence $\deg R-\deg \bar R$ est \'egale au degr\'e
du pgcd des polyn\^omes~$\bar A$ et~$\bar B$.
On dit que~$R$ a \emph{bonne r\'eduction} lorsque le degr\'e de~$\bar R$
est \'egal au degr\'e de~$R$, autrement dit lorsque $\bar A$
et $\bar B$ sont premiers entre eux. Si c'est le cas,  on a 
\[
\bar{R}\circ \pi = \pi\circ R.
\]
Soit en effet $z=[x:y]$ un point de~$\P^1(\C_p)$, o\`u $x$ et $y$
sont des \'el\'ements de~$\mathscr O$ tels que $\max(\abs x,\abs y)=1$.
Comme $\bar A$ et $\bar B$ sont premiers entre eux, 
$\bar A(\bar x,\bar y)$ et $\bar B(\bar x,\bar y)$ ne sont
pas tous deux nuls. Puisque la r\'eduction modulo~$\mathfrak m$
de $A(x,y)$ est $\bar A(\bar x,\bar y)$, et de m\^eme pour $B(x,y)$,
la r\'eduction modulo~$\mathfrak m$ du point $R(z)=[A(x,y):B(x,y)]$ 
est le point $[\bar A(\bar x,\bar y):\bar B(\bar x,\bar y)]$.
De plus, $\pi(z)=[\bar x:\bar y]$
et  $\bar R(\pi(z))$ est donc le point
de coordonn\'ees homog\`enes $[\bar A(\bar x,\bar y):\bar B(\bar x,\bar y)]$.

Dans le cas de bonne r\'eduction,
l'image r\'eciproque par~$R$ de toute boule $B(\alpha)$ 
est donc la r\'eunion des boules $B(\beta)$, pour 
$\beta \in \bar{R}^{-1}(\alpha )$.

\medskip 

Revenons au cas g\'en\'eral en conservant les hypoth\`eses faites sur~$A$ et~$B$: leurs coefficients appartiennent \`a~$\mathscr O$ et ne sont pas tous dans
l'id\'eal maximal~$\mathfrak m$. 
%% acl - je reprend les hypotheses
Soit $\Res(A,B)$
le r\'esultant{\footnote{Il s'agit du r\'esultant des polyn\^omes homog\`enes~$A$ et~$B$
de degr\'es~$d$; c'est aussi celui des polyn\^omes~$A(x,1)$ et~$B(x,1)$ consid\'er\'es comme polyn\^omes de degr\'es~$\leq d$. Il se calcule par exemple comme le d\'eterminant d'une matrice carr\'ee
(\emph{matrice de Sylvester}) de taille~$2d$.}} des polyn\^omes~$A$ et~$B$ ; notons
\[
\Delta(R) = \abs{\Res(A,B)}
\]
sa valeur  absolue; c'est une variante du \emph{discriminant} de~$R$
qu'ont introduit Morton et Silverman dans~\cite{Morton-Silverman:Crelle}.
%% acl - ajout de deux lignes
Comme le r\'esultant $\Res(A,B)$ 
est un polyn\^ome \`a coefficients entiers en les coefficients des polyn\^omes~$A$
et~$B$,
$\Res(A,B)$ appartient \`a~$\mathscr  O$,
donc $\Delta(R)\leq 1$ (cf. \'equation~\ref{eq:normalisation}).
%% acl - enleve la valeur absolue
En outre, le r\'esultant de $\bar A$ et $\bar B$ co\"{\i}ncide avec la r\'eduction 
de $\Res(A,B)$. Par suite, $\overline{\Res(A,B)}\neq 0$ si et seulement
si les polyn\^omes~$\bar A$ et $\bar B$ n'ont pas de racine
commune dans $\bar{\F}_p$. Autrement dit, 
%Serge J'ai enlever les valeurs absolues autour de Delta car Delta est d\'ej\`a une valeur absolue par d\'efinition Ainsi abs(Delta) = 1 ssi Delta =1.
${\Delta(R)}=1$ \'equivaut \`a ce que $R$ ait bonne r\'eduction. 
Le nombre $\Delta(R)$ permet donc de mesurer le {\og d\'efaut \fg} 
de bonne r\'eduction (voir~\cite{Morton-Silverman:Crelle}, \S4,
ainsi que le paragraphe \ref{par:bonne-reduction}).

\begin{rem}
La notion de bonne r\'eduction d\'epend du choix de coordonn\'ees homog\`enes
utilis\'ees pour d\'efinir la r\'eduction de~$R$. 
Si l'on conjugue~$R$ par une transformation
de M\"obius, la fraction rationnelle obtenue peut avoir bonne r\'eduction sans que
ce soit le cas pour~$R$. 
Nous dirons que $R$ est \emph{simple} si~$R$ 
est conjugu\'ee 
par un \'el\'ement de $\Mob$
\`a une transformation ayant bonne r\'eduction (voir~\S \ref{par:bonne-reduction}).
\end{rem}

\begin{ex}\label{ex.pol-br}
Soit $P$ un polyn\^ome de degr\'e~$d\geq 2$, $P(z)=\sum_{n=0}^d a_n z^n$. 
%Serge P ---> P(z)

1) D\'emontrer que $P$ a bonne r\'eduction si et seulement
si $\abs {a_d}=1$ et $\abs{a_n}\leq 1$ pour tout~$n$.

2) D\'emontrer que $P$ est simple si $\abs{a_n}^{d-1}\leq \abs{a_d}^{n-1}$  
pour tout~$n$. (Faire un changement de variables $z'=\lambda z$.)

3) Donner un exemple de polyn\^ome simple qui n'a pas bonne r\'eduction.
\end{ex}

\subsection{Boules de la droite projective}
\label{par.boules}

Pour $a\in\C_p$ et $r\in \R_+^*$, notons $\Bf(a,r)$
et $\Bo(a,r)$ l'ensemble des $z\in\C_p$ tels que $\abs{z-a}\leq r$
et $\abs{z-a} < r$ respectivement.
%Serge Ajout d'un respectivement
Ces ensembles sont appel\'es 
\emph{boules} (ou \emph{disques}) respectivement
circonf\'erenci\'ees et non circonf\'erenci\'ees.
(Contrairement \`a l'usage en topologie g\'en\'erale,
il convient ici de ne pas consid\'erer les singletons comme des boules.)
% Antoine: points ne sont pas des boules... :-(
Le diam\`etre et le rayon d'une telle boule sont tous deux \'egaux \`a~$r$ et,
contrairement au cas archim\'edien, tout point d'une boule en est un centre;
de fait, on a $\Bo(z,r)=\Bo(a,r)$ pour tout $z\in\Bo(a,r)$
et $\Bf(z,r)=\Bf(a,r)$ pour tout $z\in\Bf(a,r)$.
Suivant les valeurs du rayon, ces boules doivent \^etre
distingu\'ees. Si $r\in p^\Q$, on dira que $\Bo(a,r)$ est une \emph{boule ouverte}
et que $\Bf(a,r)$ est une \emph{boule ferm\'ee}.
Si $r\in\R_+^*\setminus p^\Q$, les boules $\Bo(a,r)$
et $\Bf(a,r)$ co\"{\i}ncident; on dit que cette boule est \emph{irrationnelle}.

La terminologie utilis\'ee est uniquement mn\'emotechnique.
En effet, la topologie de~$\C_p$ d\'efinie par la valeur absolue $p$-adique \'etant
totalement discontinue, toutes ces boules sont \`a la fois
ouvertes et ferm\'ees. 

Dans la droite projective~$\P^1(\C_p)$, on dispose, en plus des boules
pr\'ec\'edentes, de boules contenant~$\infty$: on dira qu'une partie~$B$
de~$\P^1(\C_p)$ contenant~$\infty$ est une boule ouverte
(\resp ferm\'ee, \resp  irrationnelle) si son compl\'ementaire
$\complement B$ est une boule ferm\'ee (\resp ouverte, \resp irrationnelle)
dans~$\C_p$.

Chaque boule ouverte $\Bo (a,r)$ est contenue dans une unique plus petite
boule ferm\'ee, \`a savoir $\Bf (a,r)$.
Par contre, $\Bf (a,r)$ contient une infinit\'e de boules ouvertes de 
rayon~$r$. Lorsque $a=0$ et $r=1$, ces boules sont pr\'ecis\'ement
les $B(\alpha)$, pour $\alpha\in\overline{\F_p}$
(fin du \S\ref{par.notations}).

On appellera {\emph{couronne}} une
partie~$C$ de $\P^1 (\C_p)$ dont le
compl\'ementaire est la r\'eunion disjointe de deux
parties qui sont des boules ou des points.

\begin{rem}\begin{enumerate}
\item
La boule ferm\'ee $\Bf (0,1)$  co\"{\i}ncide avec l'anneau 
$\mathscr O$. Si~$a$ appartient \`a~$\mathscr O$, la boule ouverte de centre 
$a$ et de rayon~$1$ co\"{\i}ncide avec la classe de~$a$ modulo l'id\'eal maximal~$\mathfrak m$. 

\item Soit $\alpha$ un \'el\'ement de $\P^1(\bar{\F_p})$ distinct de~$\infty$. 
La fibre $B(\alpha)$ de la projection~$\pi$
s'identifie \`a la boule ouverte de rayon~$1$ 
centr\'ee en tout \'el\'ement de $ \C_p$ congru \`a~$\alpha$  
modulo~$\mathfrak m$.
%\footnote{Trop t\^ot pour le dire si la boule contient l'infini.}
%Serge Je ne vois pas pourquoi . On a d\'ej\`a d\'efini avant les boules centr\'ees \`a l'infini.
%Antoine: pas leur rayon, et pas "congru"

\item Les transformations de M\"obius envoient boules sur boules en pr\'eservant  le type de celles-ci. En effet, le groupe
de M\"obius est engendr\'e par les translations et l'involution $z\mapsto 1/z$
et ces applications envoient boules sur boules en pr\'eservant leur type.
\end{enumerate}
\end{rem}

%Serge Dans l'exo qui suit on s'emb\^ete avec la distinction point / boule, qui n'a pas vraiment lieu d'\^etre puisqu'un point est une boule (ferm\'ee) de rayon 0 (voir la def, on accepte r=0); cette convention ne semble pas induire de contradiction .... A voir.

\begin{ex}\label{ex:boules0}
1) Soient $B$ et $B'$ deux boules d'intersection non vide
telles que $B\cup B'\neq\P^1(\C_p)$.
D\'emontrer que $B\subset B'$ ou $B'\subset B$. 
(Commencer par traiter le cas o\`u $B$ et $B'$ ne contiennent
pas~$\infty$.)

2) En d\'eduire qu'une r\'eunion de boules dont les intersections deux \`a deux 
ne sont pas vides  est ou bien une boule, ou bien le compl\'ementaire
d'un point dans~$\P^1(\C_p)$, ou bien $\P^1(\C_p)$.

3) En d\'eduire qu'une intersection d\'ecroissante de boules 
est soit vide, soit un un point, soit une boule. 
Nous verrons plus loin, notamment au paragraphe 
\S\ref{par:exempleRL}, qu'une telle intersection peut \^etre vide.
\end{ex}

\begin{prop}\label{prop:boules}
Soit~$P\in\C_p[z]$ un polyn\^ome non constant 
% en une variable~$p$-adique 
et soit $B$ une boule de~$\C_p$.

1) L'image de~$B$ par~$P$ est une boule de m\^eme nature.

2) Il existe un entier~$t\in\{1,\dots,\deg(P)\}$
et des boules $B_1,\ldots,B_t$ de m\^eme nature que~$B$, deux \`a deux disjointes,
telles que l'image r\'eciproque de~$B$ par~$P$ soit la r\'eunion des~$B_i$.
\end{prop}

L'exercice suivant sera utile \`a la d\'emonstration. 
\begin{ex}
Soit $P(z)=\sum a_k z^k$ un polyn\^ome de~$\C_p[z]$.
D\'emontrer que 
\[
\sup_{\Bo(0,r)}\abs{P} = \sup_{\Bf(0,r)} \abs{P} = \max ( \abs{a_k} r^k)
\]
(voir la preuve du lemme \ref{lemm.gauss} pour un argument complet).
\end{ex}

\begin{proof}[D\'emonstration (d'apr\`es~\cite{amice75}, p.~159)]
Le lecteur pourra consulter \cite{amice75} pour une preuve
de l'assertion~1).
On ne traite en outre l'assertion~2) que dans le cas o\`u $B=\Bf(0,1)$.
Soit $b$ un point de~$\C_p$ tel que $\abs{P(b)}\leq 1$.
La boule unit\'e $\Bf(0,1)$ \'etant ouverte dans~$\C_p$, il existe un nombre
r\'eel~$r>0$ tel que $\abs{P(z)}\leq 1$ pour tout $z\in\Bf(b,r)$.
La r\'eunion~$B$ de toutes ces boules~$\Bf(b,r)$ est une boule ouverte, 
ferm\'ee ou irrationnelle; 
elle est maximale parmi les boules
contenues dans~$P^{-1}(\Bf(0,1))$ et contenant~$b$.
Notant $r$ le rayon de~$B$, on a $B=\Bo(b,r)$ ou $B=\Bf(b,r)$;
en particulier, $\Bo(b,r)\subset B\subset \Bf(b,r)$.
%Serge L\'eg\`ere variation ci-dessus et ci-dessos
Si  $\sup_{\Bo(b,r)}\abs P<1$, la formule de l'exercice
pr\'ec\'edent prouve qu'on peut augmenter~$r$ tout en restant
contenu dans~$P^{-1}(\Bf(0,1))$. Par suite, $\sup_{\Bo(b,r)}\abs P=1$,
et la formule en question entra\^{\i}ne que $r\in p^\Q$
puis que $B=\Bf(b,r)$. 

%Serge Petits ajouts dans ce qui suit. 
Si deux boules de $\C_p$ se coupent, l'une est incluse dans l'autre. 
Si l'on note $B_1,B_2 \ldots$ les boules maximales construites
en faisant varier~$b$,
on obtient donc des boules ferm\'ees deux \`a deux disjointes dont 
la r\'eunion est \'egale \`a~$P^{-1}(B)$.

Pour conclure la d\'emonstration, il suffit de prouver que 
$P$ s'annule sur chacune des boules~$B_i$.
Posons $B_i=\Bf(b_i,r_i)$ et~\'ecrivons $P$
sous la forme
\[
P(z)=c\prod_{j=1}^{\deg(P)}(z-z_j),
\] 
o\`u $z_1,\dots,z_{\deg(P)}$
sont les z\'eros de~$P$, r\'ep\'et\'es suivant leur multiplicit\'e.
Si $P$ ne s'annule pas sur~$B_i$, on a $\abs{z_j-b_i}>r_i$
pour tout~$j$, 
donc  $\abs{z-z_j}=\abs{(b_i-z_j)+(z-b_i)} = \abs{b_i-z_j}$
pour tout~$j$ et tout $z\in B_i$.
Ainsi, $\abs{P}$ est constant sur~$B_i$,
de module $\abs c\prod \abs{b_i-z_j}$.
Ce calcul reste valable  pour tout point~$z$
de la boule~$\Bo(b_i,\min(\abs{z_j-b_i}))$, et cette boule  contient strictement~$B_i$.
Cette contradiction prouve que $P$ s'annule sur~$B_i$.
\end{proof}

\begin{ex}\label{ex:boules}
1) On consid\`ere le polyn\^ome $P(z)=z^2-1$. Pour tout nombre r\'eel~$r>0$, d\'eterminer l'image r\'eciproque par~$P$ des boules $\Bo(0,r)$ et $\Bf(0,r)$. En d\'eduire que la proposition pr\'ec\'edente ne s'\'etend pas aux boules contenant~$\infty$.

2) Soient $R_1$ et $R_2$ les fractions rationnelles d\'efinies par  $R_1(z)=z/p-z^2$
et $R_2(z)=z^2-1/z$. D\'ecrire l'image r\'eciproque de la boule $\Bf (0,r)$ par $R_1$ et 
par $R_2$  en fonction de $r>0$.

3) Compl\'eter la d\'emonstration de la proposition~\ref{prop:boules} afin de traiter
le cas g\'en\'eral.
\end{ex}

\begin{ex}\label{ex:boules-et-ptfixes}
1) Soit $P$ et $Q$ deux polyn\^omes d'une variable \`a coefficients dans~$\C_p$. 
%Serge Chgt ci-dessus
Soit $B$ 
une boule de~$\C_p$. On note $r$ le rayon de la boule $P(B)$ et
l'on suppose que $\abs{P(z)-Q(z)} < r$ pour tout point $z$ de $B$. Montrer
que $Q(B)=P(B)$.

2) Soit $P$ un \'el\'ement de~$\C_p[z]$ et soit $B$ une boule. 
%Serge Chgt ci-dessus
On suppose que $P(B)$ contient strictement la boule $B$. Montrer que
$P$ a un point fixe dans $B$. (Indication : se ramener au cas
o\`u $B$ est centr\'ee \`a l'origine, puis montrer que le polyn\^ome
$Q$ d\'efini par $Q(z)=P(z)-z$ envoie $B$ sur $P(B)$ surjectivement)
\end{ex}

Nous nous sommes content\'es de d\'ecrire l'action des polyn\^omes
sur les boules, mais les r\'esultats qui viennent d'\^etre \'enonc\'es 
peuvent \^etre \'etendus aux fractions rationnelles. Le lecteur pourra se
reporter, par exemple, au chapitre 2 de~\cite{Rivera-Letelier:Asterisque}
ou aux pr\'eliminaires de~\cite{Benedetto:2001} pour des \'enonc\'es ou
des r\'ef\'erences pr\'ecises.

\subsection{Construction {\og \`a la Dedekind\fg} de l'espace
hyperbolique~$p$-adique}\label{par:dedekind}

L'espace hyperbolique $p$-adique~$\Hp$, que nous allons maintenant d\'efinir, est 
r\'eunion de trois sous-ensembles disjoints
not\'es $\HpQ$, $\Hp^{\R\setminus\Q}$ et $\Hp^\sing$
dont les \'el\'ements sont d\'enomm\'es respectivement 
{\emph{points rationnels, irrationnels}} et  {\emph{singuliers}}. 
L'ensemble $\HpR$ d\'esignera
la r\'eunion de $\HpQ$ et $\Hp^{\R\setminus\Q}$.

\begin{defi}
On appelle \emph{coupure irrationnelle} une partition $S=\{B,B'\}$
de~$\P^1(\C_p)$ en deux boules irrationnelles disjointes.

On appelle \emph{coupure canonique} la partition
$S_\can=\{ B(\alpha)\, ; \, \alpha\in\P^1(\overline{\F_p})\}$ de~$\P^1(\C_p)$
en les boules ouvertes $B(\alpha)$ index\'ees par~$\P^1(\bar\F_p)$.

Une \emph{coupure rationnelle} est une partition de~$\P^1(\C_p)$
qui se d\'eduit de la coupure canonique par un \'el\'ement de $\Mob$.
\end{defi}

Les \'el\'ements de la coupure canonique sont donc les fibres de 
$\pi : \P^1(\C_p)\to \P^1(\bar{\F_p})$. 

On remarquera qu'une coupure irrationnelle est d\'etermin\'ee par
sa boule irrationnelle qui ne contient pas~$\infty$. De m\^eme, une 
coupure rationnelle est d\'etermin\'ee  par son unique boule ouverte
contenant l'infini, donc aussi par son compl\'ementaire, 
qui est une boule \emph{ferm\'ee} ne contenant pas~$\infty$. 
On identifiera toute boule irrationnelle (\resp ferm\'ee) avec la coupure
irrationnelle (\resp rationnelle) qui lui est ainsi associ\'ee. 
Ceci d\'efinit une bijection entre l'ensemble
des boules ferm\'ees ou irrationnelles ne contenant pas l'infini et l'ensemble des coupures.
% Serge Attention : j'ai rajout\'e "ne contenant pas l'infini" 

Pour tout nombre r\'eel~$r>0$, on notera ainsi  $S(r)$  la coupure associ\'ee
\`a la boule~$\Bf(r)$; c'est une coupure irrationnelle 
si $\log_p (r)$ est irrationnel,
et une  coupure rationnelle
(obtenue par homoth\'etie de rapport~$\xi$ avec $\abs{\xi}=r$ \`a partir de la coupure 
%Serge ajout de \xi pour d\'efinir homoth\'etie
canonique) si $\log_p (r)$ est rationnel. 
Les coupures $S(a,r)$ sont d\'efinies
de fa\c{c}on similaire.

Ce dictionnaire entre coupures et boules (ferm\'ees ou irrationnelles) 
sera constamment utilis\'e par la suite. 

\medskip

On note $\HpR$ l'ensemble des coupures. C'est la r\'eunion de deux
parties, $\HpQ$ et $\Hp^{\R\setminus\Q}$,
correspondant respectivement aux coupures rationnelles et irrationnelles.  
Puisque les transformations de M\"obius appliquent boules sur boules
en pr\'eservant leur type,
le groupe~$\Mob$ agit sur~$\HpR$ 
en stabilisant~$\HpQ$ et~$\Hp^{\R\setminus\Q}$. 

%Serge l\'eg\`ere pr\'ecision de l'\'enonc\'e dans l'exo suivant
\begin{ex}\label{exo.stabilisateur}
D\'emontrer que le sous-groupe de $\Mob$ stabilisant
la coupure canonique co\"{\i}ncide avec le groupe
$\PGL_2(\mathscr O)$ et qu'un \'el\'ement de $\PGL_2(\mathscr O)$ fixe
chaque \'el\'ement de $S_\can$ si et seulement si sa r\'eduction modulo
$ \mathfrak m$ est l'identit\'e.  En d\'eduire que  $\HpQ$ s'identifie \`a l'espace
homog\`ene $\PGL_2(\C_p)/\PGL_2(\mathscr O)$.
\end{ex}

\subsection{Structure d'arbre r\'eel}

Un \emph{arbre r\'eel} est 
un espace m\'etrique $(X, \dist)$ tel qu'il existe un unique
arc topologique entre deux points distincts donn\'es et dans lequel
tout arc peut-\^etre param\'etr\'e isom\'etriquement par un intervalle de $\R$.
Un tel arc isom\'etrique \`a un intervalle sera appel\'e \emph{chemin g\'eod\'esique.}
Le but de ce paragraphe est de munir $\HpR$ d'une structure d'arbre
r\'eel en suivant la construction pr\'esent\'ee dans~\cite{Rivera-Letelier:CMH}. 

\begin{lemm}
Soient $S, S'$ deux coupures distinctes.

\begin{enumerate}
\item Il existe un unique couple $(B,B')$ de boules
ouvertes ou irrationnelles 
tel que $B\in S$, $B'\in S'$ et $\P^1(\C_p)=B\cup B'$.
\item Il existe $g\in\Mob$ et des nombres r\'eels strictement 
positifs $r_0$ et $r_1$ tels que 
$g(B\cap B')$ soit la couronne $\{r_0<\abs z<r_1\}$;
de plus, le rapport~$r_1/r_0$ ne d\'epend pas du choix d'une telle
homographie~$g$.
\end{enumerate}
\end{lemm}

\begin{proof}
1) Commen\c{c}ons par d\'emontrer l'unicit\'e du couple
de boules $(B,B')$. Soient en effet deux couples $(B_1,B'_1)$
et $(B_2,B'_2)$ de boules telles que $\P^1(\C_p)=B_1\cup B'_1=B_2\cup B'_2$,
o\`u $B_1,B_2\in S$ et $B'_1,B'_2\in S'$. 
Si $B_1=B_2$, $B'_1$ et $B'_2$ ont tout point de~$\complement B_1$ en commun, 
donc sont égales, par définition d'une coupure. Sinon, $B_1$ et~$B_2$
sont disjointes ainsi, par le même argument, que $B'_1$ et~$B'_2$.
On a donc  $B'_2\subset\complement B'_1\subset B_1$,
car $B_1\cup B'_1=\P^1(\C_p)$. Par sym\'etrie, $B_1\subset B'_2$,
d'o\`u  $B_1=B'_2$. 
Par d\'efinition d'une coupure, cela entra{\^{\i}}ne  $S=S'$, 
ce qui est contraire \`a l'hypoth\`ese.

D\'emontrons maintenant l'existence d'un tel couple.
Soient $D$ et $D'$ les boules ferm\'ees ou irrationnelles
ne contenant pas~$\infty$ correspondant aux coupures~$S$ et~$S'$;
par construction, $\complement D\in S$ et $\complement D'\in S'$.
Les boules~$\complement D$ et $\complement D'$ ont le point~$\infty$
en commun; si $\complement D\cup\complement D'\neq\P^1(\C_p)$,
l'une est contenue dans l'autre. Supposons, ce qui est loisible,
que $\complement D\subset\complement D'$, c'est-\`a-dire $D'\subset D$.
Par translation, on peut alors supposer $D'=\Bf(r')$ et $D=\Bf(r)$,
o\`u $r$ et~$r'$ sont deux nombres r\'eels tels que $r'\leq r$.
On a $r\neq r'$ car $S\neq S'$.
Si $D$ est irrationnelle, $D\in S$ et $D\cup\complement D'=\P^1(\C_p)$ ; il
suffit donc de prendre $B=D$ et $B'=\complement D'$.
En revanche, si $D$ est rationnelle, nous pouvons,
quitte \`a effectuer une homographie, supposer que $S=S_\can$,
d'o\`u $D=\Bf(1)$;
alors, $\Bo(1)\in S$. En prenant $B= \Bo(1)$ et
$B'= \complement D'$, nous obtenons  $B\cup B' =\P^1(\C_p)$.

2) Comme $\Mob$ est $2$-transitif, on peut supposer
que $0$ appartient \`a~$B\setminus B'$ et $\infty$ appartient \`a ~$B'\setminus B$.
Alors, il existe $r_0$ et $r_1$
tels que $B=B(r_1)$, $B'=\{\abs z>r_0\}$; comme $S\neq S'$,
$r_0\neq r_1$.
Comme $B\cup B'=\P^1(\C_p)$, on a $r_0\leq r_1$, d'o\`u $r_0<r_1$.
Alors, $B\cap B'$ est la couronne indiqu\'ee.

Soit $C$ cette couronne et soit $g\in\Mob$ tel que $g(C)$
soit encore une couronne $\{s_0<\abs z<s_1\}$.
Quitte \`a  changer $g$ en~$1/g$, ce qui change $(s_0,s_1)$
en~$(1/s_1,1/s_0)$ et ne modifie pas le rapport $s_1/s_0$,
on suppose que $\abs{g(0)}\leq s_0$;
alors, remplacer $g$ par $g-g(0)$ ne modifie pas l'image~$g(C)$,
ce qui permet de supposer $g(0)=0$. Nous pouvons donc
maintenant \'ecrire $g$ sous la forme $g(z)=z/(cz+d)$. 
En conjuguant~$g$ par $z\mapsto 1/z$,
nous sommes ramen\'es au cas o\`u $g(z)=c+zd$ et $g(\infty)=\infty$.
Quitte \`a  composer~$g$ par l'homoth\'etie $z\mapsto z/d$, ce qui 
ne change pas le rapport $s_1/s_0$,
on peut supposer que $d=1$ et $g(z)=z+c$.
Si $s_0<\abs c<s_1$, on trouve que $g(0)\in g(C)$,
donc $0\in C$ ce qui est absurde; si $\abs c\geq s_1$,
on a $\abs{g^{-1}(z)}=\abs{z-c}=s_1$ pour tout $z\in  g(C)$,
ce qui est absurde car $g^{-1}(g(C))=C$; on a donc
$\abs g\leq s_0$ et $g(C)=C$. En particulier, $s_1/s_0$ est \'egal
\`a $r_1/r_0$.
\end{proof}

Utilisons les notations du lemme pr\'ec\'edent et posons d\'esormais $\dist(S,S')=\Log_p(r_1/r_0)$, 
de sorte que $p^{\dist(S,S')}=r_1/r_0$.
Par d\'efinition, cette distance est le {\emph{module}} de la couronne $g(B\cap B')$.
Lorsque~$S$ est \'egale \`a~$S'$, nous poserons  $\dist (S,S') = 0$.
On notera que si~$S$ est une coupure rationnelle,
alors $S'$ est une coupure rationnelle si et seulement
si $\dist(S,S')\in\Q$. 

On d\'emontre que  \emph{$\dist$ est une distance sur l'espace $\HpR$ et l'espace m\'etrique $(\HpR,\dist)$ est un arbre r\'eel}.
Nous renvoyons \`a \cite{Rivera-Letelier:compositio} pour plus de d\'etails,
nous contentant de citer le lemme suivant qui est le coeur 
de la d\'emonstration.

\begin{lemm}\label{lemm:distances}
Soient $S, S',S''$ trois coupures distinctes. De deux choses l'une:
\begin{enumerate}
\item 
Il existe $g\in\Mob$ et des nombres r\'eels $r,r',r''$ strictement positifs
tels que $g(S)=S(r)$, $g(S')=S(r')$ et $g(S'')=S(r'')$.
Si $S$, $S'$ et $S''$ sont telles que $r\leq r'\leq r''$, on a alors
\[ 
\log_p(\frac{r''}{r'})=\dist (S,S'')=\dist(S,S')+\dist(S',S'')=\Log_p \left(\frac{r'}r\right) + \Log_p \left(\frac{r''}{r'}\right) .\]
\begin{figure}[htb]
% \centering\epsfig{figure=distance.pstex}
\centering
\includegraphics{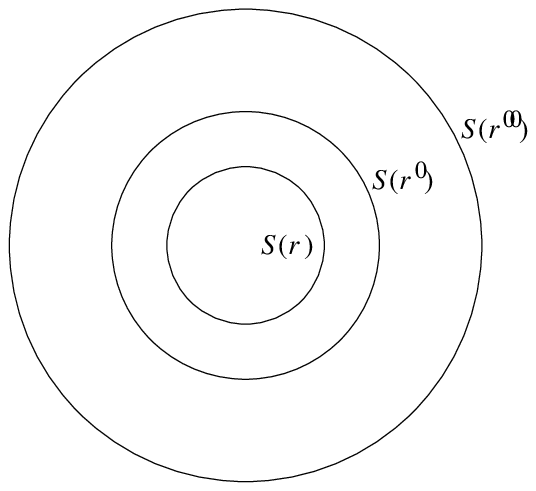} \includegraphics{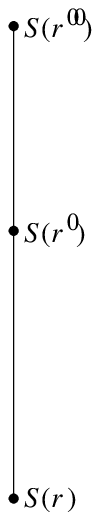} \\
\caption{\textsc{Distances, boules et coupures}: premier cas du lemme~\ref{lemm:distances}. } 
\label{fig:arbre:1}
\end{figure}

\item
Il existe $g\in\Mob$, $a\in\C_p$ et des nombres
r\'eels $r,r',r''$ strictement positifs tels que 
$\max(r,r') < \abs a  < r''$ de sorte que
%% NB: on avait mis \leq 
$g(S) =S(a,r)$, $g(S')=S(r')$ et $g(S'')=S(r'')$.
Dans ce cas, $\Bf(0,\abs a)$ est la plus petite boule qui
contient $\Bf(a,r)$ et $\Bf(r')$. Si l'on note $S^*$ la coupure associ\'ee \`a la boule $g^{-1}\Bf(0,\abs a)$, on a alors 
\[ 
\begin{array}{l}
\dist(S,S')=\dist(S,S^*)+\dist(S^*,S'),\\
 \dist(S,S'')=\dist(S,S^*)+\dist(S^*,S''),\\
 \dist(S',S'')=\dist(S',S^*)+\dist(S^*,S'').
\end{array}
\]
\begin{figure}[htb]
% \centering\epsfig{figure=distance.pstex}
\centering
\includegraphics{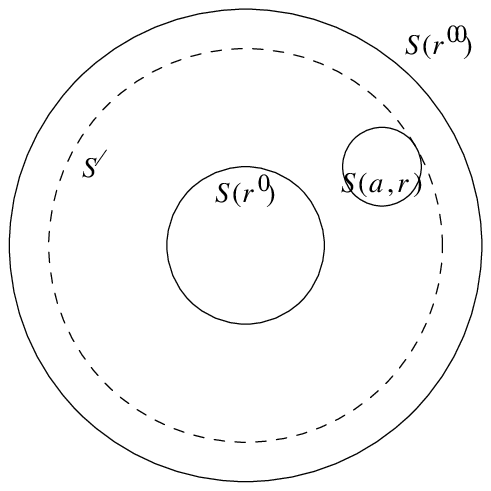} \includegraphics{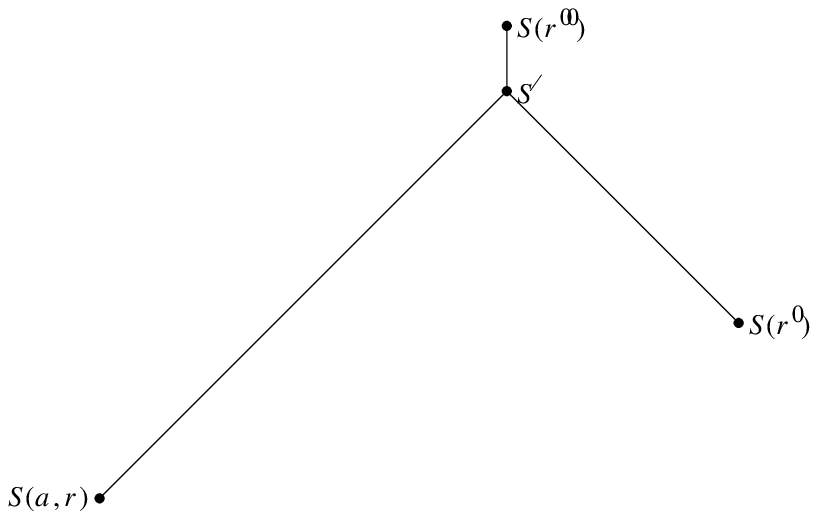} \\
\caption{\textsc{Distances, boules et coupures}: second cas du lemme~\ref{lemm:distances}. } 
\label{fig:arbre:2}
\end{figure}
\end{enumerate}
\end{lemm}

Dans le premier cas, l'ensemble des coupures $g^{-1}(S(s))$ o\`u~$s$
d\'ecrit l'intervalle $[ r, r'']$ est un segment g\'eod\'esique 
qui relie~$S$ \`a~$S''$ (si $r< r' < r''$ ). 
Dans le second cas, il y a trois segments g\'eod\'esiques 
qui relient respectivement 
$S$ \`a~$S'$, $S'$ \`a~$S''$ et $S''$ \`a~$S$ ; leur intersection 
commune est le point $S^*$ (voir les figures~\ref{fig:arbre:1}
et~\ref{fig:arbre:2}).

%%%
\subsection{Branches}\label{par:branches}
%%%

Soit $S$ une coupure.
Les composantes connexes de $\HpR\setminus\{S\}$
seront appel\'ees {\emph{branches}}
issues de~$S$. Nous allons voir que l'ensemble des branches
est en bijection naturelle avec les \'el\'ements de la partition~$S$.
% \footnote{Certains auteurs les appellent \og bouts\fg.} 

\begin{rem} On peut ainsi faire le parall\`ele avec la construction de Dedekind
des nombres r\'eels  selon les coupures.
Partant du corps des nombres rationnels~$\Q$, la coupure~$S(\xi)$ associ\'ee \`a un nombre r\'eel~$\xi$ est
la trace sur~$\Q$ des composantes connexes
du compl\'ementaire de~$\{\xi\}$ dans~$\R$. On peut caract\'eriser \emph{a priori} ces coupures, d'o\`u une construction des nombres r\'eels.
\end{rem}

Commen\c{c}ons par le cas d'une coupure irrationnelle.
Soit~$r$ un nombre r\'eel strictement 
positif pour lequel $\log_p(r)$ est irrationnel. 
Soit~$S$ la coupure $S(r)$ et $\Bf (r)$ la boule (irrationnelle) 
associ\'ee \`a~$S$.  
L'ensemble $\Hp \setminus \{S\}$ comporte deux composantes connexes : 
l'une est form\'ee des coupures
associ\'ees aux boules  (ferm\'ees ou irrationnelles) 
strictement contenues dans $\Bf (r)$, 
l'autre est alors
form\'ee des coupures associ\'ees aux boules qui ne
sont pas contenues dans~$\Bf  (r)$.  

Soit $S=S_\can$ la coupure canonique. 
Rappelons que~$S$ est la partition de $\P^1(\C_p)$ en les boules ouvertes $B(\alpha)$, pour $\alpha$ d\'ecrivant $\P^1({\bar{\F_p}})$,
et que le stabilisateur de~$S$ dans $\Mob$ co\"{\i}ncide 
avec  $\PSL_2(\mathcal{O})$ (exercice~\ref{exo.stabilisateur}).
Pour tout $\alpha\in \bar{\F_p}$,
l'\'el\'ement~$B(\alpha)$ de~$S$ d\'efinit une branche issue de~$S$ : 
les \'el\'ements 
de cette branche sont les coupures associ\'ees aux boules contenues 
dans~$B(\alpha)$; la branche associ\'ee \`a~$\infty$ est
l'ensemble des coupures associ\'ees aux boules qui ne sont
pas contenues dans~$\Bf(0,1)$\footnote{Si une boule ferm\'ee est contenue dans $\Bf(0,1)$
mais ne lui est pas \'egale, elle est contenue dans l'une des boules $B(\alpha)$ 
pour $\alpha\in{\bar{\F_p}}$.}. 
L'ensemble des branches issues de la coupure canonique est ainsi en bijection 
naturelle avec $\P^1(\bar{\F_p})$
et cette bijection est compatible avec l'action du stabilisateur 
$\PSL_2(\mathcal{O})$ de~$S$.

Le cas des coupures rationnelles s'en d\'eduit par  transformation de M\"obius.

\medskip

Si~$S$ et $S'$ sont deux coupures distinctes, $S'$ appartient \`a une unique
branche issue de~$S$. De plus, quitte \`a \'echanger~$S$ et~$S'$,
la boule ferm\'ee ou irrationnelle $B'$ qui
d\'etermine $S'$ est contenue dans l'unique boule~$B$ de la partition 
$S$ qui d\'efinit cette branche ; le module de la couronne
$B\setminus B'$ est la distance de~$S$ \`a~$S'$. 

%%%
\subsection{Compl\'etion: bouts \`a distance finie  et points singuliers} 
%%%

Contrairement \`a ce qui se passe sur le corps des nombres
complexes, 
il existe des suites d\'ecroissantes de boules dans~$\C_p$
dont l'intersection est vide ; on dit que $\C_p$ n'est
pas \emph{sph\'eriquement complet}.
%Serge chgt ci-dessus pour enlever la paranth\`ese
Nous verrons d'ailleurs un exemple de telle suite au paragraphe
\ref{par:exempleRL}, en lien avec la dynamique des polyn\^omes.

Soit $(B_i)$ une telle suite de boules (ferm\'ees ou irrationnelles). 
La suite~$(r_i)$ des rayons des $B_i$ est  d\'ecroissante, donc
converge vers un nombre r\'eel~$r_\infty\geq 0$. Si l'on avait~$r_\infty=0$,
on d\'eduirait du th\'eor\`eme des ferm\'es embo\^{\i}t\'es que l'intersection des~$B_i$
est un singleton,  car $\C_p$ est complet. On a donc $r_\infty>0$.
Soit $(S_i)$ la suite des coupures d\'etermin\'ees par les~$B_i$. Alors
%Serge ci-dessus et dans les lignes qui suivent : r ---> r_\infty 
\begin{itemize}
\item  puisque les $B_i$ sont embo\^{\i}t\'ees, les $S_i$ sont
situ\'ees sur une demi-g\'eod\'esique de $\HpR$ (voir le lemme \ref{lemm:distances});
\item puisque les rayons $r_i$ ont une limite strictement positive $r_\infty$, la
longueur de cette demi-g\'eod\'esique est finie (major\'ee par $\log_p(r_0/r_\infty)$);
\item puisque l'intersection des $B_i$ est vide, cette demi-g\'eod\'esique sort de tout compact
de~$\HpR$ en un temps fini. 
\end{itemize}
D\'efinissons un {\emph{bout}} situ\'e {\emph{\`a distance finie}} comme la donn\'ee 
d'une demi-g\'eod\'esique $\gamma : [t_1,t_2[\to \HpR$ qui sort de tout compact
de $\HpR$ lorsque $t$ tend vers $t_2$, modulo l'\'equivalence qui identifie 
deux demi-g\'eod\'esiques $\gamma$ et $\gamma'$ d\'efinies sur $[t_1,t_2[$ et $[t_1',t_2'[$ 
respectivement si $\gamma(t_2-s)=\gamma'(t_2'-s)$ pour $s$ suffisamment petit.  
Avec cette d\'efinition, l'extr\'emit\'e finale de la demi-g\'eod\'esique contenant les 
$S_i$ d\'etermine un {\emph{bout}} de l'arbre~$\HpR$ situ\'e \`a distance finie de tout point de 
$\HpR$. 
Nous noterons $\Hpsing$ l'ensemble des bouts ainsi construits
et poserons 
\[
\Hp=\HpR\cup\Hpsing.
\] 
Par d\'efinition $\Hpsing$ et $\HpR$ sont deux parties disjointes de $\Hp$.
La distance d\'efinie sur $\HpR$ s'\'etend \`a $\Hp$ : si~$z$ est un point de $\Hpsing$ 
et $x$ est un point de $\HpR$, la distance entre~$x$ et~$z$ est la longueur de
l'unique demi-g\'eod\'esique partant de $x$ et d'extr\'emit\'e finale~$z$. 

On d\'emontre que $\Hp$ est un arbre r\'eel complet ; autrement dit, les seuls points
\`a ajouter \`a $\HpR$ pour le rendre complet sont les points de  $\Hpsing$ construits
ci-dessus. Si $S\in\Hp^\sing$, son compl\'ementaire $\Hp\setminus\{S\}$
est connexe ; autrement dit, le compl\'ementaire de chaque 
point singulier est constitu\'e d'une unique branche.

\subsection{Compl\'etion: sph\`ere \`a l'infini et d\'efinition de $\BP(\C_p)$}\label{par:sphereinfinie}

La {\emph{sph\`ere \`a l'infini}} de l'arbre $\HpR$ est l'ensemble
des demi-g\'eod\'esiques \emph{infinies} modulo la relation d'\'equivalence
pour laquelle $\gamma\simeq \gamma'$ si $\gamma$ et $\gamma'$ co\"{\i}ncident
sur des intervalles $[t_1,\infty\mathclose[$ et $[t_1',\infty\mathclose[$. Les
points de la sph\`ere \`a l'infini sont les  {\emph{bouts}} 
(ou {\og extr\'emit\'es\fg}) situ\'es {\emph{\`a distance infinie}}.

Pour tout point~$a$ de $\C_p$, la suite
des boules $S(a,r)$ pour~$r$ tendant vers~$0$ forme un arc
g\'eod\'esique infini ; on identifie son extr\'emit\'e au point~$a$.
R\'eciproquement, si $\gamma$ est une demi-g\'eod\'esique infinie,
chaque point $\gamma(t)$ est une coupure d\'etermin\'ee par une 
boule ferm\'ee ou irrationnelle $B_t$ de rayon $r_t$. D'apr\`es le lemme
\ref{lemm:distances} les boules $B_t$ sont embo\^{\i}t\'ees. La longueur
totale de $\gamma$ \'etant infinie, $r_t$ tend vers $0$ ou $\infty$.
Supposons que $r_t$ tende vers $0$ ; on obtient ainsi une suite de boules embo\^{\i}t\'ees dont 
l'intersection est un point $a$ de $\C_p$. La demi-g\'eod\'esique
$S(a,r)$ est alors dans la m\^eme classe d'\'equivalence que $\gamma$. 
Si $r_t$ tend vers l'infini, il existe un instant $t_1$ \`a partir duquel les boules
$B_t$ contiennent l'origine $0$; pour $t$ sup\'erieur \`a $t_1$, l'origine peut
\^etre choisie comme centre des boules $B_t$, si bien que la demi-g\'eod\'esique
$\gamma(t)$ est \'equivalente \`a la demi-g\'eod\'esique $S(t)$. L'extr\'emit\'e associ\'ee
est le point \`a l'infini $\infty$.
De la sorte, la sph\`ere \`a l'infini de~$\HpR$ s'identifie
\`a~$\P^1(\C_p)$.
%Serge J'ai essay\'e de pr\'eciser ci-dessus; c'est en fait le m\^eme argument que ce qui est admis au paragraphe pr\'ec\'edent quand on dit que Hpsing suffit \`a la compl\'etion; peut \^etre faut-il \^etre plus pr\'ecis, mais cela me semble superflu pour un texte de synth\`ese.

Lorsqu'on observe l'infini depuis le point $S_\can$, 
la distance $\dist$ induit une distance sur la sph\`ere \`a l'infini $\P^1(\C_p)$,
que l'on d\'efinit de la mani\`ere suivante. Soit~$z$ et $z'$ deux points
distincts de $\P^1(\C_p)$. Il existe un unique point $\langle z, z'\rangle$ de $\P^1(\C_p)$
pour lequel les demi-g\'eod\'esiques $[S_\can,z\mathclose[$ et $[S_\can, z'\mathclose[$ s'intersectent
exactement le long du segment $[S_\can, \langle z, z'\rangle]$. 
La distance visuelle $\delta(z,z')$
entre~$z$ et $z'$ est alors d\'efinie par 
\[
\delta(z,z')= p^{-\dist(S_\can, \langle z, z'\rangle)}.
\]
Cette distance co\"{\i}ncide avec la 
\emph{distance sph\'erique}\footnote{cette distance est appel\'ee {\og chordal distance \fg}   en anglais, 
qu'on traduit ici par  {\og distance sph\'erique\fg}; certains auteurs parlent 
de distance harmonique.} 
sur $\P^1(\C_p)$ (voir~\cite{Morton-Silverman:Crelle}, p. 103), d\'efinie par
\[ \delta(z,z')= \frac{\abs{z-z'}}{\max(1, \abs z) \max(1, \abs{z'})}=\begin{cases}
\abs{z-z'} & \text{si $\abs z$ et $\abs{z'}\leq 1$,} \\
\abs{\frac 1z-\frac 1{z'}} & \text{si $\abs z>1$ et $\abs{z'}>1$,} \\
1 & \text{sinon.}\end{cases}\]
Cette distance est invariante par le stabilisateur $\PSL_2(\mathscr O)$
de la coupure canonique~$S_\can$ dans~$\Mob$.

\medskip

Posons 
\[
\BP(\C_p)=\P^1(\C_p)\cup\Hp\; ;
\]
c'est encore un arbre r\'eel
mais pour la m\'etrique~$L$ telle que 
\[
\left.\frac{\mathrm dL}{\mathrm d\ell}\right|_{S}=p^{-\dist(S_\can,S)},  
\]
o\`u~$\ell$ est la longueur d'arc
de l'espace hyperbolique~$\Hp$.\footnote{Comme $x\mapsto p^{-x}$ est d'int\'egrale $1/\log(p)$ sur
$\R_+$, ce  changement  de distance ram\`ene 
\`a distance finie les points de la sph\`ere \`a l'infini.}
%% acl - legere modif de la note
Autrement dit, le long d'un arc g\'eod\'esique $s\mapsto \gamma(s)$, $s\in\R^+$, 
partant de $\gamma(0)= S_\can$, le rapport entre $L(\gamma(t), \gamma(t'))$ et $\dist(\gamma(t),\gamma(t'))$ est \'egal \`a l'int\'egrale de $p^{-u}du$ entre $t$ et $t'$. 

L'arbre $\BP(\C_p)$ est ainsi la r\'eunion de quatre parties disjointes,
\[
\BP(\C_p)=\P^1(\C_p)\cup \Hpsing\cup\Hp^{\R\setminus \Q}\cup\HpQ,
\]
qui satisfont les propri\'et\'es  suivantes:
\begin{itemize}
\item les \'el\'ements de $\P^1(\C_p)\cup\Hpsing$ sont les bouts de l'arbre~$\HpR$: 
de chaque point de $\P^1(\C_p)\cup \Hpsing$ ne part qu'une branche; 
\item d'un point de $\Hp^{\R\setminus \Q}$ partent exactement deux branches;
\item l'ensemble des branches issues d'un point de $\HpQ$ est
un espace homog\`ene sous un conjugu\'e de $\PGL_2(\mathscr O)$,
qui est isomorphe \`a l'espace homog\`ene~$\P^1(\overline{\F_p})$
muni de l'action de~$\PGL_2(\mathscr O)$.
%% acl - ajout
\end{itemize}

\subsection{Autres d\'efinitions}

\subsubsection{} 
On peut donner une d\'efinition plus directe de~$\BP(\C_p)$ \`a partir 
de l'ensemble des boules circonf\'erenci\'ees
de~$\C_p$ de rayon strictement positif.
Cet ensemble s'identifie \`a l'espace quotient
du produit $\C_p\times \R$ 
par la relation  d'\'equivalence
\[
(z,t)\sim (z',t') \quad \Leftrightarrow \quad  t= t' \, {\text{ et }}\, \log_p \abs{z-z'}\leq t,
\]
la boule~$\Bf(z,p^t)$ \'etant identifi\'ee \`a la classe du point~$(z,t)$.
La valeur absolue~$p$-adique \'etant ultram\'etrique, ceci d\'efinit une relation d'\'equivalence.
L'espace quotient $(\C_p\times \R)/\mathord\sim$ est en bijection avec $\HpR$ : pour cela il 
suffit d'associer au point~$(z,t)$ la coupure $S(z,p^t)$ associ\'ee \`a la boule $\Bf (z,p^t)$.
%Serge t ---> p^t dans la derni\`ere boule

La distance usuelle sur~$\R$ induit une distance sur l'espace quotient $(\C_p\times \R)/\mathord\sim$ ; si $(z,t)$ et $(z',t')$ appartiennent \`a~$\C_p \times \R$, la distance entre les
points qu'ils d\'eterminent dans $(\C_p\times \R)/\mathord\sim$ est
\[
2 \max(t, t', \log_p\abs{z-z'}) - t - t'. 
\]
Si l'on munit $(\C_p\times \R)/\mathord\sim$ de cette distance, la bijection $(z,t) \mapsto S(z,p^t)$
devient une isom\'etrie entre $(\C_p\times \R)/\mathord\sim$ et $\HpR$. Pour retrouver 
$\BP(\C_p)$, on ajoute alors les bouts de l'arbre $\HpR$ (c'est-\`a-dire $\Hpsing$
et $\P^1(\C_p)$).

\subsubsection{} Une troisi\`eme d\'efinition de $\BP(\C_p)$ 
peut \^etre obtenue en consid\'erant l'espace analytique associ\'e par
V.~\textsc{Berkovich} dans~\cite{berkovich1990}
\`a la droite projective sur~$\C_p$ 
(voir aussi l'annexe de~\cite{Rivera-Letelier:compositio} 
et~\cite{Ducros:Bourbaki}). 
Dans le cadre qui nous occupe, l'int\'er\^et majeur de la d\'efinition de $\BP(\C_p)$
par Berkovich r\'eside dans le fait que toute fraction rationnelle 
op\`ere naturellement
sur $\BP(\C_p)$, ce qui para\^{\i}tra  un peu miraculeux du point
de vue des coupures.

\medskip

L'espace analytique associ\'e \`a~$\P^1(\C_p)$ au sens de Berkovich est d\'efini
comme suit.
Une \emph{semi-norme multiplicative} sur l'anneau $\C_p[z]$ 
est une application $s:\C_p[z]\to \R_+$ satisfaisant 
\begin{enumerate}\def\theenumi{\roman{enumi}}\def\labelenumi{(\theenumi)}
\item $s(1)=1$, 
\item $s(P+Q) \leq s(P) + s(Q)$ et
\item $s(PQ)= s(P) s(Q)$
\end{enumerate}
pour tout couple $(P,Q)$ d'\'el\'ements de $\C_p[z]$.
Par exemple, \`a tout point~$z$ de la droite affine $\C_p$ est associ\'ee une semi-norme 
$s_z$ d\'efinie par $s_z(P)=\abs{P(z)}$.

Notons $\AB$ l'ensemble des semi-normes multiplicatives de $\C_p[z]$ 
dont la restriction \`a~$\C_p$ co\"{\i}ncide avec la valeur absolue~$p$-adique
(la notation de Berkovich est $\mathscr M(\C_p[z])$). 
Du point de vue de Berkovich, $\AB$ est la droite affine sur~$\C_p$.
Pour obtenir la droite projective, il suffit de lui adjoindre
un point \`a l'infini~$s_\infty$ qu'on peut interpr\'eter comme la
fonction d\'efinie sur $\C_p[z]$ par $s_\infty(P)=\infty$ lorsque~$P$ n'est pas constant 
et $s_\infty(P)=\abs a$ si~$P$ est la constante~$a$
(en d'autres termes, $s_\infty(P)$ est la valeur absolue de $P(\infty)$).  
On pose alors $\P^1_\Ber(\C_p)=\AB \cup \{S_\infty \}$. 

Le caract\`ere ultram\'etrique du corps~$\C_p$ se r\'ev\`ele dans le lemme suivant
qui montre que l'espace $\P^1_\Ber(\C_p)$ poss\`ede, outre $\P^1(\C_p)$,
de nombreux autres points.
\begin{lemm}\label{lemm.gauss}
Soit~$B$ une boule circonf\'erenci\'ee de la droite affine. Alors,
la fonction $s_B\colon\C_p[z]\ra\R_+$ d\'efinie par
\[
s_B(P) = \sup_{z\in B}\abs{P(z)}
\]
est une semi-norme (en fait, une norme) multiplicative sur l'anneau~$\C_p[z]$.
% Antoine: sur l'anneau...
\end{lemm}

Lorsque $B=\Bf(0,1)$, la norme~$s_B$ est appel\'ee \emph{norme de Gauss}.
D'apr\`es ce lemme, toute boule circonf\'erenci\'ee~$B$ de~$\C_p$
fournit un point de~$\P^1_\Ber(\C_p)$.
Lorsque $B= \Bf(a,r)$, la valeur de $s_B$ au polyn\^ome $z-b$
est \'egale \`a  
\[
s_B(z-b)=\min(r, \abs{b-a}).
\]
On voit ainsi que la donn\'ee de  la semi-norme~$s_B$ redonne~$B$.
D'apr\`es le dictionnaire
entre coupures et boules ferm\'ees ou irrationnelles,
cela fournit une injection de $\HpR$ dans~$\P^1_\Ber(\C_p)$.

Plus g\'en\'eralement, 
pour toute suite d\'ecroissante $\mathscr B=(B_i)$ de boules circonf\'erenci\'ees,
la limite d\'ecroissante $s_{\mathscr B}=\lim s_{B_i}$ est encore
une semi-norme multiplicative sur~$\C_p[z]$.
L'intersection des boules de~$\mathscr B$ peut \^etre
un point~$z$ de~$\P^1(\C_p)$, auquel cas $s_{\mathscr B}=s_z$;
elle peut \^etre une boule~$B$,  n\'ecessairement 
ferm\'ee ou irrationnelle (auquel cas $s_{\mathscr B}=s_B$)
et les semi-normes obtenues correspondant aux coupures rationnelles
et irrationnelles de~$\BP(\C_p)$;
elle peut aussi \^etre vide, dans ce cas la norme~$s_{\mathscr B}$
correspond \`a un point singulier de~$\BP(\C_p)$.

Inversement, on d\'emontre que $\P^1_\Ber(\C_p)$ n'a pas d'autre
point que ceux-ci, ce qui
permet d'identifier l'espace $\P^1_\Ber(\C_p)$ 
avec l'espace $\BP(\C_p)$ d\'efini dans les paragraphes pr\'ec\'edents. 

D\'emontrons maintenant le lemme~\ref{lemm.gauss}.
\begin{proof}
Supposons que le rayon de la boule $B$ appartient 
\`a~$p^{\Q}$; quitte \`a faire un changement de variable affine,
on suppose m\^eme $B=\Bf(0,1)$. 
Pour $P=a_nX^n+\dots+a_0$, posons $\norm P= \max(\abs{a_0},\dots,\abs{a_n})$
et d\'emontrons que l'on a $s_B(P)=\norm P$ pour tout $P\in\C_p[z]$.
Tout d'abord, on observe que pour tout $z\in\Bf(0,1)$,
\[ \abs{P(z)}\leq \max(\abs{a_n}\abs{z}^n) \leq \norm P.\]
Inversement, d\'emontrons qu'il existe $z\in\Bf(0,1)$ tel que
$\abs{P(z)}= \norm P$. On peut supposer que $P\neq 0$.
Quitte \`a diviser~$P$ par un de ses coefficients
de valeur absolue maximale, on a $\norm P=1$ ; on peut ainsi
supposer que les $a_i$
appartiennent \`a~$\mathscr O$ et que l'un d'entre eux au moins
n'est pas dans l'id\'eal maximal~$\mathfrak m$, de sorte
que le polyn\^ome r\'eduit~$\overline P$ n'est pas nul.\label{page.gauss}
Soit $\alpha$
un \'el\'ement de~$\overline{\F_p}$ qui n'est pas une racine
de~$\overline P$;
pour tout $z\in B(\alpha)$, on a $P(z)\in\mathscr O$
et $\overline{P(z)}=\overline P(\alpha)\neq 0$,
autrement dit $\abs{P(z)}=1$. 

Nous pouvons maintenant terminer la d\'emonstration du lemme
(dans le cas o\`u $B=\Bf(0,1)$).
Il est clair que $s_B$ est une norme sur~$\C_p[z]$,
le point non \'evident est sa multiplicativit\'e.
Soient $P$ et $Q$ des polyn\^omes de~$\C_p[z]$ ; par
homog\'en\'eit\'e, pour
d\'emontrer que $s_B(PQ)=s_B(P)s_B(Q)$, on peut supposer
que $P$ et $Q$ v\'erifient $\norm P=\norm Q=1$ et il s'agit
de d\'emontrer que $\norm{PQ}=1$. L'in\'egalit\'e ultram\'etrique
entra\^{\i}ne imm\'ediatement que $\norm{PQ}\leq 1$ et il s'agit maintenant
de prouver l'\'egalit\'e. Les polyn\^omes $\overline P$
et $\overline Q$ \`a coefficients dans~$\overline{\F_p}$
sont non nuls, et il en est alors de m\^eme de leur produit.
Or, ce produit $\overline P\overline Q$ n'est autre
que la r\'eduction modulo~$\mathfrak m$ du polyn\^ome $PQ$. On
a donc $\overline{PQ}\neq 0$, ce qui entra\^{\i}ne $\norm{PQ}=1$,
d'o\`u le lemme, lorsque le rayon de $B$ appartient \`a~$p^\Q$. 

Le cas g\'en\'eral s'en d\'eduit par passage \`a la limite. \end{proof}

La structure d'arbre r\'eel de $\BP(\C_p)$, 
qui dans les deux d\'efinitions pr\'ec\'edentes r\'esulte des 
propri\'et\'es d'inclusions des boules, 
se traduit ici dans la relation d'ordre partiel $\leq$ entre semi-normes : 
$s\leq s'$ si $s(P)\leq s'(P)$ pour tout polyn\^ome~$P$. Par exemple,
l'arc topologique joignant deux points $s\leq s'$ est constitu\'e des semi-normes $s''$ telles que $s\leq s''\leq s'$, et les \'el\'ements minimaux co\"{\i}ncident avec les
bouts de l'arbre qui sont distincts de $\infty$.

Soit~$s$ un \'el\'ement de $\P^1_\Ber(\C_p)\setminus \P^1(\C_p)$. La semi-norme
$s$ ne s'annule donc sur aucun polyn\^ome non nul, ce qui permet de l'\'etendre en 
une valeur absolue sur le corps $\C_p(z)$. Si~$R$ est une fraction rationnelle, 
on peut alors d\'efinir une nouvelle norme
\[
R_*s : Q\mapsto s(Q\circ R).
\]
Ceci d\'efinit une action naturelle de~$R$ sur $\BP(\C_p)\setminus \P^1(\C_p)$. Nous 
d\'efinirons \`a nouveau cette action en termes de coupures et de boules dans la 
partie suivante de ce texte.

Nous renvoyons le lecteur \`a \cite{Favre-Rivera-Letelier:2008} 
et~\cite{Ducros:Bourbaki} pour plus de d\'etails et des r\'ef\'erences 
concernant l'utilisation du point de vue de Berkovich en dynamique~$p$-adique.

\subsubsection{} 
Il existe bien  d'autres fa\c{c}ons d'aborder cet espace $\BP(\C_p)$,
pour lesquelles nous renvoyons aux articles~\cite{Baker:2008} et~\cite{Rivera-Letelier:CMH}
et aux r\'ef\'erences qui s'y trouvent. L'espace $\BP(\C_p)$ est notamment reli\'e \`a l'arbre
de Bruhat--Tits de~$\SL_2(\C_p)$.

%Soit $a\in\P^1(\C_p)$ et soit~$r$ un nombre r\'eel tel que $0\leq r\leq 1$;
%notons~$B$ la boule~$\{\delta(a,z)\leq r\}$ pour la distance
%chordale ; de m\^eme, soit $B'=\{\delta(a',z)\leq r'\}$ une boule
%de centre~$a'$ et de rayon~$r'$.
%Soit $B''$ la plus petite boule qui contient~$B$ et~$B''$
%et soit $r''$ son rayon. Alors, $\delta(B,B')=r''-\frac12(r+r')$.
%Si $B\subset B'$, on a donc $\delta(B,B')=\frac12(r'-r)$;
%sinon, on a $\delta(B,B')=\delta(B,B'')+\delta(B'',B')$.

%La compl\'etion  de $\P^1(\C_p)$
%correspond pr\'ecis\'ement \`a~$\BP(\C_p)$. (EST CE BIEN CA ????)
%N'EST CE PAS EN CONTRADICTION AVEC LE FAIT QUE R EST 
%1-LIPSCHITZ LORSQUE R A BONNE REDUCTION ????????

%%
\subsection{Topologie faible sur $\BP(\C_p)$}\label{par:affinoide}

Une partie~$A$ de~$\P^1(\C_p)$ est appel\'ee \emph{affino\"{\i}de ouvert}
si c'est une intersection finie non vide de boules ouvertes (les boules
peuvent contenir l'infini).\footnote{%
Nous suivons ici la terminologie de Rivera-Letelier, Yoccoz, etc.,
encore qu'il faudrait ajouter le qualificatif \emph{connexe}.
En g\'eom\'etrie analytique rigide, un ensemble affino\"{\i}de
connexe de~$\P^1(\C_p)$ est une intersection finie 
non vide de disques ferm\'es, \emph{cf.}~\cite{fresnel-vanderput2004}.
Dans la th\'eorie de Berkovich~\cite{berkovich1990}, 
les espaces analytiques sont localement
connexes par arcs, localement compacts et les domaines affino\"\i des 
en sont des parties compactes; les affino\"{\i}des ouverts que
nous venons d'introduire sont la trace sur~$\P^1(\C_p)$ 
de certaines parties ouvertes et connexes de~$\BP(\C_p)$.}

\begin{figure}[htbp]
% \centering\epsfig{figure=affinoide.pstex}
\centering\includegraphics{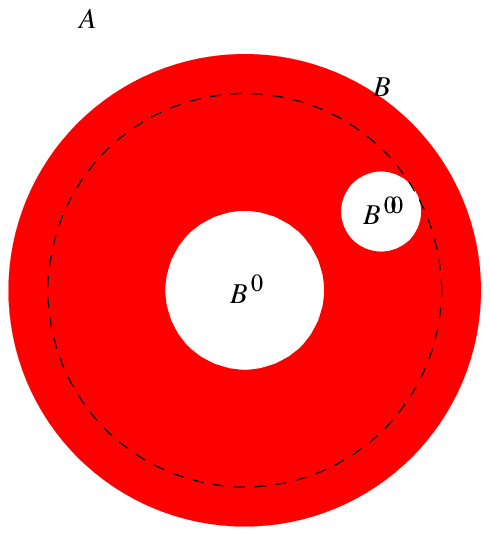}
\includegraphics{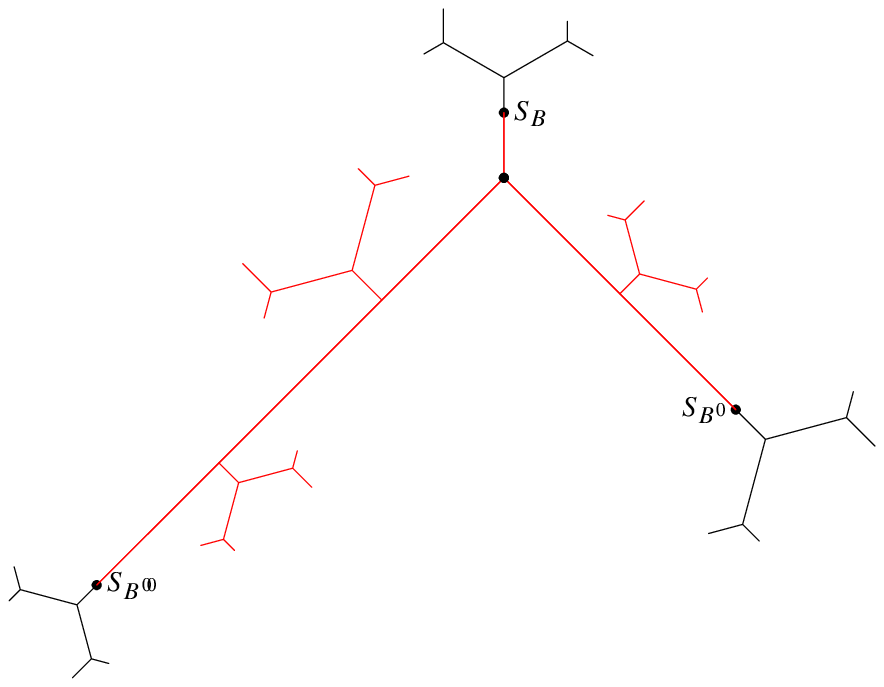}
\caption{\textsc{Affino\"{\i}des ouverts}: 
Ici, $A$ est l'affino\"{\i}de obtenu en retranchant les
boules ferm\'ees $B'$ et $B''$ \`a la boule ouverte~$B$. 
Les coupures~$S_B$, $S_{B'}$ et
$S_{B''}$ sont d\'etermin\'ees par les boules correspondantes.} 
\label{fig:affinoide}
\end{figure}

Soit~$A$ un affino\"{\i}de ouvert. 
On dit qu'une coupure~$S$ \emph{s\'epare}~$A$ si $A$ rencontre
au moins deux boules de~$S$. Par exemple, si $A$ est une boule
ouverte et si $S$ est la coupure d\'etermin\'ee par la boule $\Bf (a,r)$,
alors $S$ s\'epare~$A$ si et seulement si $\Bf (a,r)$ est contenue dans~$A$. 

Notons $\widehat A$ la partie de~$\BP(\C_p)$
r\'eunion de~$A$, de l'ensemble des coupures qui s\'eparent~$A$,
et des points singuliers de~$\BP(\C_p)$ qui correspondent
\`a une intersection vide de boules contenues dans~$A$.
C'est une partie connexe, donc un arbre r\'eel.
\'Ecrivons~$A$ comme l'intersection $\bigcap B_i$ d'un nombre minimal 
de boules ouvertes ; cette \'ecriture est unique \`a permutation pr\`es des~$B_i$. 
Pour tout~$i$, soit $S_i$ l'unique coupure qui contient~$B_i$ ;
ces coupures forment les \emph{extr\'emit\'es} de $\widehat A$ 
(voir la figure~\ref{fig:affinoide}).

La \emph{topologie faible} sur~$\BP(\C_p)$ est d\'efinie
en prenant comme base d'ouverts l'ensemble des parties~$\widehat A$,
lorsque $A$ parcourt l'ensemble des affino\"{\i}des ouverts de~$\P^1(\C_p)$.
Cette topologie est m\'etrisable et fait de $\BP(\C_p)$ un espace
compact et contractile.
De plus, l'ensemble~$\P^1(\C_p)$ est dense pour la topologie
faible qui induit d'ailleurs la topologie $p$-adique sur cet ensemble.
Restreinte au compl\'ementaire du point \`a l'infini, la topologie
faible est aussi la plus grossi\`ere des topologies pour lesquelles
les applications $P\mapsto s(P)$ sont continues, pour tout~$s\in\AB$.

%Le lien avec la topologie d'arbre r\'eel est assez facile
%\`a d\'ecrire. Soit $S$ un point de~$\BP(\C_p)$ ; un voisinage
%de~$S$ pour la topologie faible est la r\'eunion de~$S$
%et, pour toute branche~$\mathscr B$ issue de~$S$,
%d'un voisinage~$V_{\mathscr B}$ de~$S$ dans l'arbre $\{S\}\cup\mathscr B$,
%ces voisinages \'etant tous \'egaux \`a~$\{S\}\cup\mathscr B$,
%sauf pour un nombre fini de branches~$\mathscr B$.

%%%%%%%%%%%%%%%%%%%%%%%%%
%%%%%%%%%%%%%%%%%%%%%%%%%%%%%%%%%%%%%%%

\section{Action des fractions rationnelles sur l'espace hyperbolique~$p$-adique}

%%%%%%%%%%%%%%%%%%%%%%%%%%%
%%%%%%%%%%%%%%%%%%%%%%%%%%%%%%%%%%%%%%%

Si~$R$ appartient \`a~$\C_p(z)$, l'action de~$R$ sur la droite projective
$\P^1(\C_p)$ s'\'etend \`a~$\BP(\C_p)$ ; c'est ce qu'il s'agit de d\'ecrire maintenant. 
Lorsque $R$ est constante, son extension  \`a  $\BP(\C_p)$ sera \'egalement constante. Nous supposerons donc syst\'ematiquement que $R$ n'est pas constante.

\subsection{Image des petites couronnes}

Soit $B$ une boule ouverte ou irrationnelle et soit $(B_i)$
une suite strictement croissante de boules ferm\'ees ou irrationnelles dont la 
r\'eunion est \'egale \`a~$B$.
Pour tout~$i$, soit $C_i$ la couronne $B\setminus B_i$ ;
son module tend vers~$0$ et l'intersection des~$C_i$ est vide.

\begin{lemm}[\cite{Rivera-Letelier:compositio}, prop.~4.1]
Soit $R$ une fraction rationnelle non constante.
Il existe un entier naturel~$d$, une boule~$B'$ et, 
pour tout entier~$i$ assez grand, une boule~$B'_i$ ferm\'ee
ou irrationnelle contenue dans~$B'$ tels que
\begin{itemize}
\item l'image~$R(C_i)$ de la couronne~$C_i$ est la couronne 
$B'\setminus B'_i$;
\item l'application $R:C_i\to R(C_i)$
induite par~$R$ est de degr\'e~$d$ (tout point de~$R(C_i)$ a
$d$ ant\'ec\'edents dans~$C_i$ par~$R$).
\end{itemize}
De plus, le module de la couronne $R(C_i)$ 
est \'egal \`a~$d$ fois celui de~$C_i$.
\end{lemm}
C'est une application assez simple de la th\'eorie du polygone
de Newton (voir par exemple \cite{amice75,dwork-g-s94}
ainsi que \cite{Rivera-Letelier:CMH}, \S1.3.2). 
En outre, la boule~$B'$ ne dépend que de la boule~$B$ et
pas de la suite~$(B_i)$ choisie.

%Serge Ajout ci-dessous
Le module de la couronne $C_i$ est \'egal \`a la distance entre les
coupures d\'etermin\'ees par les boules (ferm\'ees ou irrationnelles) $B_i$.
La derni\`ere assertion du lemme correspond donc \`a une propri\'et\'e
de dilatation de la distance entre coupures ; nous allons pr\'eciser ce point
dans les paragraphes qui suivent.

%%%%%
\subsection{Action sur les coupures; action résiduelle sur les branches}
\label{subsec.action-coupures}
%%%%

Soit~$S$ une coupure ; soit~$B$ 
l'une des boules ouvertes ou irrationnelles constituant la partition~$S$.
Notons $\mathcal{B}$ la branche issue de~$S$ qui est d\'etermin\'ee par~$B$
(voir le \S \ref{par:branches}).
Choisissons une suite  croissante $(B_i)$
de boules ferm\'ees ou irrationnelles 
dont la r\'eunion est \'egale \`a~$B$. 
Pour tout~$i$, la boule~$B_i$ 
d\'etermine une coupure~$S_i$ qui appartient \`a~$\mathscr B$
et $S_i$ tend vers~$S$ quand $i\rightarrow\infty$.
Pour tout~$i$, posons $C_i=B\setminus B_i$. 

Puisque la suite de boules $B_i$ est croissante, la suite des coupures
$S_i$ se d\'eplace le long de l'arc g\'eod\'esique reliant $S_0$ \`a $S$ 
dans l'arbre $\Hp$.   

Pour $i$ assez grand, l'ensemble $R(C_i)$ est une couronne
d'apr\`es le lemme pr\'ec\'edent ; on peut donc l'\'ecrire
sous la forme $B'\setminus B'_i$, o\`u $B'$ et $B'_i$
sont des boules de m\^eme nature que~$B$ et $B_i$.
Ces boules d\'efinissent des coupures $S'$ et $S'_i$
telles que la suite $S'_i$ tende vers~$S'$, car le module de $R(C_i)$
tend vers~$0$.
De plus, les coupures~$S'_i$
appartiennent \`a une m\^eme branche~$\mathscr B'$ issue de~$S'$.

On d\'emontre que la coupure~$S'$ ne d\'epend que de la coupure~$S$;
on la  note $R_*(S)$ : c'est l'image de~$S$ par~$R$. 
De même, la branche~$\mathcal B'$ issue de~$S'$,
ainsi que l'entier~$d$ dont le lemme pr\'ec\'edent
affirme l'existence ne d\'ependent que de la branche~$\mathcal B$ mais
pas du choix de la suite de boules~$(B_i)$, on le
note $\deg_{\mathcal B}(R)$. 

Lorsque $S_i$ tend vers $S$ le long d'un segment g\'eod\'esique contenu 
dans la branche~$\mathcal B$, la distance entre $R_*(S_i)$ et $R_*(S)$ est 
\'egale \`a $\deg_{\mathcal B}(R) \dist(S_i, S)$. Autrement dit, pour chaque 
demi-g\'eod\'esique $\gamma:\R_+\to \Hp$ issue de $S=\gamma(0)$ et trac\'ee dans $\mathcal B$,
il existe $t>0$ tel que 
\[
\dist(R_*(\gamma(s)),R_*(\gamma(s'))) =\deg_{\mathcal B}(R) \dist(\gamma(s),\gamma(s'))=\deg_{\mathcal B}(R) \abs{s-s'}
\]
pour tout $(s,s')\in [0,t]^2$. La longueur $t$ sur laquelle cette dilatation a lieu d\'epend
du choix de la g\'eod\'esique $\gamma$ issue de $S$.

\subsubsection*{Cas des coupures irrationnelles}
La proposition suivante pr\'ecise ce que peut \^etre le degr\'e d'une fraction
rationnelle le long d'une  branche issue d'une coupure irrationnelle.

\begin{prop}\label{pro:indifferent} 
Soit $R$ un \'el\'ement non constant de~$\C_p(z)$. 
Si $S$ est une coupure irrationnelle, alors 
\begin{enumerate} 
\item $\deg_{\mathcal B} (R)$ ne d\'epend pas du choix de la branche ${\mathcal B}$ issue de~$S$ et sera not\'e $\deg_S (R)$ ;
\item  $\deg_{\mathcal B} (R)=1$ d\`es que $R_*(S)=S$. 
\end{enumerate}
\end{prop}

Si $S$ est une coupure irrationnelle fix\'ee par~$R$, 
on a $\deg_{\mathcal B} (R)=1$ et $R$
pr\'eserve les longueurs le long des deux branches issues de $S$ (sur un voisinage du point $S$). 
On parle de {\emph{point fixe indiff\'erent}} (voir le paragraphe~\ref{par:points_fixes}). 

\begin{proof}[Esquisse de d\'emonstration]
On peut supposer que~$S$ est la coupure~$S(r)$, o\`u~$r$ est un nombre
r\'eel strictement positif n'appartenant pas \`a~$p^\Q$. Comme $r\not\in p^\Q$,
il existe deux nombres r\'eels positifs $r_1$ et $r_2$ tels que $r_1< r < r_2$ 
et~$R$ n'ait pas de p\^ole dans la couronne $\{ r_1<\abs z< r_2\}$.
Notons 
\[
R(z)=\sum_{n\in\Z} h_n z^n
\]
le d\'eveloppement de~$R$ en s\'erie de Laurent sur cette couronne. 
On a donc 
\[
\lim_{\abs n\ra\infty} \abs{h_n}\abs{z}^n=0
\]
d\`es que $r_1< \abs{z} < r_2$.
Comme $r\not\in p^\Q$, il existe alors un unique entier~$m$
tel que 
\[\abs{h_m}r^m=\max(\abs{h_n}r^n)\ ;
\]
de plus,
quitte \`a diminuer~$r_2$ et \`a augmenter~$r_1$,
on peut m\^eme supposer que
\[ \abs {h_m} \abs{z}^m = \max( \abs{h_n}\abs z^n)\quad
\text{pour tout $z$ tel que $r_1<\abs z<r_2$.} \]
Quitte \`a composer 
$R$ par $z\mapsto z-h_0$
au but on peut supposer
que $h_0$ est nul; on a donc $m\neq 0$. Quitte \`a composer~$R$
 par $z\mapsto 1/z$ \`a la source, on peut aussi supposer
que $m$ est strictement positif. 

Pour tout $\rho$ tel que $r_1<\rho<r_2$, on a alors
$R_*(S(\rho))=S(\abs{h_m}\rho^m)$, d'o\`u $m=\deg_{\mathcal B}(R)$
pour chacune des branches~$\mathcal B$ issues de~$S$.

Enfin, si $R_*(S(r))=S(r)$, alors $r^{m-1}=1/\abs{h_m}$,
d'o\`u $m=1$ car $r\not\in p^\Q$. Le degr\'e local en une coupure fixe 
irrationnelle est donc \'egal \`a~$1$. 
\end{proof}

\subsubsection*{Cas des coupures rationnelles}

Soit $S$ une coupure rationnelle. La fraction
rationnelle~$R$ induit une application de l'ensemble
des branches issues de~$S$ sur l'ensemble des branches issues
de~$R_*(S)$ que nous voulons décrire maintenant.
Rappelons que la partition~$S$ ou, de mani\`ere \'equivalente, 
l'ensemble des branches de $\Hp$ issues de~$S$,
est une droite projective, isomorphe \`a~$\P^1(\bar{\F}_p)$ 
(voir le \S \ref{par:dedekind}). 
La proposition suivante utilise cette param\'etrisation
des branches par~$\P^1(\bar{\F}_p)$.
\begin{prop} 
Soit~$R$ un \'el\'ement non constant de $\C_p(z)$.
Soit~$S$  une coupure rationnelle. 
L'application induite par~$R$ entre l'ensemble des branches issues
de~$S$ et celui des branches issues de $R_*(S)$ 
est une fraction rationnelle~$\bar R$ ;
si $\mathscr B$ est une branche issue
de~$S$ correspondant à un élément~$\xi\in\P^1(\bar\F_p)$,
on a $\deg_{\mathcal B}(R_*)=\deg_\xi(\bar R)$.
\end{prop}

En particulier, le degré de~$\bar R$ ne dépend pas du
choix de paramétrage des branches et vérifie
\[ \deg(\bar R) \leq \deg (R);\]
on le notera $\deg_S(R)$.

L'application $\bar{R}$ s'obtient par r\'eduction de~$R$ en utilisant des
coordonn\'ees dans lesquelles~$S$ et $R_*(S)$ co\"{\i}ncident avec la
coupure canonique. 

Lorsque $R_*(S)=S$ et $\deg_{S} (R)>1$, 
$R$ dilate strictement les longueurs au voisinage de~$S$ le long
de chaque g\'eod\'esique partant de $S$ le long
d'une branche~$\mathscr B$ telle que $\deg_{\mathbf B}(R)>1$. 
On parlera donc de {\emph{point fixe r\'epulsif}}  
(voir le paragraphe~\ref{par:points_fixes}). 
Ces propri\'et\'es de dilatation peuvent 
\^etre pr\'ecis\'ees pour montrer que de  {\emph{l'action de~$R_*$ 
sur l'arbre~$(\HpR,\dist)$ dilate les distances 
par un facteur \og entier par morceaux \fg}}. Afin de d\'ecrire
cet \'enonc\'e, fixons un segment g\'eod\'esique $\gamma(t)=S_t$, $t\in I$, le long duquel
$R_*$ est injective. 
Pour chaque $t$ dans $I$, notons $\mathcal{B}_t$ la branche 
issue de $\gamma(t)=S_t$ dans laquelle la g\'eod\'esique~$\gamma$ 
entre \`a l'instant~$t$. Soient $t< t'$ deux \'el\'ements de $I$. 
Alors, la fonction $s\mapsto \deg_{\mathcal{B}_s}(R)$ est
une fonction en escalier et ne prend qu'un 
nombre fini de valeurs entre $t$ et $t'$; de plus,
\[
\dist(R_*(S_t),\R_*(S_t')) = \int_t^{t'} \deg_{\mathcal{B}_s}(R) ds.
\]
Le lecteur pourra consulter~\cite[cor.~4.8]{Rivera-Letelier:compositio} pour une preuve de cette formule.

\subsubsection*{Points singuliers}

On peut prolonger l'action de~$R_*$
aux bouts de~$\HpR$, de sorte que chaque point singulier~$S$ de 
$\Hp$ soit envoy\'e par~$R$ sur un point singulier $R_*(S)$.  
Ceci  {\emph{\'etend l'action de
$R$ sur $\P^1(\C_p)$ \`a~$\Hp$, et donc \`a~$\BP(\C_p)$ en une
application lipschitzienne}}
\[
R_* \colon \BP(\C_p)\to \BP(\C_p),
\]
de l'arbre~$\BP(\C_p)$ dans lui-m\^eme ; plus pr\'ecis\'ement, $R_*$ 
est lipschitzienne en restriction \`a~$\HpR$ et \`a~$\Hp$ pour la distance $\dist$,
$R_*$ est lipschitzienne en restriction \`a $\P^1(\C_p)$ pour la distance
sph\'erique $\delta$ et, enfin, $R_*$ est lipschitzienne sur l'arbre $\BP(\C_p)$
pour la distance $L$ introduite au paragraphe \ref{par:sphereinfinie}.

Nous renvoyons le lecteur int\'eress\'e aux articles~\cite{Rivera-Letelier:compositio,
Rivera-Letelier:CMH} pour plus de d\'etails. Citons simplement le fait suivant: {\emph{si 
$S$ est un point singulier qui est fix\'e par $R_*$, alors $S$ est un point fixe indiff\'erent.}}
Ainsi, dans $\HpR$, 
seules les coupures rationnelles peuvent fournir des points fixes 
(ou p\'eriodiques) r\'epulsifs. 
Nous emploierons ce fait \`a plusieurs reprises dans la partie suivante. 
Terminons celle-ci par un exemple.

%%%
\subsection{Le cas des fractions rationnelles de degr\'e~$2$}
%%%

Sur un corps alg\'ebriquement clos de caract\'eristique diff\'erente de $2$, il 
n'existe qu'une seule fraction rationnelle de  degr\'e~$2$
modulo composition
pr\`es \`a droite et \`a gauche par des
transformations de M\"obius, c'est la fraction rationnelle $z\mapsto z^2$.

Par contre, en
caract\'eristique~2, il existe deux mod\`eles qui
sont les applications $Q\colon z \mapsto z^2$ et $R\colon z \mapsto z + z^2$.
L'application quadratique $Q(z) = z^2$
commute avec les automorphismes de M\"obius de
${\mathbf{P}} ({\mathbf{F}}_2)$, donc
tous les \'el\'ements de ${\mathbf{P}} ( {\mathbf{F}}_2)$ jouent le m\^eme r\^ole pour
cette application. 
(Plus g\'en\'eralement, si $k$ est un entier strictement positif
et $M$ un automorphisme de M\"obius \`a coefficients
dans le corps~$\F_{2^k}$, l'it\'er\'e $k$-fois de~$Q$ commute avec~$M$.
Du point de vue de la dynamique, tous les \'el\'ements de~$\mathbf P(\overline{\mathbf F_2})$ jouent ainsi essentiellement le m\^eme r\^ole.)
En revanche, le point~$\infty$
joue un r\^ole particulier pour l'application~$R$:
c'est  le seul point en lequel le degr\'e local est 2, tandis que
 que tous les autres points sont g\'eom\'etriquement \'equivalents
 car $R(z+a)=R(z)+R(a)$.

\begin{ex}
D\'emontrer ces faits.
\end{ex}

Soit~$Q$ une fraction rationnelle de degr\'e~$2$ \`a coefficients
dans~$\C_p$. D'apr\`es ce qu'on vient de dire,
on peut supposer que $Q(z)=z^2$.

\subsubsection*{Supposons d'abord $p\neq 2$.}

Notons $\gamma$ la g\'eod\'esique de $\BP(\C_p)$ joignant~$0$ \`a~$\infty$ ; elle est
constitu\'ee des coupures~$S (r)$, le nombre r\'eel~$r$ d\'ecrivant~$\R_+$. Puisque $Q(\Bf (r))=\Bf (r^2)$, on obtient  $ Q_* (S (r)) = S (r^2)$ et $ \deg_{S(r)} (Q) = 2$.
Comme
\[
 {\dist}( S (r), S (r')) = \abs{ \log_p r - \log_p r' }  \; ,
\]
on observe que les distances sont multipli\'ees par~$2$ le long de $\gamma$.

Pour $a \in {\mathbf{C}}_p^*$ et $0 < r < \abs a$, 
la coupure  $S (a, r)$ (d\'etermin\'ee par la boule $\Bf (a,r)$)
satisfait
$\dist(S (a, r), \gamma) = \log_p \abs a - \log_p r $. 
Pour tout $w\in \C_p$ satisfaisant $\abs{w}\leq r$,  on a 
\[
\abs{Q(a+w)-a^2} = \abs{2aw+w^2}=  \abs{w} \abs{2a+w}= \abs{a}\abs{w}
\]
car $\abs{a} >r\geq \abs{w}$ et car~$\abs{2a}=\abs{a}$ puisque~$p\neq 2$. Ainsi
\[
Q_* (S (a, r)) = S (a^2, \abs a r) .
\]
La restriction de $Q_*$ \`a la demi-g\'eod\'esique
joignant~$a$ \`a~$\gamma$ (constitu\'ee des $S (a, t)$ avec $0 < t \leq \abs a $) 
est donc une
isom\'etrie sur la demi-g\'eod\'esique joignant
$a^2$ \`a~$\gamma$, si bien que 
\[ \deg_{S (a, r)} (Q) = 1.\]

\subsubsection*{Consid\'erons maintenant le cas $p = 2$}

Conservons les
notations pr\'ec\'edentes. Si~$S$ est situ\'ee sur $\gamma$,
la situation est la m\^eme que pr\'ec\'edemment.
Par contre, pour une coupure $S = S (a, r)$ avec
$0 < r < \abs{a }$, il faut distinguer trois
cas :
\begin{enumerate}
\item $\dist(S, \gamma) < 1 $,  \ie $
r >  \abs{a }/2$. On a alors $\abs{ Q (a + w) - Q (a) }  = 
\abs{w }^2$ pour $r > \abs{w } > \frac{1}{2}
\abs{a } $, si bien que
\begin{align*}
Q_* ( S (a, r))  & = S (a^2, r^2) \\
\dist(Q_* (S (a, r)), \gamma) & = 2  \dist(S (a, r), \gamma), \\
\deg_S (Q) & = 2,
\end{align*}
et $\bar Q (z) = z^2$ dans les
param\'etrages naturels de~$S$ et $Q_* (S)$.

\item $ \dist(S, \gamma) > 1 $,  \ie $
0 < r <  \abs{a }/2$. On a maintenant
\begin{align*}
Q_* ( S (a, r)) & = S (a^2 , \frac{1}{2} \abs{a } r) , \\
\dist(Q_* (S) , \gamma) & = \dist (S , \gamma) + 1 , \\
\deg_S ( Q) & = 1 .
\end{align*}

\item $ \dist(S, \gamma) = 1 $, \ie $ r = \abs{a }/2$. 
Dans ce cas, on a 
\begin{align*}
Q_* ( S) & = S (a^2 , \frac{1}{4} \abs{a }^2) , \\
 \dist (Q_* (S) , \gamma) & = 2 , \\
\deg_S ( Q) &= 2 ,
\end{align*}
mais dans les param\'etrages naturels,
$ \bar Q (z) = z + z^2$.
\end{enumerate}

Ainsi, on a $\deg_S(Q)=2$ dans les cas~(1) et~(3), mais les
r\'eductions $\bar Q$ diff\`erent.

%%%%%%%%%%%%%%%%%%%%%%%%%%%%%%%%%%%%%%%
%%%%%%%%%%%%%%%%%%%%%%%%%%%%%%%%%%%%%%%

\section{Th\'eorie de Fatou-Julia, domaines quasi-p\'eriodiques}\label{par:fatoujulia}

%%%%%%%%%%%%%%%%%%%%%%%%%%%%%%%%%%%%%%%
%%%%%%%%%%%%%%%%%%%%%%%%%%%%%%%%%%%%%%%

Soit~$R$ une fraction rationnelle non constante.
On a vu que $R$ d\'etermine un endomorphisme
de~$\BP(\C_p)$ qui pr\'eserve  $\P^1(\C_p)$, $\HpR$, $\HpQ$ et $\Hpsing$.

Nous allons maintenant 
d\'ecrire le type des points p\'eriodiques de~$R$, ainsi que
la structure des ensembles  de Fatou et de Julia.
Il s'agit d'analyser \`a la fois la dynamique sur $\P^1(\C_p)$ et $\BP(\C_p)$. 

\subsection{Points fixes, points p\'eriodiques et bassins d'attraction}
\label{par:points_fixes}

Dans $\P^1(\C_p)$, un point fixe $z_0$ de~$R$, c'est-\`a-dire
un \'el\'ement $z_0\in\P^1(\C_p)$ tel que $R(z_0)=z_0$,
est dit \emph{super-attractif,} \emph{attractif,} \emph{indiff\'erent} ou 
\emph{r\'epulsif}
suivant que $R'(z_0)=0$, $\abs{R'(z_0)}<1$, $\abs{R'(z_0)}=1$
ou $\abs{R'(z_0)}>1$.
Cette notion est invariante par conjugaison.\footnote
  {La  d\'eriv\'ee $R'(z_0)$ est d\'efinie
  de mani\`ere analogue \`a ce qui a \'et\'e fait p.~\pageref{page.derivee}.
  Si $z_0\neq\infty$, il s'agit de la d\'eriv\'ee usuelle;
  si $z_0=\infty$, on conjugue d'abord~$R$ par l'inversion.}
 
Dans l'espace hyperbolique~$\Hp$, une coupure~$S$
qui est fix\'ee par~$R$ est dite r\'epulsive
si $\deg_S (R)>1$ et indiff\'erente si $\deg_S (R)=1$. 
D'apr\`es le paragraphe \ref{subsec.action-coupures}, 
une coupure irrationnelle ou un bout singulier fix\'e par $R$
est un point fixe indiff\'erent.

Un point p\'eriodique de p\'eriode~$k$ sera dit attractif, indiff\'erent ou r\'epulsif
si c'est un point fixe de~$R^k$ de ce type.

\medskip

Si $\xi$ est un point de~$\BP(\C_p)$,
on appelle \emph{bassin d'attraction (fin)}
l'ensemble des $x\in\BP(\C_p)$ tels que $R^n(x)$ converge
vers~$\xi$ pour la topologie d'arbre r\'eel.
{\emph{Le bassin d'attraction (fin) d'un point fixe attractif
est un  ouvert de~$\BP(\C_p)$ pour la topologie faible.}}

Contrairement au cas complexe, il peut y avoir une infinit\'e d'orbites
p\'eriodiques attractives ; en particulier, un bassin d'attraction
ne contient pas forc\'ement de point critique.
%Serge J'ai chang\'e un peu l'annonce des exercices pour \^etre en accord avec la rem du rapporteur.

\begin{eg} Si $R(z)=z^p$, la r\'eduction de~$R$  est donn\'ee
par $\bar R(\bar z)=\bar z^p$. Ainsi, $R$ a bonne 
r\'eduction et  $R(S_\can)=S_\can$ (voir le \S \ref{par:reduction}).
Toutes les orbites de~$\bar R$ dans~$\P^1(\overline{\F_p})$
sont pr\'ep\'eriodiques ; chaque orbite de p\'eriode~$k$ se rel\`eve
en une orbite de m\^eme p\'eriode  qui est attractive.
(Si $\alpha$ est un \'el\'ement du corps fini \`a~$p^k$ \'el\'ements~$\F_{p^k}$
dans~$\overline{\F_p}$, alors $\alpha^{p^k}=\alpha$;
pour tout $z_0\in B(\alpha)$, la suite~$(R^{nk}(z_0))_n$ converge
vers une racine de l'unit\'e (ou~$0$ si $\alpha=0$)
dans~$\C_p$ qu'on appelle le \emph{repr\'esentant de Teichm\"uller} de~$\alpha$.)
\end{eg}

\begin{ex}
Soit~$Q$ la fraction rationnelle d\'efinie par 
\[
Q(z) = \lambda \left( \frac{z}{z-1} + z^p \right),
\]
o\`u $\lambda \in \C_p$ satisfait $p^{-(p-1)}< \abs{\lambda}^{p+1} < 1$. 

\begin{enumerate}\def\labelenumi{\theenumi)}
\item D\'emontrer que $0$ et~$\infty$ sont des points critiques fixes
et attractifs.
Nous noterons~$W_0$ et $W_\infty$ leurs bassins d'attraction. 

\item 
D\'emontrer que $\abs{Q(z)}= \abs{\lambda}\abs{z}$ lorsque $\abs{z}<1$. 
En d\'eduire que~$W_0$ contient la boule ouverte~$B(0)$. 

\item Si $\abs{z}>1$, d\'emontrer que $\abs{Q(z)}=\abs{\lambda} \abs{z}^p$. 
En d\'eduire  que~$W_\infty$ contient l'ensemble
des points~$z$ satisfaisant $\abs{\lambda}\abs{z}^{p-1} >1 $.
\item 
D\'emontrer que la valeur absolue de tous les points critiques de~$Q$ dans~$\C_p$ 
est \'egale \`a~$p^{1/(p+1)}$ (commencer par prouver qu'ils sont de
valeur absolue~$>1$); en d\'eduire qu'ils sont contenus dans~$W_\infty$.

\item D\'emontrer que le bassin d'attraction de l'origine 
ne contient aucun point critique. 
\end{enumerate}
\end{ex}

\subsection{Dynamique locale d'une s\'erie enti\`ere}\label{par:perio}

La dynamique locale d'une fraction rationnelle $R \in \C_p(z)$ au voisinage
d'un point fixe a \'et\'e \'etudi\'ee en d\'etail par plusieurs auteurs (voir~\cite{Rivera-Letelier:Asterisque}, \cite{Lubin:1994}, \cite{Herman-Yoccoz:1983}, 
et les r\'ef\'erences qui s'y trouvent).
Le but de ce paragraphe est de signaler quelques lemmes 
qui manifestent clairement le caract\`ere non archim\'edien de $\C_p$ et permettent d'appr\'ehender la dynamique au voisinage d'un point fixe attractif ou indiff\'erent dans $\P^1(\C_p)$. 

Pour simplifier,
nous \'etudions ici la dynamique locale d'une s\'erie enti\`ere 
convergente~$f$ fixant l'origine ; 
le cas d'un point p\'eriodique~$z_0$ de~$R$ de p\'eriode~$k$ 
s'y ram\`ene en rempla\c{c}ant~$R$ par $R^k$ et en conjuguant~$R$ 
par une transformation de M\"obius qui 
envoie~$z_0$ sur l'origine.

\subsubsection*{Rayon de convergence, principe du maximum}

Commen\c{c}ons par quelques g\'en\'eralit\'es.
Soit 
\[
f(z) = a_0 + a_1 z + a_2 z^2 + \dots + a_n z^n  + \dots
\]
une s\'erie enti\`ere \`a coefficients dans~$\C_p$.
Son rayon de convergence $\rho(f)$ est donn\'e par la formule usuelle:
\[ \rho(f) = 1 / \limsup \abs{a_k} ^{1/k} .  \]
Supposons-le strictement positif.
Comme en analyse complexe, $f$ d\'efinit une fonction continue
sur la boule non circonf\'erenci\'ee~$\Bo(0,\rho(f))$.

Comme en analyse complexe, les fractions rationnelles peuvent
\^etre d\'evelopp\'ees en s\'erie enti\`ere: si $R\in\C_p(z)$
est une fraction rationnelle sans p\^ole dans le disque non circonf\'erenci\'e~$\Bo(0,r)$, alors le d\'eveloppement en s\'erie formelle de~$R$ en l'origine
poss\`ede un rayon de convergence~$\geq r$ et la somme de cette s\'erie
co\"{\i}ncide avec~$R$ dans le disque~$\Bo(0,r)$.

Revenons au cas g\'en\'eral d'une s\'erie enti\`ere \`a coefficients
dans~$\C_p$, comme ci-dessus.
Pour $r\geq 0$, posons $M(r)=\sup_{k\geq 0}\abs{a_k} r^k$.
Supposons $r<\rho(f)$; alors,
$M(r)=\sup_{\abs{z}\leq r} \abs{f(z)}$. 
L'in\'egalit\'e $\abs{f(z)}\leq M(r)$ si $\abs z\leq r$ r\'esulte en effet
simplement de l'in\'egalit\'e ultram\'etrique, chaque terme de
la s\'erie d\'efinissant $f(z)$ \'etant de module au plus \'egal \`a~$M(r)$.

Si $\abs z<r$ et $f$ n'est pas constante,
on constate que l'in\'egalit\'e est stricte:
$\abs{f(z)}<M(r)$ --- c'est un analogue du \emph{principe du maximum}.
Si $r\not\in p^{\Q}$, la borne sup\'erieure qui définit~$M(r)$
n'est pas non plus atteinte 
puisqu'il n'existe pas d'\'el\'ement de~$\C_p$ de module~$r$.

Elle est en revanche atteinte en un point~$z$ de module~$r$
dans le cas o\`u $r\in p^{\Q}$. 
L'argument
est analogue \`a celui du lemme~\ref{lemm.gauss}:
par changement de variables et multiplication
par un scalaire analogues  \`a ce qui a \'et\'e fait p.~\pageref{page.gauss},
on se ram\`ene au cas o\`u $r=M(r)=1$. Alors, la r\'eduction~$\overline f$
modulo~$\mathfrak m$ de~$f$ est un polyn\^ome 
et pour tout \'el\'ement~$z$ de~$\mathscr O$
dont la r\'eduction modulo~$\mathfrak m$ n'est pas un z\'ero de~$\overline f$,
on a $\abs{f(z)}=1$.
(Voir aussi \cite{dwork-g-s94}, p.~114, prop.~1.1.)

Par passage \`a la limite, on en d\'eduit aussi que
$M(r)=\sup_{\abs z<r} \abs{f(z)}$ si $r\leq\rho(f)$.

Si de plus $f(0)=0$, la th\'eorie du polygone de Newton 
(voir~\cite{dwork-g-s94}, chapitre~II) entra\^{\i}ne m\^eme
que $f(\Bf(0,r))=\Bf(0,M(r))$ pour tout~$r$ tel que $0\leq r<\rho(f)$.

\subsubsection*{Boules invariantes}
%

% Dans la suite de ce paragraphe, on consid\`ere une s\'erie enti\`ere
On a $f'(0)=a_1$. Supposons $a_1\neq 0$. Alors,
pour $\abs z$ assez petit, $\abs{f(z)}=\abs{a_1}\abs z$.
Par l'in\'egalit\'e ultram\'etrique, il suffit en effet 
que $\abs z <r$, o\`u $r$ est choisi de sorte que
$\abs{a_n} r^{n-1}\leq \abs{a_1}$ pour tout~$n\geq 1$.

En particulier, si $0$ est un point fixe r\'epulsif, alors $\abs{a_1}>1$
et, pour $r$ assez petit, aucune boule $B(0,r)$ de~$\C_p$
n'est stable par~$f$.

En revanche, si $0$ est un point fixe indiff\'erent ou attractif,
toute boule de centre~$0$ et de rayon assez petit est stable par~$f$.
Cela vaut encore si $f'(0)=0$, il suffit de reprendre l'argument
pr\'ec\'edent en tenant compte du premier terme non nul
(si c'est celui de degr\'e~$d$, on aura $\abs{f(z)}=\abs{a_d}\abs{z^d}$
pour $\abs{z}$ assez petit).

Nous avons ainsi justifi\'e le lemme suivant.
\begin{lemm}\label{lemm.boules}
Supposons que $0$ soit un point fixe indiff\'erent ou attractif.
Toute boule de centre~$0$ et de rayon assez petit est alors stable par~$f$.

Inversement, si $0$ est un point fixe r\'epulsif,
aucune boule de centre~$0$ et de rayon assez petit n'est stable 
par~$f$.
\end{lemm}

Dans le complexe, c'est-\`a-dire pour la dynamique d'un polyn\^ome ou
d'une application holomorphe d'une variable, la dynamique au voisinage
d'un point fixe indiff\'erent est en g\'en\'erale tr\`es riche. Nous renvoyons le
lecteur au livre~\cite{Milnor:book}, qui contient une description de la dynamique
lorsque la d\'eriv\'ee de l'application au point fixe est une racine de l'unit\'e
(fleur de Leau-Fatou) ou d'ordre infini (ph\'enom\`enes de petits diviseurs).

Le lemme précédent montre que certains de ces phénomènes disparaissent dans
le monde de la dynamique $p$-adique. 
Il résulte en effet de ce lemme et du suivant
que la dynamique au voisinage d'un point fixe indifférent 
est celle d'une isométrie.
Par ailleurs, lorsque $f(0)=0$ et $f'(0)$ n'est pas une racine de l'unité,
Herman et Yoccoz ont démontré dans~\cite{Herman-Yoccoz:1983} qu'il existe
une série entière $g(z)=z+b_2z^2+\dots$, convergente au voisinage de~$0$,
qui conjugue~$f$ à sa partie linéaire: 
\[
g\circ f\circ g^{-1}(z) =a_1 z.
\] 
Les phénomènes de petits diviseurs disparaissent donc dans le monde
de la dynamique $p$-adique.\footnote{Par contre, ils réapparaissent dans le
cadre de la dynamique ultramétrique en caractéristique positive~\cite{lindahl2004}.}
 
% 
% Le lemme pr\'ec\'edent montre que ces ph\'enom\`enes disparaissent dans le
% monde de la dynamique $p$-adique (ou ultram\'etrique). 
% Plus pr\'ecis\'ement, il r\'esulte de ce lemme et du suivant 
% que la dynamique au voisinage d'un point fixe indiff\'erent 
% est celle d'une isom\'etrie. Les subtilit\'es 
% de la th\'eorie des syst\`emes dynamiques complexes
% caus\'es par les ph\'enom\`enes de petits diviseurs 
% n'ont donc pas d'analogue~$p$-adique. 
% 

\begin{lemm}\label{lemm.schwarz}
~$r$ un nombre r\'eel strictement positif.
Soit $f(z)= \sum_{k\geq 1} a_k z^{k}$ une s\'erie enti\`ere \`a 
coefficients dans $\C_p$. Soit $c$ un nombre r\'eel positif.

1) Les propri\'et\'es suivantes sont \'equivalentes : 
\begin{enumerate}\def\theenumi{\roman{enumi}}
\item Le rayon de convergence de~$f$ est au moins \'egal \`a~$r$
et l'on a $\abs{f(z)}\leq c \abs{z}$ pour tout point~$z$ de $\Bo (0,r)$ ;
\item Pour tout entier~$k\geq 1$, $\abs{a_k}\leq c r^{1-k}$.
\end{enumerate}

2) Si elles sont satisfaites,
l'application $f$ est $c$-lipschitzienne sur $B(0,r)$.

3) Supposons qu'elles soient satisfaites; 
les propri\'et\'es suivantes sont alors \'equivalentes:
\begin{enumerate} \def\theenumi{\roman{enumi}}
\item $f$ induit une isom\'etrie de $B(0,r)$ ;
\item $f$ induit un automorphisme de $B(0,r)$ ;
\item pour tout $z\in\Bo(0,r)$, on a  $\abs{f'(z)}=1$;
\item pour tout $z\in \Bo(0,r)$, on a  $\abs{f'(z)}\leq 1$ et l'\'egalit\'e est atteinte.
\end{enumerate}
\end{lemm}
\begin{proof}
Lorsque $\abs{f(z)}$ est inf\'erieur \`a~$c\abs z$ sur la boule de centre $0$ et
de rayon $r$, le supremum de $\abs{f(z)}$ sur cette boule est inf\'erieur \`a~$cr$ ; 
ainsi, le maximum des $\abs{a_k} r^k$ est inf\'erieur \`a~$cr$. R\'eciproquement, 
lorsque $\max ( \abs{a_k} r^{k-1})$ est major\'e par $c$, on obtient 
\[
\abs{f(z)} = \abs{\sum_{k\geq 1} a_k z^{k-1}} \abs{z} \leq c \abs{z}.
\]
Ceci d\'emontre la premi\`ere propri\'et\'e \'enonc\'ee. La deuxi\`eme s'en d\'eduit 
\`a l'aide de la majoration
\[ \abs{z^k-w^k} = \abs{z-w}\abs{z^{k-1}+z^{k-2}w+\dots+w^{k-1}}
   \leq r^{k-1} \abs{z-w} \]
pour $z,w\in\Bo(0,r)$.
Passons \`a la troisi\`eme. 

Par hypoth\`ese, $\abs{f(z)}\leq c\abs z$ si $\abs z<r$.
D'autre part, on vient de rappeler que pour $s<r$, $f(\Bf(0,s))=\Bf(0,s')$,
o\`u 
$
s'=\max \abs{a_n}s^n.
$
Lorsque $s$ tend vers~$r$, $s'$ tend vers~$r'=\max \abs{a_n}r^n=\tilde cr$,
o\`u $\tilde c$ est d\'efini par 
\[
\tilde c = \max\abs{a_n}r^{n-1}.
\]
L'image de $\Bo(0,r)$ est donc la boule~$\Bo(0,\tilde cr)$.
Sous les assertions~(i) et~(ii), on voit donc que $\tilde c=1$.

Commen\c{c}ons par \'etablir l'analogue du \emph{lemme de Schwarz}
en d\'emontrant que $\abs{f'(z)}\leq\tilde c$ pour tout~$z\in\Bo(0,r)$.
Posons en effet $g(z)=f(z)/z$ pour $z\in\Bo(0,r)$ ; c'est la somme d'une 
s\'erie enti\`ere convergente et $\abs{g(z)}\leq \tilde c r/s$ si $\abs z\leq s<r$. 
Par le principe du maximum,
on a donc $\sup_{\Bf(0,s)} \abs g\leq \tilde c r/s$.
Faisant tendre~$s$ vers~$r$, on voit que $\abs{g(z)}\leq \tilde c $ 
pour tout~$z\in\Bo(0,r)$.  En particulier, $\abs{g(0)}=\abs{f'(0)}\leq\tilde c$.
Le cas g\'en\'eral s'en d\'eduit en consid\'erant la s\'erie $f(z+w)$, pour $w
\in\Bo(0,r)$.

Nous pouvons maintenant d\'emontrer le point $(3)$.
L'implication~(i)$\Rightarrow$(ii) est \'evidente, de m\^eme
que l'implication (iii)$\Rightarrow$(iv).
Sous l'assertion~(i), la formule de Taylor
\[ \abs{f(z+w)-f(z)}=\abs{f'(z)}\abs w + \mathrm O(\abs w^2), \]
pour $z$ et $w\in\Bo(0,r)$, 
entra\^{\i}ne que $\abs{f'(z)}=1$ pour tout $z\in\Bo(0,r)$,
d'o\`u~(iii).

Supposons~(ii). Comme une s\'erie formelle n'est pas injective
au voisinage d'un point critique, $f'$ ne s'annule pas
sur~$\Bo(0,r)$. Le principe du maximum pour~$1/f'$
entra\^{\i}ne alors que $\abs{f'}$ est de module constant 
sur la boule~$\Bo(0,r)$, n\'ecessairement \'egal \`a~$1$.

Il reste \`a d\'emontrer que (iv) entra\^{\i}ne~(i). On observe
d'abord par le principe du maximum
que $\abs{f'(z)}=1$ pour tout $z\in\Bo(0,r)$. 
En particulier, $f'$ ne s'annule pas sur $\Bo(0,r)$. 
Reprenons la fonction $g(z)=f(z)/z$ pour $z\in\Bo(0,r)$ utilis\'ee
pour d\'emontrer le lemme de Schwarz.  On  a vu 
que $\abs{g(z)}\leq 1$ pour tout~$z\in\Bo(0,r)$. 
Puisque $\abs{g(0)}=\abs{f'(0)}=1$, le principe du maximum entra\^{\i}ne
de nouveau que $\abs{g(z)}=1$ pour tout $z\in\Bo(0,r)$, autrement
dit $\abs{f(z)}=\abs z$.  En appliquant ce raisonnement
\`a la s\'erie enti\`ere $f(z+w)$, pour $w\in\Bo(0,r)$ fix\'e, on 
en d\'eduit que $f$ est une isom\'etrie de~$\Bo(0,r)$.
\end{proof}

\subsection{Ensemble exceptionnel et bonne r\'eduction}

\begin{defi}
On dit qu'un point $z\in\BP(\C_p)$ est \emph{exceptionnel}
si la r\'eunion des $R^{-n}(z)$, pour $n\geq 0$, est un ensemble fini.
\end{defi}

Dans le cas complexe, on sait qu'il y a au plus deux points
exceptionnels. S'il y en a exactement un,
$R$ est conjugu\'ee \`a un polyn\^ome ; s'il y en a deux,
alors~$R$ est conjugu\'ee \`a~$z^d$ ou \`a~$z^{-d}$. 
Une d\'emonstration est possible  via la formule de Riemann--Hurwitz
(voir~\cite{Milnor:book}, lemma 4.9, ou \cite{Silverman:book}, theorem 1.6).
Elle s'\'etend \`a~$\P^1(\C_p)$ ($\C_p$ est de caract\'eristique $0$): 
\begin{prop}
Il y a au plus deux points exceptionnels dans~$\P^1(\C_p)$.
S'il y en a deux, $R$ est conjugu\'ee \`a~$z^d$ ou \`a~$z^{-d}$;
s'il y en a un seul, $R$ est conjugu\'ee \`a un polyn\^ome.
\end{prop}
Comme dans le cas complexe,
on observe que ces points exceptionnels sont attractifs.
La nouveaut\'e provient donc des autres
points de l'espace hyperbolique $p$-adique:
\begin{prop}[Rivera-Letelier, \cite{Rivera-Letelier:CMH}, \S 7]
La coupure canonique est un point exceptionnel
si et seulement si la fraction rationnelle~$R$ a bonne r\'eduction.

Il y a au plus un point exceptionnel dans~$\Hp$ ; si un tel 
point existe, il est rationnel.
\end{prop}
S'il y a un point exceptionnel suppl\'ementaire,
on peut conjuguer~$R$ de sorte que ce soit la coupure canonique;
alors $R$ a bonne r\'eduction.
Autrement dit, il y a un point exceptionnel suppl\'ementaire~$S$
si et seulement si la fraction rationnelle~$R$ est simple.
En outre, ce point exceptionnel est r\'epulsif
car $\deg_SR=\deg R>1$ ; un tel point ce situe donc dans l'ensemble
de Julia, tandis que les points exceptionnels de $\P^1(\C_p)$ appartiennent
\`a l'ensemble de Fatou. 

\begin{ex}\label{ex:polsimple}
Soit $P\in\C_p[z]$ un polyn\^ome de degr\'e au moins~$2$. D\'emontrer
que les propri\'et\'es suivantes sont \'equivalentes:
\begin{enumerate}
\item $P$ est conjugu\'e \`a une fraction rationnelle ayant bonne r\'eduction;
\item $P$ poss\`ede un point fixe totalement invariant dans~$\HpQ$;
\item $P$ est conjugu\'e \`a un polyn\^ome ayant bonne r\'eduction.
\end{enumerate}
\end{ex}

%%%
\subsection{Dynamique des applications ayant bonne r\'eduction}\label{par:bonne-reduction}
%%%

Passons maintenant \`a l'\'etude des applications ayant bonne r\'eduction. 

\subsubsection*{Dynamique sur $\P^1(\C_p)$} 

Soit $R$ une fraction rationnelle de degr\'e~$\geq 2$ ayant bonne
r\'eduction. Soit $\F_q$ un sous-corps fini de~$\overline{\F_p}$
contenant les coefficients de la r\'eduction~$\overline R$ de~$R$.
La dynamique de~$\overline R$ sur~$\P^1(\overline{\F_p})$ est assez
simple, puisque tout point est pr\'ep\'eriodique. En effet,  si $\alpha$
est un tel point, il existe un sous-corps fini~$F$ de~$\overline{\F_q}$
contenant~$\F_p$ tel que $\alpha\in\P^1(F)$; alors, 
$\P^1(F)$ est un ensemble fini stable par~$\overline R$, en particulier,
l'orbite de~$\alpha$ par~$\overline R$ est finie.

Si $\pi$ d\'esigne l'application de r\'eduction de~$\P^1(\C_p)$
sur~$\P^1(\overline{\F_p})$, on a $\pi\circ R = \overline R\circ\pi$.
Plus pr\'ecis\'ement, la th\'eorie du polygone de Newton
entra\^{\i}ne que l'image $R(B(\alpha))$ de la boule~$B(\alpha)$ par~$R$
est \'egale \`a~$B(\bar{R}(\alpha))$ et l'application
$R:B(\alpha)\to B(\bar{R}(\alpha))$
est de degr\'e $\deg_\alpha(\bar R)$.
\`A chaque point p\'eriodique $\alpha$ de $\bar R$ 
correspond en particulier une boule $B(\alpha)$ qui est p\'eriodique pour~$R$,
au sens o\`u un it\'er\'e convenable~$R^k$ de~$R$ applique $B(\alpha)$ sur elle-m\^eme.
D'apr\`es le lemme~\ref{lemm.schwarz},
la dynamique de~$R^k$ en restriction \`a cette boule est isom\'etrique 
ou attractive. La proposition suivante pr\'ecise cela. 

\begin{prop}[voir~\cite{Rivera-Letelier:Asterisque},
prop.~4.32]\label{pro:br2}
Soit $R\in \C_p(z)$ une fraction rationnelle ayant bonne r\'eduction.
Soit $\alpha$ un point de  $\P^1(\bar{\F}_p)$ qui est p\'eriodique de p\'eriode~$k$.

\begin{enumerate}
\item Pour que la boule $B(\alpha)$ contienne un point p\'eriodique attractif 
de p\'eriode~$k$, il faut et il suffit que 
$(\bar{R}^k)'(\alpha) = 0$.

\item
Si $z$ est un tel point, 
$B(\alpha)$ est la plus grande boule centr\'ee en~$z$ 
qui soit contenue dans le bassin d'attraction de~$z$.

\item Si, au contraire,  $(\bar R^k)'(\alpha)\neq 0$, 
la boule $B(\alpha)$ est contenue dans le domaine
de quasi-p\'eriodicit\'e\,\footnote
 {Le domaine de quasi-p\'eriodicit\'e est d\'efini ci-dessous au paragraphe~\ref{par:FatouQP}.} 
de~$R$ et contient un point p\'eriodique indiff\'erent de p\'eriode~$k$.
\end{enumerate} 
\end{prop} 

\begin{rem}
Cette proposition peut-\^etre compl\'et\'ee par le th\'eor\`eme suivant, que nous 
ne d\'emontrerons pas. {\emph{
 Soit $R\in \C_p(z)$ une fraction rationnelle. Si l'on munit $\P^1(\C_p)$
 de la distance sph\'erique, $R$ est lipschitzienne de rapport~$ \Delta(R)^{-2} $.
%% acl- suppression de la formule centree
 Si $R$ a bonne r\'eduction, $R:\P^1(\C_p)\to \P^1(\C_p)$ 
 est donc~$1$-lipschitzienne et tous ses points p\'eriodiques sont donc attractifs ou indiff\'erents}} (voir~\cite{Morton-Silverman:Crelle}). 
\end{rem}

\begin{proof}[\'El\'ements de d\'emonstration]
Soit $z$ un point p\'eriodique attractif de p\'eriode~$k$ dans~$B(\alpha)$ ; 
alors $\bar z=\alpha$ est un point p\'eriodique pour~$\bar R$ 
dont la p\'eriode divise~$k$.
Comme $(\bar R^k)'(\bar z)$ est la r\'eduction de $(R^k)'(z)$,
on a $(\bar R^k)'(\alpha)=0$.

Soit $B$ une boule centr\'ee en~$z$ qui n'est pas contenue
dans~$B(\alpha)$. La r\'eduction de~$B$ contient une infinit\'e de points
$\beta$ de~$\P^1(\overline{\F_p})$ pour lesquels la suite
$(\bar R^k)(\beta)$ est pr\'ep\'eriodique et ne contient pas~$\alpha$.
Pour un tel $\beta$, la boule~$B(\beta)$ est donc disjointe
du bassin d'attraction de~$z$.

Supposons, inversement, que $(\bar R^k)'(\alpha)=0$ et
d\'emontrons que $R^k$ poss\`ede un unique point fixe dans~$B(\alpha)$.
Quitte \`a conjuguer par une homographie de~$\PGL_2(\mathscr O)$,
on suppose $\alpha=0$ et on \'ecrit $R^k=a_0+a_1z+\dots$ le d\'eveloppement
en s\'erie enti\`ere de~$R^k$. Il est de rayon de convergence
au moins \'egal \`a~$1$ car $R^k$ n'a pas de p\^ole dans la boulue~$B(\alpha)=\Bo(0,1)$.
Comme $\bar{R^k}(\alpha)=\alpha$,
on a $\abs{a_0}<1$; comme $\bar{R^k}'(\alpha)=0$, $\abs{a_1}<1$. 
En outre, $\sup_{\Bo(0,1)}\abs {R^k}=\max(\abs{a_n})\leq 1$ 
car $R^k$ a bonne r\'eduction.
Soit $r$ un nombre r\'eel tel que $\abs{a_0}<r<1$. Posons $c=\max_{n\geq 1}\abs{a_n}r^{n-1}$; on a $c<1$.
D'apr\`es le lemme~\ref{lemm.schwarz}, $R^k$ est $c$-lipschitzienne
sur la boule~$\Bf(0,r)$ et $\abs{R^k(z)}\leq \max(cr,r)\leq r$
pour tout $z\in\Bf(0,r)$. Par le th\'eor\`eme du point fixe,
il en r\'esulte que $R^k$ poss\`ede un unique
point fixe~$z_0$ dans~$\Bf(0,r)$ et, par cons\'equent, un unique
point fixe, attractif, dans~$B(\alpha)$. 
Plus g\'en\'eralement, on voit que $z$ est l'unique point
p\'eriodique de~$R^k$ dans~$B(\alpha)$.
%% acl - modifications de la d\'emonstration

Nous renvoyons \`a~\cite{Rivera-Letelier:Asterisque}
pour le reste de la d\'emonstration.
\end{proof}

\subsubsection*{Dynamique sur $\Hp$} 

D\'ecrivons la dynamique de~$R$ sur~$\Hp$.
La coupure canonique $S_\can$ est totalement invariante 
et les branches issues de $S_\can$, en bijection
avec $\P^1(\overline{\F_p})$, sont toutes pr\'ep\'eriodiques. 
Si $S$ est un \'el\'ement de $\Hp$, 
on a expliqu\'e au~\S\ref{subsec.action-coupures}
que
\[
\dist(R_*(S),S_\can)\geq \dist(S,S_\can).
\]
Si le degr\'e de~$R$ sur la branche contenant~$S$ 
est strictement plus grand que~$1$, alors, au voisinage de $S_\can$, $R_*$ 
dilate strictement la g\'eod\'esique  reliant $S_\can$ \`a $S$ ; 
dans ce cas, l'in\'egalit\'e est donc stricte.
%Serge J'ai chang\'e suivant la remarque du rapporteur.
% Il y a d'ailleurs qqch \`a dire sur l'injectivit\'e de R_* le long de la g\'eod\'esique reliant S_\can \`a S ....
La figure \ref{fig:attractif} illustre ce cas lorsque
la branche est fix\'ee par~$R$ ; 
la boule qui d\'efinit cette branche contient donc un point
fixe attractif (cf. la proposition~\ref{pro:br2}). 

\begin{figure}[t]
\centering
\includegraphics{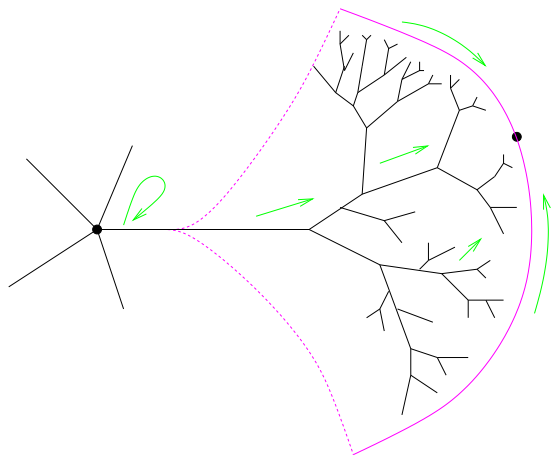} 
\caption{\textsc{Bassin d'attraction}:
Bassin d'attraction d'un point pour une
transformation rationnelle qui a bonne r\'eduction : 
$R(S_\can)=S_\can$ ; ici, le cycle
attractif est r\'eduit \`a un point fixe et la branche 
{\og pointant vers ce point\fg} est donc fix\'ee.}
\label{fig:attractif}
\end{figure}

\subsubsection*{R\'eduction ins\'eparable et cycles attractifs} 
On dit que la r\'eduction d'une fraction rationnelle~$R$ est {\emph{ins\'eparable}} 
lorsque la d\'eriv\'ee de~$\bar R$ est identiquement nulle.
Cela revient \`a dire que $\bar R$ 
est de la forme $Q(z^p)$ pour un \'el\'ement~$Q$ de $\bar{\F_p}(z)$.

\begin{prop}\label{pro:br}
Soit $R\in \C_p(z)$ une fraction rationnelle ayant bonne r\'eduction.
\begin{enumerate}
\item Si la r\'eduction de ${R}$ est ins\'eparable, alors $p$ divise le degr\'e de~$R$ et tout 
point p\'eriodique est attractif. 
\item Si $\bar{R}'$ n'est pas identiquement nulle, tout cycle attractif de~$R$ 
attire un point critique, si bien que le nombre de cycles attractifs est major\'e 
par $2\deg (R) -2$. 
\end{enumerate}
\end{prop}
\begin{proof}
Par bonne r\'eduction,
le degr\'e de~$R$ est \'egal \`a celui de~$\bar R$.
Puisque $\bar R$ est de la forme $Q(z^p)$, $\deg (R)$  
est un multiple de~$p$.
Le reste de l'assertion~(1) d\'ecoule alors de la proposition pr\'ec\'edente.
(Voir aussi~\cite{Rivera-Letelier:CMH}, prop.~6.4.)
L'assertion~(2) est le point~(iv) de la proposition~4.32
de~\cite{Rivera-Letelier:Asterisque}; recopions ce qu'\'ecrit Rivera-Letelier. 
Soit $w$ un point p\'eriodique attractif de p\'eriode $k$. Quitte \`a conjuguer
$R$ par une homographie $H\in \PSL(2,\mathscr O)$, nous supposons que $w=0$. 
D\'eveloppons alors $R^k$ en s\'erie enti\`ere au voisinage de $0$, 
\[
R^k(z)= \sum_{n=1}^\infty a_n z^n.
\]
Puisque $R$ a bonne r\'eduction, $\abs{a_n} \leq 1$ pour tout entier $n$,
et comme $0$ est attractif, $\abs{a_1} < 1$ ; comme $\bar R'$ n'est pas 
identiquement nulle, il existe un indice $i$ tel que $\abs{a_i} = 1$.
La th\'eorie du polygone de Newton montre alors qu'il existe un point~$z$ 
de valeur absolue $\abs{z}<1$ tel que $(R^k)'(z)=0$. Puisque $\abs{z}<1$,
$z$ est un point critique attir\'e par notre cycle attractif. 
\end{proof}

\begin{theo}[\cite{Rivera-Letelier:CMH}, th\'eor\`eme~2]
Une fraction rationnelle a un et un seul point p\'eriodique dans $\Hp$ si, 
et seulement si, apr\`es changement de coordonn\'ee, elle a bonne r\'eduction 
ins\'eparable. 
\end{theo}

Dans ce cas, tous les points p\'eriodiques de la fraction rationnelle
dans $\P^1(\C_p)$ sont attractifs (cf. le point~(1) 
de la proposition~\ref{pro:br}), 
et l'unique point fixe dans $\Hp$ est r\'epulsif.

\subsection{Ensembles de Julia et de Fatou}\label{par:JuliaFatou}

On appelle \emph{ensemble de Julia} de~$R$ et on note $\Jul(R)$
l'adh\'erence pour la topologie faible de l'ensemble des points
p\'eriodiques r\'epulsifs de~$\BP(\C_p)$.

\begin{theo}[Rivera-Letelier]
Soit $R\in \C_p(z)$ une transformation rationnelle de degr\'e
au moins~$2$. Alors, $R$ poss\`ede au moins un point fixe r\'epulsif 
dans~$\BP(\C_p)$. De plus, les propri\'et\'es suivantes sont \'equivalentes.
\begin{enumerate}
\item $R$ ne poss\`ede aucun point p\'eriodique dans $\Hp$;
\item $R$ n'a qu'un nombre fini de points p\'eriodiques non r\'epulsifs dans $\P^1(\C_p)$; 
\item $R$ poss\`ede un seul point p\'eriodique non r\'epulsif dans $\P^1(\C_p)$.
\end{enumerate} 
Si ces propri\'et\'es sont satisfaites, 
l'unique point p\'eriodique non r\'epulsif de~$R$ dans~$\P^1(\C_p)$ est 
un point fixe attractif.
\end{theo}

\begin{rem}
Rivera-Letelier montre aussi
qu'une fraction rationnelle $R\in \C_p(z)$ de degr\'e $\geq 2$ 
a n\'ecessairement $0$, $1$, ou une infinit\'e de
points p\'eriodiques dans $\Hp$.
Les deux th\'eor\`emes pr\'ec\'edents caract\'erisent donc les fractions rationnelles 
qui n'ont qu'un nombre fini d'orbites p\'eriodiques dans $\Hp$. \end{rem}

Nous retiendrons en particulier du th\'eor\`eme pr\'ec\'edent que
\emph{l'ensemble de Julia n'est pas vide.} En revanche, il peut ne 
contenir aucun point de~$\P^1(\C_p)$.
En effet, on a $\Jul(R)=\{S_\can\}$ lorsque $R$ a bonne r\'eduction
(voir le \S \ref{par:bonne-reduction}). Plus g\'en\'eralement, lorsque $R$ a un 
point exceptionnel~$S$ dans $\Hp$, $\Jul(R)$ est r\'eduit au 
singleton~$\{ S\}$. 

On d\'efinit l'\emph{ensemble de Fatou} de~$R$ comme
le compl\'ementaire de l'ensemble de Julia~$\Jul(R)$; on le note~$\Fat(R)$. 
Comme en dynamique holomorphe,
on peut d\'emontrer que l'ensemble de Fatou est 
le plus grand ouvert sur lequel la suite~$(R^n)$ 
est une famille normale. 
Nous renvoyons le lecteur \`a~\cite{Hsia:London} et~\cite{Silverman:book},
\S 5.4, 5.7 et 5.10.3.2 pour plus de d\'etails.

Par d\'efinition, les points
p\'eriodiques r\'epulsifs sont dans l'ensemble de Julia.
D'apr\`es les r\'esultats du paragraphe~\ref{par:perio},
la dynamique au voisinage d'un point p\'eriodique 
indiff\'erent est isom\'etrique; par suite, les points
p\'eriodiques indiff\'erents appartiennent \`a l'ensemble de Fatou, 
et il en est de m\^eme des points p\'eriodiques attractifs.
D'apr\`es le th\'eor\`eme suivant, de tels points p\'eriodiques existent toujours.
\begin{theo}[Benedetto, \cite{Benedetto:these}]
Toute transformation rationnelle $R\in \C_p(z)$ de degr\'e
au moins~$2$ poss\`ede un point fixe dans~$\P^1(\C_p)$
qui est attractif ou indiff\'erent. 
\end{theo}
Par cons\'equent, \emph{l'ensemble de Fatou~$\Fat(R)$ n'est pas vide.}
Ce r\'esultat va  \`a l'encontre de ce qui est
valable pour la dynamique complexe puisque, sur $\C$, 
l'ensemble de Fatou peut \^etre vide et il n'existe pas 
toujours de point p\'eriodique attractif ou indiff\'erent (voir par
exemple les {\og exemples de Latt\`es \fg}, \cite{Cantat:ER}, \S 7, 11 et 12).

\begin{proof}[D\'emonstration (voir aussi \cite{Benedetto:these}, \S 2.4.4)]
Soit $d\geq 2$ le degr\'e de~$R$ et
notons  $z_0,\dots,z_{d}$ les $d+1$ 
points fixes de~$R$ r\'ep\'et\'es suivant leur multiplicit\'e.
Pour $i\in\{0,\dots,d\}$, soit $\lambda_i=R'(z_i)$ 
le multiplicateur de~$R$ au point $z_i$.
Si tous les $\lambda_i$ sont diff\'erents de~$1$, alors les $z_i$
sont distincts et 
\begin{equation} \label{eq.lef}
\sum_{i=0}^{d}\frac{1}{\lambda_i-1}=1.
\end{equation}
Dans le cas complexe, c'est-\`a-dire pour $R\in \C(z)$ et 
$z_i \in \P^1(\C)$, cette \'egalit\'e r\'esulte de la formule de Lefschetz
holomorphe, qui peut \^etre d\'emontr\'ee facilement \`a l'aide du th\'eor\`eme des
r\'esidus (voir~\cite{Milnor:book}, \S 12, et~\cite{Griffiths-Harris:book}, \S 3.4).
Mais il s'agit en fait d'une identit\'e alg\'ebrique
que l'on peut d\'emontrer directement 
(voir l'exercice~\ref{exo:lef}).
Elle est donc valable dans notre cas.

Supposons maintenant par l'absurde que la valeur absolue de
chaque multiplicateur $\lambda_i$ soit strictement
sup\'erieure \`a~$1$. Dans ce cas, l'in\'egalit\'e ultram\'etrique
entra\^{\i}ne tout d'abord que
\[
\abs{\frac{1}{\lambda_i-1}}= \frac{1}{\abs{\lambda_i}} < 1
\]
pour tout indice~$i$, puis que la somme des inverses des
$\abs{\lambda_i -1}$ est elle-m\^eme strictement inf\'erieure
\`a~$1$. Ceci contredit la formule de Lefschetz, et montre donc le r\'esultat annonc\'e. 
\end{proof}

\begin{ex}\label{exo:lef}
Soit $R$ une fraction rationnelle de degr\'e~$d$,
qu'on \'ecrit $P/Q$, o\`u $P$ et $Q$ sont des polyn\^omes premiers
entre eux.
On suppose pour simplifier que $\deg (Q)=d$, c'est-\`a-dire
que $\infty$ n'est pas un point fixe de~$R$.
Les points fixes de~$R$ sont alors les $d+1$ z\'eros,
disons du polyn\^ome $zQ(z)-P(z)$; nous les noterons $z_0,\dots,z_d$. 

1) Si $z_i$ est un point fixe multiple, d\'emontrer
que son multiplicateur $R'(z_i)$ est \'egal \`a~$1$.

2) Si les~$z_i$ sont distincts deux \`a deux, \'ecrire
la formule d'interpolation de Lagrange pour le polyn\^ome
$Q$ et en d\'eduire la formule~\ref{eq.lef}.
\end{ex}

\medskip

L'ensemble de Julia  et  l'ensemble de Fatou sont totalement
invariants :
\[
R^{-1}(\Jul(R))=\Jul(R), \quad R^{-1}(\Fat(R))=\Fat(R).
\] 
Ils jouissent des propri\'et\'es suivantes.
La premi\`ere proposition est \`a opposer au cas complexe, puisque l'ensemble
de Julia peut, dans ce cas, \^etre \'egal \`a~$\P^1(\C)$ tout entier.
\begin{prop} Soit $R$ un \'el\'ement de $\C_p(z)$ de degr\'e~$\geq 2$.
Alors
\begin{enumerate}
\item l'ensemble de Julia est d'int\'erieur vide ;
\item si $R$ n'a pas de point exceptionnel,  $\Jul(R)$ n'a pas
de point isol\'e ;
\item l'ensemble de Fatou $\Fat(R)$ est  ouvert et dense.
\end{enumerate}
\end{prop}

La seconde proposition est l'analogue d'une propri\'et\'e bien connue de
dynamique holomorphe (\cite{Milnor:book}, theorem 4.10).

\begin{prop}
Soit~$V$ un ouvert pour la topologie faible
qui rencontre~$\Jul(R)$ mais ne contient pas de point
exceptionnel.
Il existe alors un entier~$N>0$ tel que $R^{N}(V)$
contienne~$\Jul(R)$ et la r\'eunion des~$R^n(V)$,
pour $n\geq 0$, est le compl\'ementaire des points exceptionnels
dans~$\BP(\C_p)$.
\end{prop}

\subsection{Structure de l'ensemble de Fatou et domaine de quasi-p\'eriodicit\'e}\label{par:FatouQP}

Soit $R$ une fraction rationnelle de degr\'e au moins~$2$.
On dit qu'un point~$z$ de~$\BP(\C_p)$ 
est \emph{r\'ecurrent} si l'adh\'erence de son orbite contient~$z$.
On appelle  \emph{domaine de quasi-p\'eriodicit\'e}~$\whE(R)$
l'int\'erieur de l'ensemble des points r\'ecurrents pour la topologie faible. 
Il est donc contenu dans l'ensemble de Fatou.
De plus, il v\'erifie les propri\'et\'es suivantes 
(\cite{Rivera-Letelier:Asterisque}, \S 4.2 et \S 5):
\begin{enumerate}
\item
c'est un ensemble invariant, $R_*(\whE(R))=\whE(R)$,
et la restriction de~$R$ \`a~$\whE(R)$
d\'efinit un hom\'eomorphisme de~$\whE(R)$ sur lui-m\^eme 
(mais, en g\'en\'eral,  $\whE(R)$ n'est pas totalement invariant) ;

\item  pour tout entier $n>0$, on a $\whE(R)=\whE(R^n)$ ;

\item l'intersection de~$\whE(R)$ et de~$\P^1(\C_p)$
est form\'e des points $z_0\in\P^1(\C_p)$
tels qu'il existe une boule ouverte~$B$ contenant~$z_0$
et une suite~$(n_i)$ d'entiers tels que $R^{n_i}$ converge
uniform\'ement vers l'identit\'e sur~$B$ ;

\item en particulier, l'ensemble~$\whE(R)$  
contient tout point p\'eriodique 
indiff\'erent de~$\P^1(\C_p)$ (voir le paragraphe~\ref{par:perio}) ;

\item une coupure $S_0$ appartient \`a~$\whE(R)$
si et seulement s'il existe un entier~$N>0$ et un nombre
r\'eel~$\delta>0$ tel que $R^N_*(S)=S$ pour toute coupure~$S$
telle que $\dist(S,S_0)<\delta$.
\end{enumerate}

\begin{eg} Si $R(z)= p^3 z^2 + z$, de sorte que $\bar R (z)= z$,
alors $S_\can$ appartient \`a~$ \whE(R)$ et satisfait la propri\'et\'e~(5) ci-dessus.
\end{eg}

\subsection{Composantes errantes}

Soit $C$ une composante connexe de l'ensemble de Fatou~$\Fat(R)$.
Alors, $R(C)$ est une composante connexe de~$\Fat(R)$
et $R^{-1}(C)$ est une r\'eunion finie de composantes connexes de~$\Fat(R)$.

Dans le cas complexe, Sullivan a d\'emontr\'e
que les composantes connexes de~$\Fat(R)$
sont pr\'ep\'eriodiques; autrement dit, {\og l'ensemble de Fatou n'a pas
de composante errante\fg} 
(voir~\cite{Milnor:book}, appendix F et les r\'ef\'erences qui s'y trouvent). 
En revanche, dans le cas~$p$-adique,
Benedetto a exhib\'e une famille de polyn\^omes dont l'ensemble de Fatou 
poss\`ede des composantes errantes
(voir~\cite{Benedetto:2002}). Il s'agit de polyn\^omes
\`a coefficients dans $\C_p$ (pas dans $\bar{\Q}_p$), que nous d\'ecrirons
au paragraphe \ref{par:Benedetto}. 

\begin{theo}[Rivera-Letelier]
Toute composante connexe p\'eriodique de~$\Fat(R)$ est 
soit contenue dans un bassin d'attraction d'une orbite
p\'eriodique attractive, soit une composante connexe du domaine de 
quasi-p\'eriodicit\'e.
\end{theo}

Rappelons que dans le cas complexe, une composante p\'eriodique de l'ensemble
de Fatou est soit un bassin d'attraction 
d'un point fixe attractif ou parabolique,
soit un disque de Siegel, soit un anneau d'Herman 
(voir~\cite{Milnor:book}, \S 15).
Les composantes connexes de $\whE(R)$ sont les analogues  
des disques de Siegel et anneaux d'Herman.

Le th\'eor\`eme suivant utilise la notion d'affino\"{\i}de ouvert 
et les notations introduites au paragraphe \ref{par:affinoide}.

\begin{theo}[Rivera-Letelier] Soit~$C$  une composante
 connexe de $\whE(R)$.  Alors $C$ est p\'eriodique. 
De plus, si $C$ est fixe, il existe un affino\"{\i}de ouvert 
$A\subset\P^1(\C_p)$ tel que 
\begin{enumerate}
\item $R(A)=A$, $\widehat A=C$ et $A$ est maximal dans~$\whE(R)$ ;
\item les extr\'emit\'es $S_i$ de $\widehat A$  sont des points p\'eriodiques \emph{r\'epulsifs} de~$R$ ; en particulier, les $S_i$ sont des coupures
rationnelles (voir la proposition \ref{pro:indifferent}).
\end{enumerate}
\end{theo}

Nous renvoyons le lecteur au chapitre~5 de~\cite{Rivera-Letelier:Asterisque}
pour les d\'emonstrations.
%%%%%%%%%%%%%%%%%%%%%%%%%%%%%%%%%%%%%%%
%%%%%%%%%%%%%%%%%%%%%%%%%%%%%%%%%%%%%%%

\section{Le cas des polyn\^omes}\label{par:poly}

%%%%%%%%%%%%%%%%%%%%%%%%%%%%%%%%%%%%%%%
%%%%%%%%%%%%%%%%%%%%%%%%%%%%%%%%%%%%%%%

Dans cette partie, nous pr\'ecisons la dynamique sur $\P^1(\C_p)$
induite par un polyn\^ome, puis nous d\'ecrivons les exemples de Benedetto. 
Soit~$P$ un \'el\'ement de $\C_p[z]$ de degr\'e $d>1$, que  nous \'ecrirons 
\[
P(z) = \sum_{i=0}^{d} a_i z^i.
\]

\subsection{Polyn\^omes ayant bonne r\'eduction}

Supposons que le polyn\^ome~$P$ a bonne r\'eduction, c'est-\`a-dire que
$\abs{a_d}= \max_i \abs{a_i}=1$ (voir l'exercice \ref{ex.pol-br})
et notons $\bar P$ sa r\'eduction.

D'apr\`es l'\'etude
de la dynamique des fractions rationnelles
ayant bonne r\'eduction,
 expos\'ee au paragraphe \ref{par:bonne-reduction},
l'ensemble de Julia est r\'eduit \`a~$\{ S_\can\}$ et 
deux cas peuvent se pr\'esenter 
suivant que le polyn\^ome~$\bar P$ est ins\'eparable ou non.

Lorsque le polyn\^ome $\bar P $ est  ins\'eparable (c'est-\`a-dire
lorsque sa d\'eriv\'ee est identiquement nulle), 
toutes les orbites p\'eriodiques de~$P$ situ\'ees dans $\P^1(\C_p) $ 
sont attractives. Le domaine de quasi-p\'eriodicit\'e est vide
et l'ensemble de Fatou est constitu\'e de bassins d'attraction.

Inversement, lorsque le polyn\^ome $\bar P$ est s\'eparable,
il n'existe qu'un nombre fini d'orbites p\'eriodiques attractives et toutes
les autres sont indiff\'erentes.

\subsection{Codage et hyperbolicit\'e} 

Le point \`a l'infini est un point fixe exceptionnel de~$P$.

\'Ecrivons~$P$ sous la forme 
\[
P(z) = z + a_d\prod_{i=1}^d (z-z_i)
\]
o\`u les $z_i$ sont les points fixes de~$P$ dans~$\C_p$, r\'ep\'et\'es
selon leur multiplicit\'e.  

Si $\abs{a_d}=1$ et $\abs{z_i}\leq 1$ pour tout~$i$,
alors $P$ a bonne r\'eduction.
Plus g\'en\'eralement, 
le changement de variables
 $z'=\lambda z$
transforme~$P$ en un polyn\^ome ayant bonne r\'eduction
d\`es que $\lambda$ v\'erifie $\abs{\lambda}^{d-1}\abs{a_d}=1$
et $\abs{z_i/\lambda}\leq 1$ pour tout~$i$.

Si $P$ n'a qu'un seul point fixe dans~$\C_p$, on peut supposer que
c'est~$0$, quitte \`a conjuguer~$P$ par une translation;
on en conclut que $P$ est conjugué au polynôme~$z+z^d$,
qui a bonne réduction.

Dans la suite, nous supposons que
\emph{le polyn\^ome~$P$ n'est pas conjugu\'e \`a un polyn\^ome
ayant bonne r\'eduction.} D'apr\`es l'exercice~\ref{ex:polsimple},
$P$ n'est  pas non plus conjugu\'e \`a une fraction rationnelle 
ayant bonne r\'eduction.
Le polynôme~$P$ a donc au moins deux points fixes;
quitte à conjuguer~$P$ par une homothétie, 
on peut donc supposer que 
$\max(\abs{z_i})= 1$. On a alors $\abs{a_d}>1$.

\medskip

On appelle \emph{ensemble de Julia rempli} de~$P$
l'ensemble $\JulK(P)$ des points de~$\C_p$ dont l'orbite est born\'ee.
C'est un ensemble totalement invariant de~$\C_p$.
Pour tout $z\in\C_p$ tel que  $\abs{z}>1$,   on a
$\abs{z-z_i}=\abs z$ pour tout~$i$, d'o\`u 
$\abs{P(z)}= \abs{a_d}\abs{z}^d$. Par suite, $P^n(z)\ra\infty$
si $\abs z>1$.
Cela entra\^{\i}ne que $\JulK(P)$ est contenu dans~$\Bf(0,1)$,
donc aussi dans~$P^{-n}(\Bf(0,1))$ pour tout entier~$n\geq 1$.
Puisque l'intersection des $P^{-n}(\Bf(0,1))$ est contenue dans
$\JulK(P)$, nous obtenons l'\'egalit\'e 
\begin{equation}
\JulK(P) = \bigcap_{n\geq 1}P^{-n}(\Bf (0,1)).
\end{equation}

D'apr\`es la proposition~\ref{prop:boules}, 
$P^{-n}(\Bf (0,1))$ est la r\'eunion d'une famille finie~$\Sigma_n$ 
de boules ferm\'ees deux \`a deux disjointes. 

Si~$B$ appartient \`a~$\Sigma_n$, il existe une unique boule~$B'$ 
appartenant \`a~$\Sigma_{n-1}$ qui contient~$B$ ; nous la noterons $i(B)$.
On d\'efinit ainsi une application $i:\Sigma_n\to \Sigma_{n-1}$. 
Soit $\Sigma_\infty$ la limite projective des $\Sigma_n$ 
(suivant les applications~$i$) ; 
un \'el\'ement $B_\infty$ de $\Sigma_\infty$ est donc une suite~$(B_n)$ de boules embo\^{\i}t\'ees,
o\`u $B_n\in \Sigma_n$ pour tout~$n$.

\`A toute boule $B\in \Sigma_n$ sont associ\'es deux nombres : 
\begin{enumerate}
\item le diam\`etre de~$B$, not\'e $\diam (B)$ ; c'est un \'el\'ement de $p^\Q$.
\item le degr\'e de $P_{\vert B}: B\to P(B)$, not\'e $\deg (B)$ ; c'est un entier inf\'erieur ou \'egal \`a~$\deg (P)$.
\end{enumerate}
Soit $B_\infty$ un point de $\Sigma_\infty$, 
d\'efini par une suite  $(B_n)$  de boules embo\^{\i}t\'ees,
avec $B_n\in\Sigma_n$ pour tout~$n$. 
Les suites $(\diam (B_n))$ et $(\deg (B_n))$ convergent
car elles sont d\'ecroissantes. Nous noterons $\diam(B_\infty)$
et $\deg(B_\infty)$ leurs limites.

Par ailleurs, si $B$ appartient \`a~$\Sigma_n$, alors $P(B)$ 
est une boule qui appartient \`a~$\Sigma_{n-1}$ ; 
si l'on note encore~$P$ l'application de~$\Sigma_n$ dans~$\Sigma_{n-1}$ ainsi
d\'efinie, on a la relation 
\[
P\circ i = i \circ P.
\]
Par cons\'equent,
le polyn\^ome~$P$ induit une application de~$ \Sigma_\infty $ 
dans lui-m\^eme, toujours not\'ee~$P$.

Soit $z$ un point de~$\JulK(P)$.
Pour tout entier $n\geq 1$,  $z\in P^{-n}(\Bf(0,1))$
donc il existe une unique boule~$B\in\Sigma_n$ telle que $z\in B$.
Notons $j_n(z)$ cette boule. On a $i(j_n(z))=j_{n-1}(z)$;
on peut donc d\'efinir $j(z)$ comme \'etant l'\'el\'ement $(j_n(z))$
de $\Sigma_\infty$.

\begin{rem}\label{rema.j-1}
Rappelons qu'\`a toute boule ferm\'ee ou irrationnelle~$B\subset\C_p$
est associ\'ee une coupure de~$\BP(\C_p)$,
correspondant \`a la semi-norme~$\nu_B(\cdot)=\sup_B \abs{\cdot} $ 
de l'espace de Berkovich  de~$\P^1({\C_p})$.
Alors, \`a une suite~$B_\infty=(B_n)$ de boules embo\^{\i}t\'ees 
correspond la semi-norme~$\nu_{B_\infty}=\inf \nu_{B_n}$ de~$\P^1_{\Ber}(\C_p)$
dont la nature d\'epend de l'intersection~$\bigcap B_n$.

On a  $j^{-1}(B_\infty)=\bigcap B_n$.
Lorsque $\diam(B_\infty)=0$, le
th\'eor\`eme des ferm\'es embo\^{\i}t\'es affirme que l'intersection des boules en
question est un singleton de~$\C_p$; c'est l'unique ant\'ec\'edent de~$B_\infty$
par~$j$.
En revanche, lorsque $\diam(B_\infty)>0$, deux cas sont possibles:
si l'intersection $\bigcap B_n$ est vide, 
le point $\nu_{B_\infty}$ est un point singulier de~$\BP(\C_p)$
et $B_\infty$ n'a pas d'ant\'ec\'edent par~$j$;
si cette intersection n'est pas vide, c'est une boule
ferm\'ee ou irrationnelle,  $\nu_{B_\infty}$ est le point correspondant
de~$\BP(\C_p)$, et tout point de cette boule est un ant\'ec\'edent de 
$B_\infty$ par~$j$.

En particulier, le lecteur prendra garde au fait que $B_\infty$ d\'esigne un point
de $\Sigma_\infty$, et pas une boule.
\end{rem}

\begin{defi}
On dit qu'un point $B_\infty$ de $\Sigma_\infty$ est un point critique si 
$\deg(B_\infty)>1$.
\end{defi}

\begin{rem}\label{rem:pt_crit}
1) Soit $z\in\C_p$ un point critique de~$P$ situ\'e dans~$\JulK(P)$ ; alors $B_\infty=j(z)$
est un point critique de~$\Sigma_\infty$. En effet, la restriction
de~$P$ \`a toute boule centr\'ee en~$z$ est de degr\'e~$\geq 2$.
%Antoine: c -> z

2) Inversement, soit $B_\infty\in\Sigma_\infty$ un point critique.
Si $\diam(B_\infty)=0$, 
$j^{-1}(B_\infty)$ est un point critique de~$P$. 
En revanche, si $\diam(B_\infty)>0$,
il est possible que $j ^{-1}(B_\infty)$ ne contienne pas de point critique,
m\^eme si l'intersection de la suite $(B_n)$ d\'efinissant~$B_\infty$ 
n'est pas vide.  
\end{rem}

L'application $j:\JulK(P)\to \Sigma_\infty$
n'est en g\'en\'eral ni injective ni surjective. 
Elle s'ins\`ere cependant dans un diagramme commutatif
\[
\xymatrix {
\JulK(P) \ar[r]^j \ar[d]^P & \Sigma_\infty  \ar[d]^P \\
\JulK(P) \ar[r]^j & \Sigma_\infty }
\]
qui fournit un mod\`ele symbolique (partiel, si $j$ n'est
pas un hom\'eomorphisme) pour la dynamique de~$P$ sur~$\JulK(P)$.

\begin{defi}
On dit que le polyn\^ome~$P$ est \emph{hyperbolique de type Cantor}
s'il existe~$c>1$ et~$n\geq 1$ 
tels que 
\[
\abs{(P^n)'(z)}\geq c
\]
pour tout point~$z$ de l'ensemble $\JulK(P)$. 
\end{defi}

\begin{rem}La terminologie provient de l'analogie avec le cas complexe:
si la propri\'et\'e analogue est v\'erifi\'ee,
$P$ est \emph{expansive} sur son ensemble de Julia
rempli, et est donc \emph{hyperbolique} ;  les points critiques sont alors
attir\'es par l'infini et l'ensemble de Julia
est un ensemble de Cantor (voir par exemple \cite{Beardon:Book}, \S 9.8). 
\end{rem}

Supposons que $P$ soit hyperbolique de type Cantor.
Dans ce cas, les diam\`etres des boules $B_n$
tendent vers $0$, si bien que $\JulK(P)$ est compact
et que $j$ est une application bijective (voir par
exemple \cite{Bezivin:2004b}, propositions~15 et~16).
L'application $j$ est donc un hom\'eomorphisme qui conjugue
la dynamique de $P$ sur $\JulK(P)$ \`a celle du d\'ecalage
$P:\Sigma_\infty \to \Sigma_\infty$. 

\begin{prop}
Soit $P\in \C_p[z]$ un polyn\^ome de degr\'e sup\'erieur ou \'egal \`a~$2$. 
Les deux propri\'et\'es suivantes sont \'equivalentes. 
\begin{enumerate}\def\theenumi{\roman{enumi}}\def\labelenumi{(\theenumi)}
\item Le polyn\^ome~$P$ est hyperbolique de type Cantor.
\item Le polyn\^ome~$P$ n'est pas simple et $\Sigma_\infty$ ne contient pas
de point critique.
\end{enumerate} 
Lorsqu'elles sont v\'erifi\'ees, 
l'application~$j$ est un hom\'eomorphisme et la dynamique de~$P$ est
conjugu\'ee par~$j$ \`a celle d'un \emph{d\'ecalage}.
\end{prop}

Esquissons la preuve de l'\'equivalence entre les propri\'et\'es (i) et (ii).

\begin{proof}[Esquisse de d\'emonstration]
Lorsque $P$ est hyperbolique de type Cantor, $P$ ne peut avoir de point 
critique dans $\JulK(P)$. 
Le deuxi\`eme point de la remarque~\ref{rem:pt_crit} entra\^{\i}ne alors qu'il n'y a pas de point critique
dans~$\Sigma_\infty$. De plus, $P$ n'est pas simple, car sinon
$P$ aurait des points p\'eriodiques attractifs ou indiff\'erents, en contradiction
avec l'hypoth\`ese $\abs{(P^n)'(z)}\geq c > 1$ sur $\JulK(P)$. 

Supposons inversement que $P$ ne soit pas simple
et que $\Sigma_\infty$ ne contienne pas
de point critique. L'absence de point critique assure l'existence
d'un entier positif $n$ tel que $P$ soit de degr\'e $1$ en restriction
\`a chaque boule $B$ de $\Sigma_n$ ($\Sigma_n$ param\`etre les
boules constituant $P^{-n}(\Bf (0,1))$). 
Soit $B$ une boule de $\Sigma_{n}$. Ses pr\'eimages par $P$ sont
des boules $B'$ sur lesquelles $P:B'\to B$ est de degr\'e $1$ ; 
ainsi, $P:B'\to B$ dilate  les
distances d'un facteur $c=\diam (B)/\diam (B')$ (lemme \ref{lemm.schwarz}).
Puisque $P$ n'est pas simple, on montre que $c>1$, pour tout choix
du couple $(B,B')$ dans $\Sigma_{n}\times  \Sigma_{n+1}$ et $P(B')=B$.
It\'erant cette remarque, on observe que les diam\`etres des \'el\'ements
de $\Sigma_{n+k}$ tendent exponentiellement vite vers $0$. 
On en d\'eduit alors que $\JulK(P)$ est un sous-ensemble
compact de $\P^1(\C_P)$ sur lequel $j$ est une bijection, puis que
$P$ est hyperbolique (voir par exemple \cite{Bezivin:2004b}, prop.~17).
\end{proof}

\begin{ex}
(Voir~\cite{Rivera-Letelier:Asterisque}, exemple~6.1,
ainsi que~\cite{Benedetto:2002b}, exemple~7, p.~253 
et~\cite{Hsia:London}, exemple~4.1, p.~699.)
Soit~$P$ le polyn\^ome d\'efini par 
$
P(z)= (z-z^p)/p
$.
Le but de cet exercice est de d\'emontrer que $\JulK(P)=\Z_p$ et 
que le polyn\^ome~$P$ est  hyperbolique de type Cantor.
Notons~$W$ le bassin d'attraction de $\infty$.

\begin{enumerate}
\item 
V\'erifier que les points fixes de~$P$ distincts de l'origine sont
de valeur absolue \'egale \`a~$1$,
conform\'ement \`a la normalisation employ\'ee dans ce paragraphe. 

\item Montrer que $\abs{P'(z)}= p$ si $\abs{z}\leq 1$; en d\'eduire que 
l'\'egalit\'e $\JulK(P)=\Z_p$ entra\^{\i}ne que $P$ est hyperbolique de
type Cantor. 

\item D\'emontrer que  $P(\Z_p)$ est contenu dans $\Z_p$ ; en d\'eduire
que $\Z_p$ est inclus dans $\JulK(P)$. Il s'agit donc d\'esormais de d\'emontrer
l'inclusion $\JulK(P)\subset \Z_p$.

\item D\'emontrer que le bassin~$W$ contient $\{\abs{z} > 1\}$. 

\item  Soit $w\in \C_p$ satisfaisant 
$\abs{w-i}=1$ pour tout~$i\in\{0, 1, \dots, p-1\}$. D\'emontrer 
que $\abs{w-w^p}=1$ puis en d\'eduire que~$w\in W$. 

\item Soit~$i$ un \'el\'ement de $\{0, 1, \dots, p-1\}$. Soit 
$B_i$ la boule ouverte de centre~$i$ et de rayon~$1$. Montrer que 
$P(B_i)$ co\"{\i}ncide avec $\Bo (p)= \{\abs z<p\}$ et que l'application induite 
$P: B_i \to  \Bo(p)$ est bijective. Dans la suite, nous noterons 
$P_i ^{-1}:   \Bo(p)\to B_i$ son inverse.

\item D\'emontrer par r\'ecurrence sur~$k$ que 
\[
\JulK(P) \subset \bigcup_{i_0,\dots, i_k} P_{i_0}^{-1}\circ \dots \circ P_{i_k}^{-1}(\Bo (p)) .
\]
\item Montrer que 
\[
P_{i_0}^{-1}\circ \dots \circ P_{i_k}^{-1}(\Bo (0,p))=\left\{z \; ; \; \, \abs{z-(i_0+ \dots + i_k p^k)}< p^{-k}\right\}.
\]
\item D\'eduire des questions pr\'ec\'edentes que~$\JulK(P)=\Z_p$.
\end{enumerate}
\end{ex}

\Subsection{Points p\'eriodiques et disques errants}

\subsubsection{Orbites finies (voir ~\cite{Rivera-Letelier:compositio} \S5.2)}
Soit   $B_\infty$ un point de $\Sigma_\infty$ qui est p\'eriodique de p\'eriode~$T$ 
pour l'application~$P\colon\Sigma_\infty\ra\Sigma_\infty$. 
Posons 
\[
D= \prod_{i=0}^{T-1} \deg(P^i(B_\infty)).
\]
Soit $(B_n)$ la suite de boules embo\^{\i}t\'ees correspondant \`a~$B_\infty$;
par d\'efinition, $P(B_\infty)$ est la suite de boules
$(P(B_{n+1}))$ et on a vu que les applications~$P\colon B_n\ra P(B_n)$
sont surjectives. On a donc $P^T(B_n)=B_{n-T}$ pour tout~$n\geq T$.

Soit $B$ l'ensemble $j^{-1}(B_\infty)$. On a
expliqu\'e dans la remarque~\ref{rema.j-1}
qu'il s'agit d'un point, d'une boule 
ou de l'ensemble vide suivant que le point
correspondant de~$\BP(\C_p)$ 
est un point classique de~$\P^1(\C_p)$, 
une coupure, ou un point singulier.
En tant que point de l'espace de Berkovich,
ce point est p\'eriodique.

S'il s'agit d'un point classique, ce point
est r\'epulsif par construction, et l'on a donc $D=1$.
De m\^eme, les points p\'eriodiques de $R$ dans $\Hp$ 
qui sont singuliers ou irrationnels ont un degr\'e local 
\'egal \`a~$1$ (voir le paragraphe \ref{subsec.action-coupures}). 
Ainsi, lorsque $D>1$, $B$ est une boule ferm\'ee et celle-ci d\'etermine une
coupure rationnelle. 
Cette boule est stable par~$P^T$ et $D$ est le degr\'e de 
l'application~$P^T\colon B\ra P^T(B)$.  
Si $r$ est le diam\`etre de~$B$, celui de~$P^T(B)$ est de la forme~$\abs a r^D$,
o\`u $a\in\C_p^*$ ; on a donc bien $r\in p^\Q$.

\subsubsection{Orbites infinies}
Soit maintenant $B_\infty$ un point de~$\Sigma_\infty$
qui n'est pas pr\'ep\'eriodique. En 
supposant que son diam\`etre est strictement positif, 
deux cas sont possibles:
ou bien $j^{-1}(B_\infty)$ est vide, 
ou bien $j^{-1}(B_\infty)$ est un disque errant.
Si cette derni\`ere situation se produit, elle fournit alors
un {\og contre-exemple\fg} \`a l'analogue du th\'eor\`eme de Sullivan. 

\begin{lemm}\label{lemm.disqueserrants}
Soit $P$ un polyn\^ome \`a coefficients dans $\C_p$ qui n'est pas simple. 
Soit $B_\infty$ un \'el\'ement de $\Sigma_\infty$ qui n'est pas pr\'ep\'eriodique. 
Si $j^{-1}(B_\infty)$ n'est pas vide, la limite inf\'erieure de 
la suite $\diam (P^n(j^{-1}(B_\infty)))$ est \'egale \`a~$0$.
Si, de plus, les coefficients de~$P$ appartiennent \`a~$\bar{ \Q_p}$, 
cette suite tend vers~$0$.
\end{lemm}
\begin{proof}[Esquisse de d\'emonstration
(voir  \cite{Rivera-Letelier:Asterisque}, lemme 4.29 pour plus de d\'etails)]
Le cas d'un singleton \'etant \'evident, on suppose que $j^{-1}(B_\infty)$
est une boule ferm\'ee ou irrationnelle.
Posons $B_n=P^n(j^{-1}(B_\infty))$, $d_n=\diam(B_n)$, et
observons que les boules~$B_n$ sont deux \`a deux disjointes,
car $B_\infty$ n'est pas pr\'ep\'eriodique.

Commen\c{c}ons par traiter  le cas particulier
o\`u les coefficients de~$P$ appartiennent
\`a~$\bar{\Q_p}$. Soit
donc $K$ une extension finie de~$\Q_p$ 
qui contient les coefficients de~$P$ 
et au moins un point de $j^{-1}(B_\infty)$. 
Pour tout~$n$, $B_n\cap K$ est une boule de~$K$ dont le diam\`etre
est au moins celui de~$B_n$ divis\'e par~$p$.
Si la suite~$(d_n)$ ne tendait pas vers~$0$, on pourrait alors
extraire de la suite~$(B_n\cap K)$ une suite infinie
de boules disjointes 
et de rayons minor\'es dans~$K$,
en contradiction avec le caract\`ere localement compact du corps~$K$.
Nous avons ainsi d\'emontr\'e que $d_n$ tend vers~$0$.

Dans le cas g\'en\'eral, on raisonne encore par l'absurde 
en supposant qu'il existe
$c>0$ tel que $d_n>c$ pour tout entier~$n\geq 0$.
Soit~$Q$ un polyn\^ome \`a coefficients dans 
$\bar \Q_p$ satisfaisant 
\[
\sup_z \delta( P(z),Q(z)) \leq c/2,
\]  o\`u 
$\delta$ d\'esigne la distance sph\'erique. 
Soit $D$ une boule dont l'image par~$P$ a un diam\`etre sup\'erieur \`a~$c$ ;
on montre alors que les deux boules $Q(D)$ et $P(D)$ s'intersectent 
et ont le m\^eme diam\`etre, donc que $P(D)=Q(D)$ 
(voir l'exercice \ref{ex:boules-et-ptfixes}).
En particulier, $B_n=P^n(j^{-1}(B_\infty))=Q^n(j^{-1}(B_\infty))$
pour tout~$n$.
Le raisonnement utilis\'e dans le cas o\`u les
coefficients de~$P$ sont alg\'ebriques sur~$\Q_p$ fournit
alors la contradiction recherch\'ee.
\end{proof}

\subsection{Un exemple de Rivera-Letelier}\label{par:exempleRL}

L'exemple que nous traitons ici est tir\'e de~\cite{Rivera-Letelier:Asterisque};
il s'agit d'un polyn\^ome pour lequel l'image de l'application~$j$
est l'ensemble d\'enombrable des points pr\'ep\'eriodiques de~$\Sigma_\infty$.
Comme $\Sigma_\infty$ a la puissance du continu,  on voit
que l'application~$j$ est loin d'\^etre surjective.

 Soit~$P$ le polyn\^ome d\'efini par 
\[
P(z) = \frac{1}{p}\left( z^p - z^{p^2}\right).
\]
Comme $0$ est fixe par~$P$ et $P(1)=0$, les point~$0$ et~$1$
appartiennent \`a~$\JulK(P)$. Par ailleurs, 
\begin{align*}
\abs{P(z)} &= p \abs{z}^{p^2} & \,  & \text{si } \abs{z} > 1 \ ; \\
\abs{P(z)} &= p \abs{z}^p \, \, \, &  \, & \text{si } \abs{z} < 1.
\end{align*}
On en d\'eduit que tout point~$z$ tel que $\abs z>1$
appartient au bassin d'attraction~$W_\infty$ de l'infini,
et que tout point~$z$ tel que $\abs z<p^{-1/(p-1)}$
appartient au bassin d'attraction~$W_0$ de~$0$.
En particulier,
 $\JulK(P)$ est contenu dans la boule ferm\'ee de centre~$0$ et de rayon~$1$.
Le point~$z=0$ est un point critique de~$P$; les autres points
critiques~$z$ satisfont 
$1=pz^{p+1}$ ; ils ont donc une valeur absolue
$>1$ et appartiennent \`a~$W_\infty$.

\medskip

\begin{lemm}
Soit~$B$ une boule contenue dans~$\Bf(0,1)$
dont le diam\`etre v\'erifie 
\[
\diam(B)>p^{-1/(p-1)}.
\]
L'ensemble $P^{-1}(B)$ est une r\'eunion finie de~$p$ 
boules deux \`a deux disjointes~$(B_i)_{0\leq i\leq p-1}$
telles que 
\begin{enumerate}
\item le  diam\`etre de chaque~$B_i$ est \'egal \`a~$(\diam(B)/p)^{1/p}$;
\item pour tout $1\leq i \leq p$, la boule~$B_i$ est contenue dans~$\Bf(i,p^{-1/p})$. 
\end{enumerate}
\end{lemm}
\begin{proof}
Soit $w\in\C_p$ tel que $\abs w\leq 1$.
En d\'eveloppant
\[ P(z+w)-P(w)=P(z)+\sum_{k=1}^{p-1} \frac1p \binom pk w^{p-k}z^k - \sum_{k=1}^{p^2-1}
  \frac1p \binom{p^2}k w^{p^2-k}z^k, \]
on constate que l'on a
\[ \abs{P(z+w)-P(w)}=\abs{P(z)}=p\abs z^p \qquad\text{si $p^{-1/(p-1)}<\abs z<1$.} \]

Appliquons la proposition~\ref{prop:boules} \`a
l'ensemble $P^{-1}(\Bf(0,1))$: c'est la r\'eunion d'une famille
finie $(D_i)$ de boules disjointes telles que l'application
induite $P\colon D_i\ra\Bf(0,1)$ soit surjective. Ces boules
sont contenues dans $\Bf(0,1)$, et 
d'apr\`es le calcul qui vient d'\^etre pr\'esent\'e, 
ces boules sont de diam\`etre~$p^{-1/p}$.
Les solutions non nulles de l'\'equation~$z^{p^2}-z^p=0$ sont
les racines $(p^2-p)$-i\`emes de l'unit\'e et pour chacune d'entre
elles, il existe un unique entier~$i$ de~$\{1,\ldots,p-1\}$
tel que $\abs {z-i}=p^{-1/p}$. 
Il en r\'esulte que $P^{-1}(\Bf(0,1))$ est la r\'eunion
disjointe des boules $\Bf(i,p^{-1/p})$, pour $0\leq i<p$,
et $P\colon\Bf(i,p^{-1/p})\ra \Bf(0,1)$ est surjective.
Quitte \`a permuter les $D_i$, nous pouvons donc supposer
que $D_i= \Bf(i,p^{-1/p})$ pour tout $0\leq i \leq p-1$.

Soit maintenant $B$ une boule v\'erifiant les hypoth\`eses du lemme.
D'apr\`es la proposition~\ref{prop:boules}, l'ensemble $P^{-1}(B)$ est la
r\'eunion disjointe d'un nombre fini de boules $B_j$, $1\leq j \leq s$,
telles que $P(B_j)=B$. Puisque chaque application 
$P: D_i \to \Bf(0,1)$ est
surjective, le nombre $s$ de boules $B_j$ est au moins \'egal \`a~$p$.

Soit $c$ un centre de $B$ pour lequel $\abs c > p ^{-1/(p-1)}$.  
Soit $w_j$ un ant\'ec\'edent de $c$ situ\'e dans~$B_j$. Le calcul 
pr\'ec\'edent et la preuve de la proposition~\ref{prop:boules} montrent que $B_j$
co\"{\i}ncide avec la boule de centre $w_j$ et de rayon $(\diam (B)/p)^{1/p}$. 

Observons aussi que $P(\xi z) = P(z)$ 
pour toute racine~$p$-i\`eme de l'unit\'e~$\xi$. 
Comme $\abs{\xi-1}=p^{-1/(p-1)}<p^{-1/p}$,  chacune
des boules $B_j$ est stable par multiplication par~$\xi$
(autrement dit, $\xi B_j=B_j$ pour tout~$j$)
et la restriction \`a~$B_j$ du polyn\^ome~$P$ est
donc au moins de degr\'e~$p$. Par suite, $s\leq p$.

On a ainsi $s=p$ et l'on peut indexer les boules~$B_i$
de sorte que $B_i\subset D_i= \Bf(i,p^{-1/p})$. 
\end{proof}

Comme on l'a vu dans la d\'emonstration, $P\colon B_i\ra B$
est de degr\'e~$p$
et $\xi B_i=B_i$ pour toute racine $p$-i\`eme
de l'unit\'e~$\xi$. Notons en outre que 
\[ (\diam(B)/p)^{1/p}> (p^{-p/(p-1)})^{1/p} = p^{-1/(p-1)}. \]

Partant de la boule~$\Bf(0,1)$,
on peut donc appliquer ce lemme par r\'ecurrence aux boules de $P^{-n}(\Bf(0,1))$.
On en d\'eduit que pour tout entier~$n\geq 0$, l'ensemble $P^{-n}(\Bf(0,1))$
est la r\'eunion de boules deux \`a deux disjointes
$B_{a_0, a_1, \dots, a_{n-1}}$, 
index\'ees par l'ensemble des suites $(a_0,\dots,a_{n-1})$
d'entiers~$a_i$ tels que $a_i\in \{0, 1, \dots , p-1\}$. Ces
boules ont toutes le m\^eme diam\`etre, \`a savoir  
\[
p^{-(1/p+1/p^2+ \dots+ 1/p^n)}= p^{-(1-p^{-n})/(p-1)}.
\]
Lorsque $n$ tend vers~$+\infty$,
ce diam\`etre tend donc vers $p^{-1/(p-1)}$.

L'ensemble $\Sigma_\infty$ s'identifie alors \`a l'ensemble~$\{0,\dots,p-1\}^\N$
des suites infinies $a_\infty=(a_0, a_1 ,\dots)$ o\`u chaque $a_n$ appartient 
\`a~$\{0, 1, \dots, p-1\}$ ; \`a une telle suite $a_\infty$ correspond
une suite~$B_\infty=(B_n)$ de boules embo\^{\i}t\'ees telle que 
$P^{n-1}(B_n)=\Bf(a_n,p^{-1/p})$. 
L'intersection~$j^{-1}(B_\infty)$ de ces boules est soit vide, 
soit une boule de rayon $p^{-1/(p-1)}$.
Nous noterons $B_\infty \mapsto a(B_\infty)$
cette bijection entre $\Sigma_\infty$ et ~$\{0,\dots,p-1\}^\N$.

Consid\'erons d'abord une suite $B_\infty\in \Sigma_\infty$, dont le
codage associ\'e  $a(B_\infty)$ est
p\'eriodique de p\'eriode~$T$. On a vu dans ce cas que~$j^{-1}(B_\infty)$
est une boule. En outre, utilisant le fait que $P^T$
applique $B_{n+T}$ sur~$B_n$ mais que le diam\`etre de~$B_{n+T}$
est strictement inf\'erieur \`a celui de~$B_n$, on d\'emontre
(voir l'exercice \ref{ex:boules-et-ptfixes} ou \cite{Rivera-Letelier:Asterisque}, lemme 4.13)
que l'\'equation $P^T(z)=z$ a une solution dans~$B_{n+T}$.
Autrement dit, chacune des boules~$B_n$ contient un point
de p\'eriode~$T$ et il en est de m\^eme de leur intersection~$j^{-1}(B_\infty)$.
La boule $j^{-1}(B_\infty)$ contient donc un point p\'eriodique de p\'eriode~$T$.

Plus g\'en\'eralement, si $B_\infty$ est pr\'ep\'eriodique, $j^{-1}(B_\infty)$
est encore une boule ferm\'ee qui contient un point pr\'ep\'eriodique.

Soit enfin $B_\infty$ une suite qui n'est pas pr\'ep\'eriodique.
Comme le diam\`etre des disques~$B_n$ est minor\'e par~$p^{-1/(p-1)}$,
il r\'esulte du lemme~\ref{lemm.disqueserrants} que $j^{-1}(B_\infty)$
est vide.

Nous avons ainsi d\'emontr\'e que \emph{l'image de $\JulK(P)$ par~$j$
est \'egale \`a l'ensemble des points
pr\'ep\'eriodiques de~$\Sigma_\infty$.  }
Comme l'ensemble $\Sigma_\infty$ a la puissance du continu
et que l'ensemble des points pr\'ep\'eriodiques est d\'enombrable,
la majeure partie des \'el\'ements de~$\Sigma_\infty$ n'a pas d'ant\'ec\'edent
par~$j$ dans~$\JulK(P)$.

\begin{rem}
Comme promis, cela construit un ensemble non d\'enombrable de points
singuliers dans~$\Hp$.
\end{rem}

\subsection{L'exemple de Benedetto}\label{par:Benedetto}

Nous terminons ce texte en esquissant 
la construction, due \`a R.~Benedetto, d'un polyn\^ome~$P\in\C_p[z]$
poss\'edant des disques errants
(voir~\cite{Benedetto:2002,Benedetto:2006}). 

\begin{theo}[Benedetto]
Pour $\lambda\in\C_p$, soit $P_\lambda$ le polyn\^ome de~$\C_p[z]$ donn\'e par
\[
P_\lambda(z)= \lambda z^p + (1-\lambda) z^{p+1}.
\]
Il existe  un \'el\'ement  $\lambda$ de $\C_p$ v\'erifiant $\abs\lambda>1$
et tel que $P_\lambda$ a un disque errant 
qui ne soit pas attir\'e par un cycle attractif. 
\end{theo}

\begin{proof}[Esquisse de d\'emonstration] 
La d\'emonstration consiste \`a produire un param\`etre~$\lambda$
et un \'el\'ement~$z$ de~$\JulK(P_\lambda)$
dont la dynamique symbolique ne soit pas pr\'ep\'eriodique. Un tel
\'el\'ement appartiendra \`a un disque errant.

Supposons $\abs\lambda>1$ et posons $\rho=\abs\lambda^{-1/(p-1)}$.

Lorsque $\abs z > 1$, $\abs{P_\lambda(z)} > \abs \lambda \abs{z}^p$, si bien que
le compl\'ementaire de $\Bf(0,1)$ est contenu dans le bassin d'attraction de l'infini
pour $P_\lambda$.

\medskip

\noindent
\emph{Dynamique sur la boule de rayon $1$}. --- 
Analysons maintenant le comportement de~$P_\lambda$
sur $\Bf(0,1)$. L'origine~$0$ est un point fixe super-attractif. De plus,
si $\abs z<1$, on a $\abs{P_\lambda(z)}=\abs\lambda \abs z^p$.
On en d\'eduit que le bassin d'attraction de~$0$
contient la boule~$\Bo(0,\abs{\lambda}^{-1/(p-1)})=\Bo(0,\rho)$.

En revanche, lorsque $\rho<\abs z<1$ on a 
$\abs{P_\lambda(z)}>\abs z$,
de sorte que les orbites de~$P_\lambda$ sur cette couronne
s'\'eloignent de l'origine. 

Au voisinage de~$1$, la situation est plus simple.
En effet, $P_\lambda(1)=1$ et 
\[
P_\lambda'(1)=p\lambda-(p+1)(1-\lambda)
\]
est de module~$\abs\lambda$; autrement dit, $1$ est un point
fixe r\'epulsif. Plus g\'en\'eralement, $P_\lambda'$
est de module constant~$\abs\lambda$ sur le disque~$\Bo(1,1)$.

\begin{ex}

1)  Montrer que 
$\abs{P'(z)} < 1 $
pour tout point~$z$ satisfaisant $\abs z = \rho$.

2) Montrer que 
$\abs{P'(z)}=\abs\lambda $
pour tout point~$z$  satisfaisant $\abs z = 1$.
En particulier, la valeur absolue de $P'$ est \'egale \`a
$\abs \lambda$ sur la boule $\Bo(1,1)$. 
\end{ex}

\medskip

\noindent\emph{Dynamique symbolique. ---}
L'image r\'eciproque de la boule~$\Bf(0,1)$ par~$P_\lambda$ est la r\'eunion
disjointes de deux boules
\begin{itemize}
\item $B_0=\Bf(0,\lambda^{-1/p})$ sur laquelle~$P_\lambda$
est de degr\'e~$p$,
\item  $B_1=\Bf(1,1/\abs\lambda)$
sur laquelle $P_\lambda$ est de degr\'e~$1$ ; 
\end{itemize}
chacune de ces boules est appliqu\'ee surjectivement sur~$\Bf(0,1)$ par~$P_\lambda$.
 
Par r\'ecurrence, $P_\lambda^{-n}(\Bf(0,1))$ est r\'eunion
d'un ensemble fini de boules disjointes dont l'une, 
not\'ee~$B_0^{(n)}$ est centr\'ee en~$0$ et une autre,
not\'ee $B_1^{(n)}$ est centr\'ee en~$1$. En outre,
$\diam(B_1^{(n)})=\abs\lambda^{-n}$ et
\[
r_n:=\diam(B_0^{(n)})=\abs{\lambda}^{-(p^n-1)/p^n(p-1)}.
\]
La suite $r_n$ d\'ecro\^{\i}t et tend vers~$\abs\lambda^{-1/(p-1)}=\rho$ lorsque~$n$ 
tend vers~$+ \infty$.

\medskip

\noindent\emph{Strat\'egie. ---}
La strat\'egie de la construction consiste \`a 
construire un \'el\'ement~$B_\infty$ de~$\Sigma_\infty$
qui n'est pas pr\'ep\'eriodique et
pour lequel $j^{-1}(B_\infty)$ est une boule (de rayon
strictement positif). 

Pour cela il s'agit de construire un point $z$, et un param\`etre $\lambda$, tels
que l'orbite de $z$ 
s'approche de $\Bo(0,\rho)$ pendant un temps~$n_1$, puis 
en ressort n\'ecessairement pour arriver dans le disque autour de~$1$
au temps~$n_2$, y retourne au temps~$n_3$, ..., et ainsi de suite. 

Le point~$B_\infty$ de~$\Sigma_\infty$ correspond alors \`a la suite
de boules $(B_{a_n}^{(n)})$, o\`u $a_n=0$ jusque~$n_1$,
puis $a_n=1$ jusque~$n_2$, puis $a_n=0$ jusque~$n_3$, etc.
Pour que $j^{-1}(B_\infty)$  soit effectivement un disque errant, 
on doit emp\^echer la suite~$(P_\lambda^n(B_\infty))$ d'\^etre pr\'ep\'eriodique
et s'assurer que $\diam(B_\infty)$ est strictement positif. L'ensemble
$j^{-1}(B_\infty)$ ne sera pas vide car
il contient le point $z$ ; ce sera donc bien une boule de rayon strictement positif.

Nous passons sous silence la condition effectivement
impos\'ee sur la suite~$(n_k)$; disons juste
que la suite $(n_{2k}-n_{2k-1})$ tend vers l'infini, mais pas trop vite,
de sorte que le point~$z$ passe des temps de plus en plus grands pr\`es de~$1$
pour emp\^echer la pr\'ep\'eriodicit\'e, mais pas trop pour que le diam\`etre
de $B_\infty$ soit strictement positif.

\medskip

\noindent\emph{Construction.} ---
Revenons maintenant \`a la construction de Benedetto.
On commence par choisir~$z$ de sorte que
$P_\lambda^{n_0}(z)=1$ mais $P_\lambda^{n_0-1}(z)\neq 1$.
Ce point~$z$ sera d\'esormais fix\'e; c'est le param\`etre~$\lambda$
qu'on va modifier.

De m\^eme qu'on a facilement pu analyser  le comportement de~$P_\lambda$
dans la couronne centr\'ee en~$0$ et dans le disque~$\Bo(1,1)$,
il est relativement simple, quoique technique,
d'\'evaluer l'effet d'une perturbation de~$\lambda$. Posant $\lambda_0=\lambda$,
cela permet de modifier~$\lambda_0$ en~$\lambda_1$ de sorte que 
$\abs{\lambda_1-\lambda_0}$ soit petit,
et que l'orbite de~$z$ par~$P_{\lambda_1}$
soit convenable jusqu'\`a l'instant~$n_1$, avec $P_{\lambda_1}^{n_1}(z)=0$.
On perturbe de m\^eme~$\lambda_1$ en~$\lambda_2$
pour que l'orbite de~$z$ par~$P_{\lambda_2}$
soit convenable jusqu'\`a l'instant~$n_2$, avec $P_{\lambda_2}^{n_2}(z)=0$,
et ainsi de suite.
En choisissant judicieusement les $n_i$, Benedetto montre que
la suite~$(\lambda_n)$ est de Cauchy et que sa limite $\lambda_\infty$
est un param\`etre qui convient. 
Nous renvoyons \`a l'article~\cite{Benedetto:2006} pour les d\'etails.
\end{proof}

La preuve pr\'ec\'edente ne fournit pas d'exemple avec $\lambda\in \bar \Q_p$
car, comme l'indique Benedetto dans~\cite{Benedetto:2006},  
$\lambda_\infty$ ne peut appartenir \`a aucun sous-corps de~$\C_p$
dont la valuation  soit discr\`ete.
En outre, les points critiques de~$P_\lambda$ 
sont $0$, $\infty$ et $c=p\lambda/(p+1)(\lambda-1)$;
les deux premiers sont superattractifs donc appartiennent \`a l'ensemble
de Fatou et le troisi\`eme appartient souvent au bassin d'attraction 
de~$0$ (si $\abs\lambda<p^{p-1}$) ou de~$\infty$ (si $\abs\lambda>p^p$).
Dans ces deux cas, il r\'esulte des r\'esultats de~\cite{Benedetto:2000}
qu'il n'y a pas de domaine errant si $\lambda$ est alg\'ebrique.

\`A l'heure actuelle, on ne sait pas si une fraction rationnelle 
\`a coefficients alg\'ebriques
peut avoir un disque errant qui ne soit pas attir\'e 
par un point p\'eriodique attractif.  D'apr\`es le th\'eor\`eme~1.2
d\'emontr\'e par Benedetto dans~\cite{Benedetto:2000},
une telle fraction rationnelle poss\`ede un point critique r\'ecurrent
dont l'indice de ramification est divisible par~$p$. Plus pr\'ecis\'ement,
la d\'emonstration de cet article prouve le r\'esultat suivant.

\begin{theo}[Benedetto]% \footnote{V\'ERIFIER}
Soit~$P$ un \'el\'ement de $\bar \Q_p[z]$ de degr\'e sup\'erieur ou \'egal \`a~$2$.
Soit $B_\infty$ un \'el\'ement \emph{non pr\'ep\'eriodique} 
de $\Sigma_\infty$ pour lequel $j^{-1}(B_\infty)$ 
n'est pas vide. Il existe alors une sous-suite $(P^{n_k}(B_\infty))$ 
qui converge 
vers un point critique $z$ de $P$ satisfaisant :
\begin{enumerate}
\item $z$ est r\'ecurrent ($z$ est adh\'erent \`a son 
orbite positive $\{P^n(z), \, n> 0\}$) ;

\item l'indice de ramification de $P$ en $z$ est divisible par $p$.
\end{enumerate}
\end{theo}

\bibliographystyle{smfplain}
 \bibliography{yoccoz-smf}

\providecommand{\noopsort}[1]{}\providecommand{\url}[1]{\textit{#1}}
\providecommand{\bysame}{\leavevmode ---\ }
\providecommand{\og}{``}
\providecommand{\fg}{''}
\providecommand{\smfandname}{\&}
\providecommand{\smfedsname}{\'eds.}
\providecommand{\smfedname}{\'ed.}
\providecommand{\smfmastersthesisname}{M\'emoire}
\providecommand{\smfphdthesisname}{Th\`ese}
\begin{thebibliography}{10}

\bibitem{amice75}
{\scshape Y.~Amice} -- \emph{Les nombres $p$-adiques}, Collection SUP: Le
  Math\'ematicien, vol.~14, Presses Universitaires de France, Paris, 1975.

\bibitem{Baker:2008}
{\scshape M.~Baker} -- {\og An introduction to {B}erkovich analytic spaces and
  non-{A}rchimedean potential theory on curves\fg}, in \emph{{$p$}-adic
  geometry}, Univ. Lecture Ser., vol.~45, Amer. Math. Soc., Providence, RI,
  2008, p.~123--174.

\bibitem{baker-rumely2010}
{\scshape M.~Baker {\normalfont \smfandname} R.~Rumely} -- \emph{Potential
  theory and dynamics on the {B}erkovich projective line}, AMS Surveys and
  Mathematical Monographs, vol. 159, 2010.

\bibitem{Beardon:Book}
{\scshape A.~F. Beardon} -- \emph{Iteration of rational functions}, Graduate
  Texts in Mathematics, vol. 132, Springer-Verlag, New York, 1991, Complex
  analytic dynamical systems.

\bibitem{Benedetto-Briend-Perdrix:2007}
{\scshape R.~Benedetto, J.-Y. Briend {\normalfont \smfandname} H.~Perdry} --
  {\og Dynamique des polyn\^omes quadratiques sur les corps locaux\fg},
  \emph{J. Th\'eor. Nombres Bordeaux} \textbf{19} (2007), no.~2, p.~325--336.

\bibitem{Benedetto:these}
{\scshape R.~L. Benedetto} -- {\og Fatou components in {$p$}-adic dynamics\fg},
  \emph{PhD Thesis.} (1998), p.~1--98.

\bibitem{Benedetto:2000}
\bysame , {\og {$p$}-adic dynamics and {S}ullivan's no wandering domains
  theorem\fg}, \emph{Compositio Math.} \textbf{122} (2000), no.~3, p.~281--298.

\bibitem{Benedetto:2001}
\bysame , {\og Hyperbolic maps in {$p$}-adic dynamics\fg}, \emph{Ergodic Theory
  Dynam. Systems} \textbf{21} (2001), no.~1, p.~1--11.

\bibitem{Benedetto:2002b}
\bysame , {\og Components and periodic points in non-{A}rchimedean
  dynamics\fg}, \emph{Proc. London Math. Soc. (3)} \textbf{84} (2002), no.~1,
  p.~231--256.

\bibitem{Benedetto:2002}
\bysame , {\og Examples of wandering domains in {$p$}-adic polynomial
  dynamics\fg}, \emph{C. R. Math. Acad. Sci. Paris} \textbf{335} (2002), no.~7,
  p.~615--620.

\bibitem{Benedetto:2006}
\bysame , {\og Wandering domains in non-{A}rchimedean polynomial dynamics\fg},
  \emph{Bull. London Math. Soc.} \textbf{38} (2006), no.~6, p.~937--950.

\bibitem{berkovich1990}
{\scshape V.~G. Berkovich} -- \emph{Spectral theory and analytic geometry over
  non-{A}rchimedean fields}, Mathematical Surveys and Monographs, vol.~33,
  American Mathematical Society, Providence, RI, 1990.

\bibitem{Bezivin:2004b}
{\scshape J.-P. B{\'e}zivin} -- {\og Sur la compacit\'e des ensembles de
  {J}ulia des polyn\^omes {$p$}-adiques\fg}, \emph{Math. Z.} \textbf{246}
  (2004), no.~1-2, p.~273--289.

\bibitem{Cantat:ER}
{\scshape S.~Cantat} -- {\og Quelques aspects des syst{\`e}mes dynamiques
  polynomiaux: existence, exemples, rigidit{\'e}\fg}, \emph{Ce volume} (2006).

\bibitem{Douady-Hubbard:1985}
{\scshape A.~Douady {\normalfont \smfandname} J.~H. Hubbard} -- {\og On the
  dynamics of polynomial-like mappings\fg}, \emph{Ann. Sci. \'Ecole Norm. Sup.
  (4)} \textbf{18} (1985), no.~2, p.~287--343.

\bibitem{Ducros:Bourbaki}
{\scshape A.~Ducros} -- {\og Espaces analytiques {$p$}-adiques au sens de
  {B}erkovich\fg}, \emph{Ast\'erisque} (2007), no.~311, p.~Exp. No. 958, viii,
  137--176, S\'eminaire Bourbaki. Vol. 2005/2006.

\bibitem{dwork-g-s94}
{\scshape B.~Dwork, G.~Gerotto {\normalfont \smfandname} F.~J. Sullivan} --
  \emph{An introduction to ${G}$-functions}, Annals of Math. Studies, no. 133,
  Princeton Univ. Press, 1994.

\bibitem{Favre-Rivera-Letelier:2008}
{\scshape C.~Favre {\normalfont \smfandname} J.~Rivera-Letelier} -- {\og
  Th{\'e}orie ergodique des fractions rationnelles sur un corps
  ultram\'etrique\fg}, \emph{Manuscript} (2007), p.~1--39.

\bibitem{fresnel-vanderput2004}
{\scshape J.~Fresnel {\normalfont \smfandname} M.~van~der Put} -- \emph{Rigid
  analytic geometry and its applications}, Progress in Mathematics, vol. 218,
  Birkh\"auser Boston Inc., Boston, MA, 2004.

\bibitem{Griffiths-Harris:book}
{\scshape P.~Griffiths {\normalfont \smfandname} J.~Harris} -- \emph{Principles
  of algebraic geometry}, Wiley Classics Library, John Wiley \& Sons Inc., New
  York, 1994, Reprint of the 1978 original.

\bibitem{Herman-Yoccoz:1983}
{\scshape M.~Herman {\normalfont \smfandname} J.-C. Yoccoz} -- {\og
  Generalizations of some theorems of small divisors to non-{A}rchimedean
  fields\fg}, in \emph{Geometric dynamics (Rio de Janeiro, 1981)}, Lecture
  Notes in Math., vol. 1007, Springer, Berlin, 1983, p.~408--447.

\bibitem{Hsia:London}
{\scshape L.-C. Hsia} -- {\og Closure of periodic points over a
  non-{A}rchimedean field\fg}, \emph{J. London Math. Soc. (2)} \textbf{62}
  (2000), no.~3, p.~685--700.

\bibitem{lindahl2004}
{\scshape K.-O. Lindahl} -- {\og On {S}iegel's linearization theorem for fields
  of prime characteristic\fg}, \emph{Nonlinearity} \textbf{17} (2004), no.~3,
  p.~745--763.

\bibitem{Lubin:1994}
{\scshape J.~Lubin} -- {\og Non-{A}rchimedean dynamical systems\fg},
  \emph{Compositio Math.} \textbf{94} (1994), no.~3, p.~321--346.

\bibitem{McMullen:survey}
{\scshape C.~T. McMullen} -- {\og Frontiers in complex dynamics\fg},
  \emph{Bull. Amer. Math. Soc. (N.S.)} \textbf{31} (1994), no.~2, p.~155--172.

\bibitem{Milnor:book}
{\scshape J.~Milnor} -- \emph{Dynamics in one complex variable}, third
  \smfedname, Annals of Mathematics Studies, vol. 160, Princeton University
  Press, Princeton, NJ, 2006.

\bibitem{Morton-Silverman:Crelle}
{\scshape P.~Morton {\normalfont \smfandname} J.~H. Silverman} -- {\og Periodic
  points, multiplicities, and dynamical units\fg}, \emph{J. Reine Angew. Math.}
  \textbf{461} (1995), p.~81--122.

\bibitem{Rivera-Letelier:Asterisque}
{\scshape J.~Rivera-Letelier} -- {\og Dynamique des fonctions rationnelles sur
  des corps locaux\fg}, \emph{Ast\'erisque} (2003), no.~287, p.~xv, 147--230,
  Geometric methods in dynamics. II.

\bibitem{Rivera-Letelier:compositio}
\bysame , {\og Espace hyperbolique {$p$}-adique et dynamique des fonctions
  rationnelles\fg}, \emph{Compositio Math.} \textbf{138} (2003), no.~2,
  p.~199--231.

\bibitem{Rivera-Letelier:CMH}
\bysame , {\og Points p\'eriodiques des fonctions rationnelles dans l'espace
  hyperbolique {$p$}-adique\fg}, \emph{Comment. Math. Helv.} \textbf{80}
  (2005), no.~3, p.~593--629.

\bibitem{Silverman:book}
{\scshape J.~H. Silverman} -- \emph{The arithmetic of dynamical systems},
  Graduate Texts in Mathematics, vol. 241, Springer, New York, 2007.

\bibitem{Flexor}
{\scshape J.-C. Yoccoz} -- {\og Dynamique des polyn\^omes quadratiques\fg}, in
  \emph{Dynamique et g\'eom\'etrie complexes (Lyon, 1997)}, Panor. Synth\`eses,
  vol.~8, Soc. Math. France, Paris, 1999, Notes prepared by Marguerite Flexor,
  p.~x, xii, 187--222.

\end{thebibliography}

%
% \bibliography{aclab,acl,reference}
%

\end{document}